\documentclass[11pt]{article}
\usepackage[margin=1in]{geometry}
\usepackage{amsmath, amssymb, amsthm}
\usepackage[style=trad-abbrv, doi=false, url=false, isbn=false, eprint=true, sortcites=true
]{biblatex}
\usepackage[dvipsnames]{xcolor}
\usepackage{booktabs}
\usepackage{enumitem}
\usepackage{hyperref}
\usepackage{graphicx}
\usepackage{multirow}
\newtheorem{theorem}{Theorem}[section]
\newtheorem{proposition}[theorem]{Proposition}
\newtheorem{lemma}[theorem]{Lemma}
\newtheorem{corollary}[theorem]{Corollary}
\newtheorem*{thm}{Theorem}

\newtheorem{remark}[theorem]{Remark}

\theoremstyle{definition}

\newtheorem{example}[theorem]{Example}
\newtheorem*{assumption}{Assumption}

\newcommand{\qm}[1]{\begin{bmatrix}#1\end{bmatrix}}
\newcommand{\ab}[1]{\left\langle#1\right\rangle}
\newcommand{\litab}[1]{\langle#1\rangle}

\newcommand{\rb}[1]{\left(#1\right)}

\newcommand{\cb}[1]{\left\{#1\right\}}
\newcommand{\vb}[1]{\left\vert#1\right\vert}

\newcommand{\C}{\mathbb{C}}

\newcommand{\Acal}{\mathcal{A}}
\newcommand{\Bcal}{\mathcal{B}}
\newcommand{\Xcal}{\mathcal{X}}

\newcommand{\ran}{\text{\rm range}}
\newcommand{\diag}{\mathrm{diag}}
\newcommand{\rank}{\mathrm{rank}}
\newcommand{\norm}[1]{\left\lVert #1\right\rVert}
\newcommand{\litnorm}[1]{\lVert#1\rVert}
\newcommand{\wn}{\omega_n}

\begin{document}

\title{Dynamics-Aware Sampling and Recovery for Tensor-Valued Dynamical Systems}
\author{Seok-Young Chung
\footnote{sychung@msu.edu. Department of Mathematics, Michigan State University, East Lansing, MI 48824 USA}
\and Longxiu Huang
\footnote{huangl3@msu.edu. Department of Computational Mathematics, Science and Engineering, Department of Mathematics, Michigan State University, East Lansing, MI 48824 USA}
\and Yisen Wang
\footnote{wangyise@msu.edu. Department of Computational Mathematics, Science and Engineering, Michigan State University, East Lansing, MI 48824 USA}
}
\date{}

\maketitle

\begin{abstract}
We study sampling and recovery for tensor-valued dynamical systems evolving under the t-product. Given an unknown initial tensor state, the goal is to recover it from partial spatial samples collected over multiple time instances. Although the t-product Fourier transform formally decomposes the evolution into slicewise matrix dynamics, we show that this is not the effective recovery decomposition. When eigenvalues are shared across multiple Fourier-domain frontal slices, the corresponding eigenspaces couple into aggregated spectral blocks. These blocks, rather than the individual slices, determine the tensor-specific recovery geometry and the associated space--time tradeoff.

We prove deterministic identifiability results for recovering the initial tensor state from spatio-temporal samples. In the diagonalizable case, recovery is characterized by frame conditions on the aggregated eigenspace blocks, and the analysis is extended to non-diagonalizable systems through generalized eigenspace conditions associated with Jordan structure. We also identify the finite-horizon stacked recovery system as the computational object behind Tensor Dynamical Sampling (DS) and relate it to a decoupled two-stage formulation once the temporal threshold is reached. Under independent spatial subsampling, we obtain dynamics-aware randomized recovery guarantees by analyzing empirical Gram matrices on the aggregated blocks. The resulting bounds depend on intrinsic spectral quantities, including aggregated multiplicities, incoherence, and population conditioning. We then establish stability estimates that separate the spatial frame factor from the temporal Vandermonde factor, showing that temporal oversampling can improve identifiability while worsening numerical conditioning.

Finally, we implement the resulting Tensor DS reconstruction framework and validate the theory through synthetic and real-data experiments. Phase transition experiments track the predicted blockwise sampling thresholds, while stability experiments illustrate the Vandermonde conditioning mechanism. Experiments on sea surface temperature and turbulence velocity data show that a learned t-product system operator gives more accurate and structurally faithful reconstructions than vectorized low-rank, Tucker-type, graph-regularized, and graph-based dynamical sampling baselines.
\end{abstract}

\sloppy

\section{Introduction}

Many modern sensing and monitoring problems involve recovering an evolving
physical field from incomplete measurements. Examples include sea surface
temperature, where one seeks to reconstruct a large spatio-temporal
geophysical field from partial observations, and turbulence data, where
complex multi-dimensional flow patterns evolve over time while only limited measurements are available. In such settings, acquiring dense measurements at every time instance is often prohibitively expensive or simply impossible.
This naturally raises the question of whether sparse spatial measurements can be compensated for by exploiting temporal evolution.

This question lies at the heart of \emph{dynamical sampling}, which asks when an unknown initial state can be recovered from coarse spatial measurements collected at multiple time instances
\cite{lu2009spatial,aldroubi2015exact,aldroubi2017dynamical}.
A central principle in this area is the resulting \emph{space--time tradeoff},
where additional temporal samples compensate for limited spatial sensing by leveraging the underlying dynamics. Over the past decade, dynamical sampling has developed into a rich mathematical theory involving exact reconstruction, frame-theoretic characterizations, operator-theoretic analysis, robustness to noise, and system identification
\cite{aldroubi2016krylov,aldroubi2017iterative,tang2017system,
aldroubi2018dynamical,aldroubi2019frames,christensen2019frame,cabrelli2020dynamical,christensen2020frame}.

However, many evolving datasets of practical interest are not naturally
vector-valued. Instead, each temporal state is intrinsically
\emph{tensor-valued}, carrying multiway structure associated with spatial,
spectral, channel, or other physical organization. Sea surface temperature
and turbulence data are representative examples: each snapshot is naturally represented as a multidimensional array, and flattening the state into a
vector may obscure the underlying dynamical structure. This motivates a
fundamental question: how should dynamical sampling theory be formulated when the evolving states and the governing operators preserve tensor structure?

In this work, we study this question in the t-product framework for
third-order tensors \cite{kilmer2011factorization,Kilmer2013}. Let
$\Xcal\in\C^{m\times p\times n}$ denote an unknown initial tensor state and
let $\Acal\in\C^{m\times m\times n}$ be a tensor system operator. We consider the tensor-valued dynamical system
\begin{equation}\label{evol1-intro}
    \Xcal_t=\Acal^t * \Xcal,
    \qquad t=0,1,2,\dots,
\end{equation}
where $*$ denotes the t-product and powers are taken in the induced tensor
algebra. Given a spatial sampling set
$\Omega\subseteq[m]\times[p]\times[n]$ and a time horizon $T$, we observe
\begin{equation}\label{eq:T-intro}
    \Psi_{\Xcal}
    =
    \bigl\{
    [\Xcal_t]_{i,j,k}:
    (i,j,k)\in\Omega,\;
    t=0,\ldots,T-1
    \bigr\}.
\end{equation}
Our goal is to determine when these measurements uniquely determine the
initial tensor state, how many spatial samples are required, and how the
recovery can be performed effectively.

The t-product system admits an exact block-circulant lifting to a
finite-dimensional dynamical sampling problem. This lifting connects the
tensor model with classical dynamical sampling theory, but it does not reveal the tensor-specific structure governing recovery. In particular, applying the discrete Fourier transform along the third mode converts the t-product evolution into frontal-slice matrix dynamics, but recovery does not generally
decouple across the slices. When the same eigenvalue $\lambda$ appears in
multiple Fourier-domain frontal slices, the corresponding components share the same temporal signature $\lambda^t$. Therefore, their spatial contributions cannot be separated through time evolution alone and must be recovered jointly. 
The associated eigenspaces combine into an aggregated direct-sum spectral block indexed by the shared eigenvalue $\lambda$.

This aggregated spectral geometry provides the correct recovery viewpoint for tensor dynamical sampling. 
Instead of treating each frontal slice as an independent recovery problem, the effective subproblems are organized by shared spectral groups. The dimension
\begin{equation*}
    d_\lambda=\sum_{s\in I_\lambda}m_{\lambda,s}
\end{equation*}
of an aggregated block determines the number of spatial degrees of freedom
that must be captured by the sampling pattern. This observation reveals how the tensor Fourier structure modifies the classical space--time tradeoff and forms the basis of our deterministic recovery theory and randomized sampling analysis.

Our analysis proceeds in three stages. First, we characterize recovery for an arbitrary fixed spatial sampling mask. The finite-horizon stacked observation system admits a temporal--spectral factorization that separates spatial sampling limitations from temporal spectral ambiguity. 
After the temporal components become separable, recovery reduces to frame conditions on explicit sensing vectors associated with the aggregated eigenspace blocks.

Second, we study a simple randomized strategy for selecting the spatial
sampling pattern. Each coordinate is sampled independently with probability $\alpha$, and the resulting mask is kept fixed throughout the observation horizon. 
Rather than designing sensor locations through a combinatorial optimization problem, this strategy provides a simple and implementable
sampling mechanism. We show that, when the sampling rate is sufficiently
large, this random mask satisfies the required blockwise recovery conditions with high probability. The resulting sampling rate is determined by the aggregated spectral dimensions, incoherence of the eigenspaces, and the geometry of the corresponding Gram matrices.

Third, we study the stability of the recovery process. Temporal observations separate the spectral groups through a Vandermonde system. Although sufficient
temporal samples guarantee identifiability, the resulting Vandermonde inverse may still be poorly conditioned. We separate the spatial frame contribution from the temporal conditioning contribution and analyze how unequal spectral magnitudes affect the numerical stability of temporal separation.

For computation, we implement the finite-horizon stacked reconstruction
formulation, which avoids explicit eigendecomposition and numerical grouping of nearly identical eigenvalues. For real datasets, the t-product system operator is learned from training snapshots through a slicewise truncated DMD-type least-squares procedure in the Fourier domain. 
Synthetic experiments test the predicted sampling and conditioning behavior, while experiments on sea surface temperature and turbulence velocity data demonstrate the effectiveness of exploiting t-product dynamical structure for practical spatio-temporal reconstruction.

\subsection{Related work}

\medskip\noindent
\textbf{Classical dynamical sampling.}
The space--time tradeoff underlying dynamical sampling goes back at least to \cite{lu2009spatial}, with exact reconstruction theory developed in
\cite{aldroubi2015exact,aldroubi2017dynamical}.
Subsequent works established frame-theoretic and operator-theoretic characterizations, iterative and Krylov-based methods, robustness analysis, and system identification approaches \cite{aldroubi2016krylov,aldroubi2017iterative,tang2017system,
aldroubi2018dynamical,aldroubi2019frames,cabrelli2020dynamical}.

\medskip\noindent
\textbf{Multidimensional, randomized, and graph-based sampling.}
Dynamical sampling has also been studied beyond one-dimensional settings,
including multidimensional evolution systems
\cite{aceska2015multidimensional} and graph-based models. Recent work has
investigated robust estimation from randomized space--time samples, random
sampling of graph diffusion fields, and randomized sampling strategies for
graph dynamical systems
\cite{huang2024robust,huang2025random,gong2026randomized}.
These works emphasize the importance of adapting sampling strategies to the spectral structure of the underlying dynamics.

Graph signal processing provides another related framework in which sampling and reconstruction are organized by operator-adapted spectral spaces rather than ambient coordinates alone
\cite{Ortega2018}. Shift-invariant spaces, graph bandlimited spaces, and
reproducing-kernel spaces on finite graphs provide examples where spectral
geometry determines recovery conditions
\cite{ChungSun2024,ChenChungSun2026}.
In contrast, the present work develops a tensor-valued theory in which the
relevant spectral geometry arises from aggregated Fourier-slice eigenspaces of the t-product dynamics.

\medskip\noindent
\textbf{Tensor-valued dynamical systems.}
Tensor models have become important for preserving multiway structure in
remote sensing, imaging, tensor time series, and other evolving physical
systems
\cite{KoldaBader2009,Rogers2013,
WangHongHanLiYaoGaoZhangChanussot2023,MDLRTDynamicMRI2024}.
The t-product framework is particularly attractive because it combines an
algebraic tensor structure with Fourier-domain linearization
\cite{kilmer2011factorization,Kilmer2013}.
Recent works have studied tensor dynamical systems from system-theoretic and data-driven perspectives
\cite{chen2024tensor,mao2025data,HeHuLouChen2026TDMD}.

The preliminary work \cite{WangCaiHuang2024} introduced dynamical sampling
for three-dimensional tensor signals through a stacked reconstruction
formulation. The present paper develops a substantially broader theory by
establishing deterministic identifiability conditions, extending the theory to Jordan structures, providing randomized sampling guarantees, analyzing stability and temporal conditioning, and validating the framework through expanded numerical experiments.

\subsection{Contributions}

The main contributions of this paper are as follows.

\begin{itemize}

\item 
We formulate dynamical sampling for third-order tensor signals evolving under the t-product and identify the tensor-specific spectral structure that governs recovery beyond a generic lifted representation.

\item 
We show that shared eigenvalues across Fourier-domain frontal slices generate aggregated direct-sum eigenspace blocks. These blocks determine the effective recovery subproblems, and their dimensions $d_\lambda$ govern the spatial
sampling burden.

\item 
Theorem~\ref{thm:stacked} characterizes the finite-horizon stacked recovery system through a temporal--spectral factorization. Theorem~\ref{thm:det} provides necessary and sufficient frame conditions for the aggregated eigenspace blocks after temporal separation. 
The characterization is further extended to non-diagonalizable systems through generalized eigenspaces and
Jordan structure.

\item 
We analyze a Bernoulli spatial sampling strategy in which each coordinate is selected independently and the resulting mask is fixed over time.
Theorem~\ref{thm:union} establishes high-probability recovery guarantees by analyzing empirical Gram matrices on the aggregated blocks. The resulting sampling rate depends on intrinsic spectral quantities, including aggregated block dimensions and eigenspace incoherence.

\item 
We establish stability estimates separating the spatial frame factor from the temporal Vandermonde factor. Theorem~\ref{thm:temp-conditioning} identifies how unequal spectral magnitudes can deteriorate the relative conditioning of temporal spectral separation.

\item 
We implement Tensor DS reconstruction using learned t-product system operators. Synthetic experiments demonstrate sampling and conditioning trends consistent with the theory, while experiments on sea surface temperature and turbulence velocity data demonstrate the practical effectiveness of preserving t-product dynamical structure.

\end{itemize}

\subsection{Organization.}
The remainder of the paper is organized as follows.
Section~\ref{sec:prelim} reviews the t-product framework, the Fourier-domain representation of the tensor dynamics, and the lateral-slice-wise stacked recovery system.
Section~\ref{sec:det} develops deterministic identifiability results in the diagonalizable and Jordan settings. 
In particular, it analyzes the temporal--spectral structure of the stacked recovery system and derives frame conditions on aggregated eigenspace blocks. It also presents the stacked and decoupled reconstruction viewpoints.
Section~\ref{sec:rand} studies the Bernoulli spatial sampling strategy and establishes high-probability recovery guarantees through Gram-matrix concentration.
Section~\ref{sec:stability} analyzes recovery stability and temporal conditioning associated with the Vandermonde separation step.
Section~\ref{sec:experiment} presents numerical experiments, including synthetic phase-transition and temporal-conditioning studies, followed by sea-surface-temperature and turbulence-velocity reconstruction experiments.
The paper concludes with a summary of the theoretical and computational findings, and the proofs of the main results are provided afterward. Appendix~\ref{app:jordan} extends the deterministic theory to Jordan structures, while Appendix~\ref{app:scaling} discusses basis scaling and conditioning issues.

\section{Preliminaries and notation}\label{sec:prelim}
We begin by introducing the notation and basic tensor-algebraic preliminaries used throughout the paper; for convenience, the principal symbols are summarized in Table~\ref{tab:notation}.
For vectors in complex Euclidean spaces, we use the convention $\langle x,y\rangle=x^*y=\overline{x}^\top y$.

\begin{table}[ht!]
\centering
\small
\renewcommand{\arraystretch}{1.08}
\caption{Summary of notation used throughout the paper.}
\label{tab:notation}
\begin{tabular}{p{0.2\linewidth}p{0.75\linewidth}}
\toprule
\textbf{Notation} & \textbf{Description} \\
\midrule
\midrule
\multicolumn{2}{@{}l}{\textit{General notation}}\\
\midrule
 
$m,p,n$ & Tensor sizes: $\Xcal\in\C^{m\times p\times n}$, $\Acal\in\C^{m\times m\times n}$ 
\\
$T$ & Number of time steps
\\
$\Acal$ & Tensor driven operator in t-product evolution $\Xcal_t=\Acal^t*\Xcal$ 
\\
$\Xcal,\ \Xcal_t$ & Initial tensor state and tensor state at time $t$
\\
$\Omega$ & Spatial sampling set: $\Omega\subseteq[m]\times[p]\times[n]$ 
\\
$\Omega_j$ & Sampling set of the lateral slice $j$: $\Omega_j=\{(i,k):(i,j,k)\in\Omega\}\subseteq[m]\times[n]$ 
\\
$\widehat{\cdot}$ & DFT along the third mode (e.g., $\widehat{\Acal}$, $\widehat{\Xcal}$)
\\
$\wn$ & $n$-th root of unity: $\wn=e^{2\pi i/n}$
\\
\midrule
\multicolumn{2}{l}{\textit{Core spectral notation}}\\
 \midrule
$U_s,\ D_s$ & Slice diagonalization $[\widehat{\Acal}]_{:,:,s}=U_sD_sU_s^{-1}$ 
\\
$\Lambda_s$ & Set of distinct eigenvalues of $D_s$, $s\in [n]$
\\
$\Lambda$ & Set of distinct eigenvalues across all frontal slices: $\Lambda = \bigcup_{s\in [n]} \Lambda_s$
\\
$P_{\lambda,s}$ & Orthogonal projector onto the $\lambda$-eigenspace in slice $s$ 
\\
$\mathcal I_\lambda$ & Slice index set where $\lambda$ appears: $\mathcal I_\lambda=\{s\in[n]:\lambda\in\Lambda_s\}$ 
\\
$m_{\lambda,s}$ & Geometric multiplicity of $\lambda$ in slice $s$: 
$m_{\lambda,s}=\rank(P_{\lambda,s})$
\\
$d_\lambda$ & Aggregated $\lambda$-eigenspace dimension: $d_\lambda=\sum_{s\in\mathcal I_\lambda} m_{\lambda,s}$ 
\\
$\bigoplus$ & Block concatenation / direct-sum notation indexed by $\mathcal I_\lambda$ 
\\
$U_{\lambda,s}$ & Eigenvector submatrix for the $\lambda$-eigenspace in slice $s$ 
\\
$W_{\lambda,s}$ & Left inverse of $U_{\lambda,s}$ induced by $U_s^{-1}$
\\
\midrule
\multicolumn{2}{@{}l}{\textit{Stacked and temporal-decoupling notation}}\\
\midrule
$A_1(t)$ & Block-diagonal Fourier-domain evolution matrix at time $t$, defined in \eqref{eq:prelim-A1}
\\
$A_3(j)$ & Columnwise sampling operator in the Fourier domain, defined in \eqref{eq:prelim-A3}
\\
$M_j(T)$ & Columnwise stacked recovery matrix, defined in \eqref{eq:prelim-Mj}
\\
$\mathsf{V}_T(\Lambda)$ & Vandermonde matrix associated with the distinct eigenvalues in $\Lambda$, defined in \eqref{eq:vander}
\\
$V_\lambda^{(i,k)}$ & Deterministic sensing vector, given by \eqref{sensing1}
\\
\midrule
\multicolumn{2}{@{}l}{\textit{Random sampling notation}}\\
\midrule
${\bf v}_\lambda^{(i,k)}$ & Coordinate sensing vector, given by \eqref{sensing2}
\\
$\delta_{i,j,k}$,\ $\alpha$ & i.i.d.\ Bernoulli random variables with sampling rate $\alpha$
\\
$G_\lambda^j$ & Empirical Gram matrix: 
$G_\lambda^j= n \sum_{i=1}^m \sum_{k=1}^n \delta_{i,j,k}
{\bf v}_\lambda^{(i,k)}({\bf v}_\lambda^{(i,k)})^*$
\\
$\Sigma_\lambda$ & Population Gram matrix: $\mathbb E[G_\lambda^j]=\alpha\Sigma_\lambda$ 
\\
$\mu_{\lambda,s}$ & Incoherence parameter: 
$\mu_{\lambda,s} = \frac{m}{m_{\lambda,s}}  \max_{i\in [m]} \litnorm{ U_{\lambda,s}^*e_i }_2^2$
\\
$\mu_\lambda$ & Maximal incoherence over $\mathcal I_\lambda$: $\mu_\lambda=\max_{s\in\mathcal I_\lambda}\mu_{\lambda,s}$
\\
$\gamma_{\lambda,\min}$ & The smallest eigenvalue of $\Sigma_\lambda$: $\gamma_{\lambda,\min}=\lambda_{\min}(\Sigma_\lambda)$ 
\\
\bottomrule
\end{tabular}
\end{table}

\subsection{t-product and Fourier-domain dynamics}

For completeness, we briefly recall the t-product framework for third-order tensors (see \cite{kilmer2011factorization} for more details). 
Let $\Acal\in \C^{m\times r\times n}$ and $\Bcal\in \C^{r\times p\times n}$.
For each $k\in [n]:=\{1,2,\cdots,n\}$, write
$\Acal^{(k)}:=[\Acal]_{:,:,k}$ and $\Bcal^{(k)}:=[\Bcal]_{:,:,k}$
for their $k$-th frontal slices.
We define
\begin{equation*}
\operatorname{unfold}(\Bcal):=
\begin{bmatrix}
\Bcal^{(1)}\\
\Bcal^{(2)}\\
\vdots\\
\Bcal^{(n)}
\end{bmatrix},
\qquad
\operatorname{fold}\rb{\operatorname{unfold}(\Bcal)}:=\Bcal,
\end{equation*}
and
\begin{equation*}
\operatorname{bcirc}(\Acal):=
\begin{bmatrix}
\Acal^{(1)} & \Acal^{(n)} & \cdots & \Acal^{(2)}\\
\Acal^{(2)} & \Acal^{(1)} & \cdots & \Acal^{(3)}\\
\vdots & \vdots & \ddots & \vdots\\
\Acal^{(n)} & \Acal^{(n-1)} & \cdots & \Acal^{(1)}
\end{bmatrix}.
\end{equation*}
The t-product of $\Acal$ and $\Bcal$ is then defined by
\begin{equation*}
\Acal * \Bcal :=
\operatorname{fold}\rb{\operatorname{bcirc}(\Acal)\operatorname{unfold}(\Bcal)}
\in \C^{m\times p\times n}.
\end{equation*}

In this paper, we are mainly interested in the case
$\Acal\in \C^{m\times m\times n}$ and $\Xcal\in \C^{m\times p\times n}$, for which $\Acal * \Xcal \in \C^{m\times p\times n}$. 
Let $\widehat{\Acal}$ denote the discrete Fourier transform (DFT) of $\Acal$ along the third mode. Throughout the paper, we use the standard (unnormalized) DFT convention, consistent with MATLAB's \texttt{fft}. 
Let 
$\wn=e^{2\pi i/n}$.
Thus, for any tensor $\mathcal Z\in\C^{m\times p\times n}$,
\begin{equation*}
[\widehat{\mathcal Z}]_{:,:,s} =
\sum_{k=1}^n \wn^{-(s-1)(k-1)}[\mathcal Z]_{:,:,k}, \qquad s\in[n],
\end{equation*}
and
\begin{equation*}
[\mathcal Z]_{:,:,k} =  \frac{1}{n}\sum_{s=1}^n \wn^{(s-1)(k-1)}
[\widehat{\mathcal Z}]_{:,:,s},   \qquad k\in[n].
\end{equation*}
In computations,
\begin{equation*}
\widehat{\Acal}=\operatorname{fft}(\Acal,[\,],3)\in\C^{m\times m\times n}.
\end{equation*}
We write its frontal slices as
$[\widehat{\Acal}]_{:,:,s}\in \C^{m\times m}$, $s\in [n]$.
Under this transform, the t-product reduces to slicewise matrix multiplication, i.e.,
\begin{equation*}
[\widehat{\Acal * \Xcal}]_{:,:,s} =
[\widehat{\Acal}]_{:,:,s}[\widehat{\Xcal}]_{:,:,s},
\qquad s\in [n].
\end{equation*}
Consequently, the time evolution \eqref{evol1-intro} satisfies
\begin{equation*}
[\widehat{\Xcal}_t]_{:,:,s} =
[\widehat{\Acal}]_{:,:,s}^{t} [\widehat{\Xcal}]_{:,:,s},
\qquad s\in [n].
\end{equation*}

\subsection{Spectral structure across frontal slices in Fourier domain}\label{prelim:2.2}

In the main development, we restrict attention to the diagonalizable case, namely, we assume that each frontal slice $[\widehat{\Acal}]_{:,:,s}$ is diagonalizable.
Thus, for every $s\in[n]$, there exists an invertible matrix $U_s$ and a diagonal matrix $D_s$ such that
\begin{equation}\label{eq:diag}
[\widehat{\Acal}]_{:,:,s} = U_s D_s U_s^{-1},\qquad s\in[n].
\end{equation}
For each slice $s$, we denote by $\Lambda_s$ the set of distinct eigenvalues of $D_s$. We then define
\begin{equation}\label{lambda1}
    \Lambda := \bigcup_{s\in [n]} \Lambda_s,
\end{equation}
that is, $\Lambda$ is the set of all distinct eigenvalues appearing across all frontal slices, where each element is pairwise distinct by construction. 

For each $s\in[n]$ and each $\lambda\in\Lambda_s$, let $P_{\lambda,s}$ denote the orthogonal projection onto the eigenspace $E_{\lambda,s}$ associated with $\lambda$ of $D_s$, that is, $E_{\lambda,s}=\ker(D_s-\lambda I)$. Then $D_s$ admits the spectral decomposition
\begin{equation}\label{eq:spec1}
    D_s = \sum_{\lambda \in \Lambda_s} \lambda P_{\lambda,s},\quad s\in [n].
\end{equation}

\begin{remark}
If $\Acal \in \mathbb{R}^{m\times m\times n}$ is real-valued, then its Fourier-domain frontal slices satisfy the conjugate symmetry
\begin{equation*}
[\widehat{\Acal}]_{:,:,s}
=
\overline{[\widehat{\Acal}]}_{:,:,\,n-s+2},
\qquad
s=2,\ldots,\Big\lfloor \frac{n}{2}\Big\rfloor+1.
\end{equation*}
This follows from the conjugate symmetry of the discrete Fourier transform along the third mode; see, e.g., \cite{Kilmer2013}. Consequently, the eigenvalues of paired frontal slices occur in complex-conjugate pairs. In particular, any real eigenvalue must appear in both slices $s$ and $n-s+2$, which may cause the index set $\mathcal{I}_\lambda$ defined below to contain multiple slices.
\end{remark}

Motivated by this observation, for each eigenvalue $\lambda\in\Lambda$ we define the index set
\begin{equation}\label{C1}
    \mathcal{I}_\lambda := \cb{ s\in[n] : \lambda\in\Lambda_s }.
\end{equation}
Thus $\mathcal{I}_\lambda$ specifies precisely which frontal slices contain the given eigenvalue $\lambda$. 

In particular, for a fixed eigenvalue $\lambda \in \Lambda$, the associated eigenspaces across all frontal slices sharing $\lambda$ form the direct sum space
\begin{equation*}
    \bigoplus_{s\in \mathcal{I}_\lambda} \textrm{ran} \rb{P_{\lambda,s}},
\end{equation*}
whose dimension is
\begin{equation}\label{dim1}
    d_\lambda := \sum_{s\in \mathcal{I}_\lambda} m_{\lambda,s}
\end{equation}
where $m_{\lambda,s} := \rank\rb{ P_{\lambda,s}}$ denotes the geometric multiplicity of $\lambda$ within the associated slice $s\in\mathcal{I}_\lambda$.

Throughout the paper, we identify this direct sum with its canonical block representation. Accordingly, the notation $\bigoplus$ may be interpreted concretely as block concatenation indexed by $\mathcal{I}_\lambda$.
For each fixed $\lambda$, we choose and fix an arbitrary ordering of the indices in $\mathcal{I}_\lambda$, and all concatenations 
are taken consistently with respect to this ordering.  We do not distinguish between the external direct sum and its canonical block vector representation in $\mathbb{C}^{d_\lambda}$.

For each $s\in [n]$ and each $\lambda\in \Lambda_s$, let
\begin{equation*}
U_{\lambda,s}\in \C^{m\times m_{\lambda,s}}
\end{equation*}
denote the submatrix of $U_s$ whose columns correspond to the eigenvalue $\lambda$, and let
\begin{equation*}
W_{\lambda,s}\in \C^{m_{\lambda,s}\times m}
\end{equation*}
denote the corresponding row submatrix of $U_s^{-1}$.
Equivalently,
\begin{equation*}
U_{\lambda,s}W_{\lambda,s} = U_sP_{\lambda,s}U_s^{-1}, \quad W_{\lambda,s}U_{\lambda,s}=I_{m_{\lambda,s}},
\end{equation*}
where $I_{m_{\lambda,s}}$ denotes the identity matrix on $\C^{m_{\lambda,s}}$.

The importance of this grouping is that recovery is not governed by the individual frontal slices separately, but by the aggregated spectral blocks associated with common eigenvalues.  Indeed, once the temporal samples are decoupled through the distinct powers $\lambda^t$, the remaining spatial information for a fixed eigenvalue $\lambda$ lives naturally on the direct sum space
\begin{equation*}
\bigoplus_{s\in \mathcal{I}_\lambda} \ran\rb{P_{\lambda,s}} .
\end{equation*}
Thus the relevant independent recovery subproblems are indexed by $\lambda\in\Lambda$, not by the frontal slices $s\in[n]$.  This viewpoint will underlie both the deterministic injectivity characterization in Section~\ref{sec:det} and the randomized Gram-matrix analysis in Section~\ref{sec:rand}.

\subsection{The lateral-slice-wise stacked recovery system}

To connect the analysis below with the reconstruction method in \cite{WangCaiHuang2024}, we recall its lateral-slice-wise stacked formulation in the Fourier domain.

Let $P_\Omega \in \cb{0,1}^{m\times p\times n}$ be the sampling mask, that is,
\begin{equation*}
[P_\Omega]_{i,j,k} =
\begin{cases}
1, & (i,j,k)\in\Omega,\\
0, & (i,j,k)\notin\Omega.
\end{cases}
\end{equation*}
Applying the DFT along the third mode, we represent the recovery problem for fixed lateral slice $j$ in terms of the Fourier coefficients of the unknown tensor. Fix $j\in[p]$, and define
\begin{equation}\label{eq:prelim-xj}
x(j) := \qm{
[\widehat{\Xcal}]_{:,j,1}\\
[\widehat{\Xcal}]_{:,j,2}\\
\vdots\\
[\widehat{\Xcal}]_{:,j,n}
} \in \C^{mn}.
\end{equation}
For each $t = 0,\ldots,T-1$, define the sampled Fourier-domain data vector
\begin{equation}\label{eq:prelim-bjt}
b(j,t) := \qm{
[\widehat{P_\Omega(\Xcal_t)}]_{:,j,1}\\
[\widehat{P_\Omega(\Xcal_t)}]_{:,j,2}\\
\vdots\\
[\widehat{P_\Omega(\Xcal_t)}]_{:,j,n}
} \in \C^{mn}.
\end{equation}

The temporal evolution is represented by the block diagonal matrix
\begin{equation}\label{eq:prelim-A1}
A_1(t) := \textrm{blkdiag}\rb{
[\widehat{\Acal}]_{:,:,1}^{\,t},
[\widehat{\Acal}]_{:,:,2}^{\,t},
\ldots,
[\widehat{\Acal}]_{:,:,n}^{\,t}
} \in \C^{mn\times mn},
\end{equation}
where $\textrm{blkdiag}(\cdot)$ denotes the block diagonal concatenation of its matrix arguments.
This matrix collects the Fourier-domain evolution of all frontal slices at time $t$.

To encode the sampling geometry for the $j$th lateral slice, define $A_2(j)\in \C^{n\times n\times m}$ by
\begin{equation}\label{eq:prelim-A2}
[A_2(j)]_{:,:,\,\ell} := \frac{1}{n} \operatorname{circ}\rb{[\widehat{P_\Omega}]_{\ell,j,:}},
\qquad \ell\in[m],
\end{equation}
and let
\begin{equation}\label{eq:prelim-A3}
A_3(j) :=
\begin{bmatrix}
\diag([A_2(j)]_{1,1,:}) & \diag([A_2(j)]_{1,2,:}) & \cdots & \diag([A_2(j)]_{1,n,:})\\
\diag([A_2(j)]_{2,1,:}) & \diag([A_2(j)]_{2,2,:}) & \cdots & \diag([A_2(j)]_{2,n,:})\\
\vdots & \vdots & \ddots & \vdots\\
\diag([A_2(j)]_{n,1,:}) & \diag([A_2(j)]_{n,2,:}) & \cdots & \diag([A_2(j)]_{n,n,:})
\end{bmatrix}.
\end{equation}
Thus $A_3(j)\in \C^{mn\times mn}$ records how the sampling mask acts on the Fourier coefficients for fixed $j$.

With these definitions, recovery for this fixed $j$ reduces to the linear systems
\begin{equation}\label{eq:prelim-local-system}
A_3(j) A_1(t)\, x(j) = b(j,t),
\qquad t = 0,\ldots,T-1.
\end{equation}
Stacking over time, we obtain the lateral-slice-wise system
\begin{equation}\label{eq:prelim-stacked}
M_j(T)\, x(j) = \mathbf b_j(T),
\end{equation}
where
\begin{equation}\label{eq:prelim-Mj}
M_j(T) := \begin{bmatrix}
A_3(j)A_1(0)\\
A_3(j)A_1(1)\\
\vdots\\
A_3(j)A_1(T-1)
\end{bmatrix},
\qquad
{\bf b}_j(T) := \begin{bmatrix}
b(j,0)\\
b(j,1)\\
\vdots\\
b(j,T-1)
\end{bmatrix}.
\end{equation}

Hence the recovery problem decomposes into $p$ independent stacked systems, one for each lateral slice $j$. This is precisely the computational starting point of the method in \cite{WangCaiHuang2024}. In the next section, we show that, under additional spectral structure, the stacked system \eqref{eq:prelim-stacked} already contains a hidden temporal--spectral decomposition, which leads to the refined identifiability and stability analysis developed later in the paper.

\section{Deterministic identifiability: necessary and sufficient conditions}\label{sec:det}

\subsection{Lifted-matrix viewpoints: observability and classical dynamical sampling}\label{spatial}

\

\begin{figure}[ht!]
    \centering
    \includegraphics[width=1\linewidth]{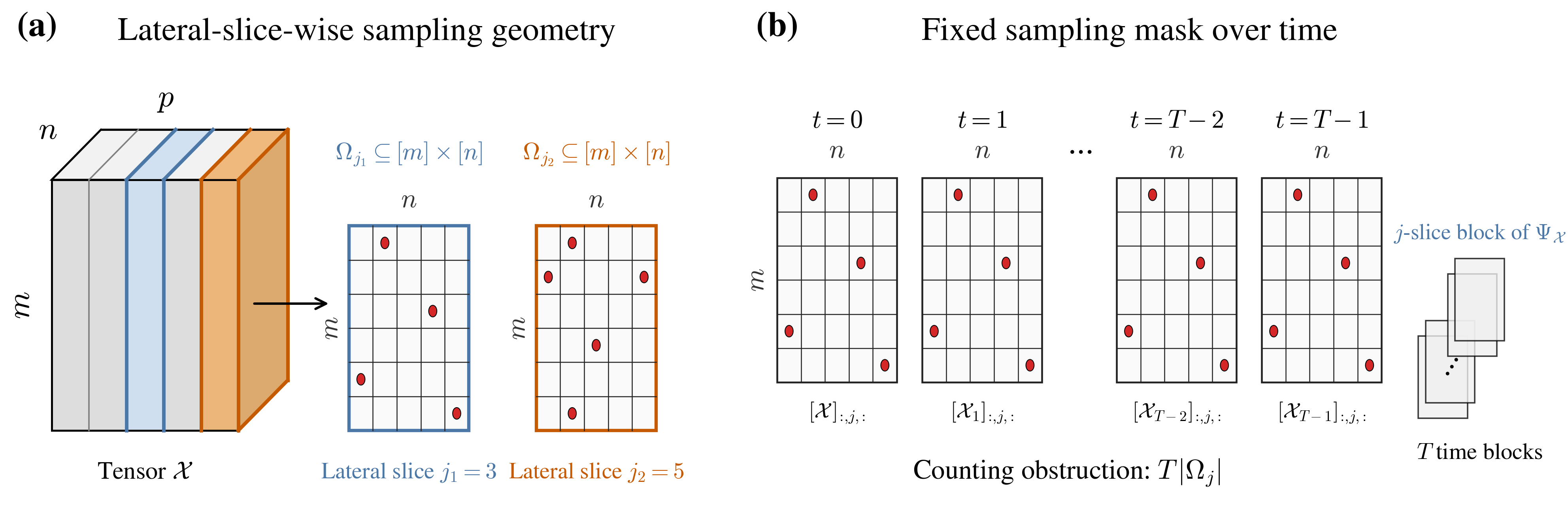}
    \caption{Time-independent lateral-slice sampling convention. (a) For each lateral slice $j$, the set $\Omega_j\subseteq[m]\times[n]$ specifies which entries of a lateral slice are observed, indexed by the mode-$1$ and mode-$3$ coordinates $(i,k)$. The mask may vary with $j$. (b) For fixed $j$, the samples at every time $t=0,\ldots,T-1$ are taken from $[\Xcal_t]_{:,j,:}$ on the same coordinate set $\Omega_j$.}
\end{figure}

We begin by interpreting the recovery problem in the spatial domain. 
Under the t-product evolution \eqref{evol1-intro}, the tensor dynamics can be represented as a linear dynamical system through vectorization. 
Let $\mathbf{M} \in \mathbb{C}^{mn \times mn}$ be the block circulant matrix associated with the t-product action $\mathcal{A} \,*\,$, defined as
\begin{equation*}
\mathbf{M} := \textrm{bcirc}(\mathcal{A}) = 
\begin{bmatrix} 
[\mathcal{A}]_{:,:,1} & [\mathcal{A}]_{:,:,n} & \dots & [\mathcal{A}]_{:,:,2} \\
[\mathcal{A}]_{:,:,2} & [\mathcal{A}]_{:,:,1} & \dots & [\mathcal{A}]_{:,:,3} \\
\vdots & \vdots & \ddots & \vdots \\
[\mathcal{A}]_{:,:,n} & [\mathcal{A}]_{:,:,n-1} & \dots & [\mathcal{A}]_{:,:,1}.
\end{bmatrix}
\end{equation*}
Then for each lateral slice $j\in[p]$, the evolution \eqref{evol1-intro} takes the form
\begin{equation*}
    \mathrm{vec}([\mathcal{X}_t]_{:,j,:})=\mathbf{M}^t \mathbf{u}_j,
\quad \mathbf{u}_j:=\mathrm{vec}([\mathcal{X}]_{:,j,:}),  
\end{equation*}
where $\mathrm{vec}(\cdot)$ denotes the standard column-stacking vectorization.

Thus, for each fixed $j\in[p]$, the tensor recovery problem can be embedded into existing finite-dimensional matrix theories. We first record the corresponding lifted-matrix observability viewpoint. At the same time, the block circulant structure of $\mathbf{M}$ already reflects a basic difference from the classical setting: dynamics are coupled across the tensor modes, so the evolution of each entry of $\mathcal{X}$ depends on all frontal slices of $\widehat{\Acal}$.

For each $j \in [p]$, define
\begin{equation}\label{index1-intro}
    \Omega_j:=\{(i,k):(i,j,k)\in\Omega\}\subseteq[m]\times[n],
\end{equation}
so that $\Omega=\bigcup_{j=1}^p \Omega_j$.
Let $\mathbf{S}_{\Omega_j} \in \{0,1\}^{|\Omega_j| \times mn}$ be the sampling matrix that extracts the coordinates of $\mathbf{u}_j\in \C^{mn}$ indexed by $(i,k)\in\Omega_j$.
The measurements collected at time $t$ are therefore
\begin{equation}\label{matrixsys}
    \mathbf{y}_{j,t}= \mathbf{S}_{\Omega_j}\mathbf{M}^t \mathbf{u}_j.
\end{equation}

Under this formulation, the recovery reduces to an observability problem for a linear dynamical system.
The following statement is a direct consequence of the classical observability criterion for linear systems (see, e.g., \cite[Chap. 6]{Chen1999}) and is recorded here to clarify the connection with our tensor sampling model.

\begin{corollary}\label{thm:spatial}
The tensor signal $\mathcal{X}$ is uniquely recoverable from the spatio-temporal samples $\Psi_\mathcal{X}$ in \eqref{eq:T-intro}  if and only if for every $j \in [p]$, the spatio-temporal observability matrix $\mathbf{O}_j$ satisfies
\begin{equation} \label{eq:observ}
    \text{rank}(\mathbf{O}_j) = mn,
\end{equation}
where $\mathbf{O}_j \in \mathbb{C}^{|\Omega_j|T \times mn}$ is defined as
\begin{equation*}
    \mathbf{O}_j := \begin{bmatrix} 
    \mathbf{S}_{\Omega_j} \mathbf{M}^0\\ 
    \mathbf{S}_{\Omega_j} \mathbf{M}^1 \\ 
    \vdots\\ 
    \mathbf{S}_{\Omega_j} \mathbf{M}^{T-1} 
    \end{bmatrix}.
\end{equation*}
In particular,
a necessary condition for \eqref{eq:observ} is
\begin{equation*}
    |\Omega_j|T \ge mn.
\end{equation*}
\end{corollary}

The same lifted system may also be viewed through the classical finite-dimensional dynamical sampling framework. Indeed, if $\mathbf e_{i,k}\in\C^{mn}$ denotes the standard coordinate vector associated with $(i,k)\in[m]\times[n]$ under the vectorization of $[\Xcal]_{:,j,:}$, then
\begin{equation*}
    [\Xcal_t]_{i,j,k} = \ab{\mathbf e_{i,k},\mathbf M^t \mathbf u_j}
    = \ab{(\mathbf M^*)^t \mathbf e_{i,k},\mathbf u_j}, \qquad (i,k)\in\Omega_j.
\end{equation*}
Thus, for each fixed $j$, the tensor recovery problem may be viewed as finite dynamical sampling generated by the family
\begin{equation*}
    \cb{(\mathbf M^*)^t \mathbf e_{i,k} : (i,k)\in\Omega_j,\ 0\le t\le T-1 }.
\end{equation*}

In the diagonalizable setting of \eqref{eq:diag}, the lifted matrix $\mathbf M$ and hence $\mathbf M^*$ are diagonalizable. Therefore the classical finite-dimensional dynamical sampling theorem \cite[Theorem 2.3]{aldroubi2017dynamical} applies to the lifted system \eqref{matrixsys}.
\begin{corollary}
Assume that each Fourier-domain frontal slice is diagonalizable as in \eqref{eq:diag}, and let $Q_\lambda$ denote the spectral projection of $\mathbf M^*$ onto the eigenspace associated with $\overline{\lambda}$, $\lambda\in\Lambda$. If $T \ge |\Lambda|$, then the tensor signal $\Xcal$ is uniquely recoverable from the spatio-temporal samples $\Psi_\Xcal$ if and only if, for every $j\in[p]$ and every $\lambda\in\Lambda$, the family
\begin{equation*}
    \cb{Q_\lambda \mathbf e_{i,k} : (i,k)\in\Omega_j}
\end{equation*}
forms a frame for $\ran\rb{Q_\lambda} = Q_\lambda \rb{ \C^{mn} }$. In particular, a necessary condition is
\begin{equation*}
    |\Omega_j| \ge d_\lambda, \qquad \lambda\in\Lambda.
\end{equation*}
\end{corollary}

The two lifted viewpoints above are useful because they place the fixed-$j$ tensor recovery problem within standard finite-dimensional theories. The observability criterion gives a global rank test for the matrix $\mathbf O_j$, while the classical dynamical sampling formulation gives a global projected-frame condition for the eigenspaces of $\mathbf M^*$. Both criteria are correct, but they remain lifted-matrix conditions. They do not explain how the eigenspaces of $\mathbf M=\textrm{bcirc}\rb{\Acal}$ are assembled from the Fourier-domain frontal slices of $\widehat\Acal$, nor how the tensor sampling coordinates $(i,k)$ act on those slice components. The decoupled formulation below resolves this global condition in tensor Fourier coordinates. It identifies the aggregated eigenspace blocks associated with shared eigenvalues and derives the corresponding slice-aware sensing vectors. This is the structure needed later to separate temporal Vandermonde conditioning from spatial sampling geometry and to derive the blockwise randomized bounds.

\subsection{The temporal--spectral structure of the stacked system}\label{stacked}

In this subsection, we work in the diagonalizable setting of Section~\ref{sec:prelim}. 
Thus we retain the notation $U_s$, $D_s$, $\Lambda$,  $m_{\lambda,s}$, and $d_\lambda$ introduced in subsection~\ref{prelim:2.2}. The non-diagonalizable case requires additional Jordan-block notation and is discussed separately in Appendix~\ref{app:jordan}.

We first examine the stacked Fourier-domain system before imposing the temporal threshold $T\ge|\Lambda|$.
The purpose is to reorganize the columns of $M_j(T)$ according to the distinct eigenvalues in $\Lambda$ and to identify the spatial blocks that correspond to each aggregated eigenspace.

For each $t=0,\ldots,T-1$, define block diagonal matrices
\begin{equation}\label{eqn:WDeltat}
W := \textrm{blkdiag}\rb{U_1,\ldots,U_n}, \qquad
\Delta_t := \textrm{blkdiag}\rb{D_1^t,\ldots,D_n^t}.
\end{equation}
Then the block diagonal matrix $A_1(t)$ defined in \eqref{eq:prelim-A1} satisfies
\begin{equation}\label{eq:decomp}
    A_1(t) = W \Delta_t W^{-1}.
\end{equation}
Enumerate the distinct eigenvalues in $\Lambda$ as $\Lambda = \cb{\lambda_1,\ldots,\lambda_R}$,
and choose a permutation matrix $P$ so that
\begin{equation}\label{eq:reordering1}
    P^{-1}\Delta_tP = \textrm{blkdiag}\rb{\lambda_1^t I_{d_{\lambda_1}},\ldots,\lambda_R^t I_{d_{\lambda_R}}}.
\end{equation}
Set $Q := WP$.
We partition the matrix $A_3(j)Q$ according to the eigenvalue ordering in \eqref{eq:reordering1}
\begin{equation}\label{eq:Bblock}
    A_3(j)Q = \qm{ B_{j,\lambda_1} & \cdots & B_{j,\lambda_R} }, \qquad
B_{j,\lambda}\in\C^{mn\times d_\lambda}.
\end{equation}
Thus $B_{j,\lambda}$ is the $d_\lambda$-column block of $A_3(j)Q$ corresponding to the eigenvalue $\lambda$ in the reordered coordinates \eqref{eq:reordering1}.

We also define the Vandermonde matrix associated with the distinct eigenvalues in $\Lambda$ by
\begin{equation}\label{eq:vander}
\mathsf V_T(\Lambda):=
\qm{ 1 & \cdots & 1\\
\lambda_1 & \cdots & \lambda_R\\
\vdots & & \vdots\\
\lambda_1^{T-1} & \cdots & \lambda_R^{T-1} }\in\C^{T\times R}.
\end{equation}
We use $\otimes$ to denote the Kronecker product.

\begin{theorem}\label{thm:stacked}
For fixed $j\in[p]$, the following are equivalent:
\begin{enumerate}[label={\rm(\roman*)}]
    \item $M_j(T)$ has full column rank.
    \item Each $B_{j,\lambda}$ has full column rank, and
    \begin{equation*}
        \ker\rb{\mathsf{V}_T(\Lambda)\otimes I_{mn}} \cap \ran\rb{ \text{\rm blkdiag}\rb{B_{j,\lambda_1},\ldots,B_{j,\lambda_R}}} = \cb{0}.
    \end{equation*}
\end{enumerate}
\end{theorem}

Theorem~\ref{thm:stacked} separates the stacked recovery criterion into a spatial obstruction and a temporal obstruction. Thus the stacked system is not merely a black-box least-squares matrix. The full-column-rank condition on $B_{j,\lambda}$ is a spatial requirement on the aggregated $\lambda$-block, while the intersection condition describes the ambiguity left by the Vandermonde factor. This gives a structural refinement of the lateral-slice-wise stacked formulation in \cite{WangCaiHuang2024}. Rank failure can arise from an insufficient aggregated spatial block, from unresolved temporal ambiguity, or from the interaction of the two when the time window is short. More precisely, if $T<|\Lambda|$, then $\mathsf V_T(\Lambda)$ has a nontrivial kernel, and recovery depends on whether the range of the block spatial map intersects this temporal ambiguity. If $T\ge|\Lambda|$, then $\mathsf V_T(\Lambda)$ has full column rank, so the intersection condition becomes trivial and the criterion reduces to the full column rank of the aggregated spatial blocks $B_{j,\lambda}$.

Consequently, the stacked viewpoint contains both the finite-window obstruction relevant for short time horizons and the decoupled spatial obstruction that remains after temporal separation. The next subsection resolves this latter condition into frame conditions on explicit sensing vectors. The same stacked viewpoint also gives the following necessary counting conditions. Although $M_j(T)$ is written as a $Tmn\times mn$ matrix, each time block uses the same spatial sampling set $\Omega_j$.

Indeed, let $\mathcal{F}_n:=\rb{\wn^{-(s-1)(k-1)}}_{s,k=1}^n$ denote the DFT matrix. Then $\mathcal{F}_n^{-1}=n^{-1}\mathcal{F}_n^*$, and the Fourier-domain sampling matrix satisfies
\begin{equation*}
A_3(j)=(\mathcal{F}_n\otimes I_m)\mathbf S_{\Omega_j}^*\mathbf S_{\Omega_j}(\mathcal{F}_n^{-1}\otimes I_m),
\end{equation*}
up to the fixed ordering convention used in \eqref{eq:prelim-A3}. Hence
\begin{equation*}
\rank A_3(j)=\rank\rb{\mathbf S_{\Omega_j}^*\mathbf S_{\Omega_j}}=|\Omega_j|.
\end{equation*}
Using the rank subadditivity for vertically stacked matrices and the inequality $\rank(AB)\le \rank A$, we obtain
\begin{equation*}
\rank M_j(T)
\le \sum_{t=0}^{T-1}\rank\rb{A_3(j)A_1(t)}
\le T|\Omega_j|.
\end{equation*}
Therefore the full-column-rank condition $\rank M_j(T)=mn$ implies the necessary condition
\begin{equation*}
T|\Omega_j|\ge mn.
\end{equation*}
Moreover, since $B_{j,\lambda}$ is a column block of $A_3(j)Q$, the requirement that $B_{j,\lambda}$ have full column rank also implies
\begin{equation*}
|\Omega_j| \ge d_\lambda ,\qquad \lambda\in\Lambda.
\end{equation*}
Thus the stacked viewpoint already contains both the global counting obstruction and the blockwise spatial obstruction, while the decoupled formulation below identifies the latter as a frame condition on the aggregated eigenspace blocks.

\begin{corollary}\label{coro:conjecture}
If the tensor signal $\Xcal\in\mathbb{C}^{m\times p\times n}$ is uniquely recoverable from the spatio-temporal samples $\Psi_{\Xcal}$ (see~\eqref{eq:T-intro}), then
\begin{equation*}
\bigcup_{(i,j,k)\in\Omega}\cb{j} = [p].
\end{equation*}
Equivalently, $\Omega_j\ne\varnothing$ for every $j\in[p]$.
\end{corollary}

In particular, Corollary~\ref{coro:conjecture} proves Conjecture~1 of \cite{WangCaiHuang2024}. In the special lattice case $\Omega=I\times J\times[n]$, it follows that $J=[p]$, which implies
\cite[Theorem 3.1]{WangCaiHuang2024}.

\subsection{Temporal decoupling and frame conditions on aggregated eigenspaces}\label{decoupled}

We now separate the temporal part of the recovery problem.
For each sampled spatial coordinate, the time samples form an exponential system indexed by the distinct eigenvalues in $\Lambda$. After this temporal system is inverted, the remaining recovery problem is spatial and takes place on the aggregated eigenspaces associated with those eigenvalues.

\begin{proposition}\label{prop:temp-regroup}
Assume that each frontal slice $[\widehat{\Acal}]_{:,:,s}$, $s\in[n]$, is diagonalizable as in \eqref{eq:diag}. Fix $j\in[p]$ and $(i,k)\in\Omega_j$, and define
\begin{equation*}
y_t(i,k):=[\Acal^t * \Xcal]_{i,j,k},
\qquad t=0,1,\ldots,T-1.
\end{equation*}
Then
\begin{equation}\label{sys1}
y_t(i,k)=\sum_{\lambda\in\Lambda}\lambda^t\, g_\lambda(i,k),
\end{equation}
where
\begin{equation}\label{eq:g-lambda-coord}
g_{\lambda}(i,k)=
\sum_{s\in \mathcal{I}_{\lambda}}
\frac{1}{n}\,\wn^{(s-1)(k-1)}
\ab{
U_{\lambda,s}^* e_i,\;
W_{\lambda,s}[\widehat{\Xcal}]_{:,j,s}
}.
\end{equation}
\end{proposition}

Proposition~\ref{prop:temp-regroup} shows how the slicewise Fourier dynamics merge into an aggregated spectral system.
For a fixed eigenvalue $\lambda$, the coefficient $g_\lambda(i,k)$ collects the slice contributions with indices in $\mathcal I_\lambda$.
Thus temporal separation occurs at the level of distinct eigenvalues, while the spatial component for each separated coefficient lives on the corresponding aggregated eigenspace.

Write
\begin{equation*}
\Lambda=\cb{\lambda_1,\ldots,\lambda_R}, \qquad R:=|\Lambda|,
\end{equation*}
and consider the Vandermonde matrix $\mathsf{V}_T(\Lambda)$ defined in \eqref{eq:vander}.
For each sampled sensor $(i,k)$, let
\begin{equation}\label{yg-system}
{\bf y}^{(i,k)}:=
\qm{ y_0(i,k)\\
y_1(i,k)\\
\vdots\\
y_{T-1}(i,k) }\in \C^{T} ,
\qquad
{\bf g}^{(i,k)}:=
\qm{ g_{\lambda_1}(i,k)\\
g_{\lambda_2}(i,k)\\
\vdots \\
g_{\lambda_R}(i,k) } \in \C^{R}.
\end{equation}
Then \eqref{sys1} is equivalent to
\begin{equation}\label{eq:temp-vand-threshold}
{\bf y}^{(i,k)}=\mathsf{V}_T(\Lambda)\,{\bf g}^{(i,k)}.
\end{equation}

Since the eigenvalues in $\Lambda$ are distinct, the matrix $\mathsf{V}_T(\Lambda)$ has full column rank if and only if $T\ge R$.
Therefore, once $T\ge |\Lambda|$, the temporal samples uniquely determine the coefficients $\cb{g_\lambda(i,k)}_{\lambda\in\Lambda}$ for each sampled sensor $(i,k)$.

Accordingly, the recovery problem reduces to a spatial spanning problem on each aggregated eigenspace block
\begin{equation*}
\bigoplus_{s\in\mathcal I_\lambda}\ran\rb{P_{\lambda,s}},
\qquad \lambda\in\Lambda.
\end{equation*}
Since all spaces in this paper are finite-dimensional, we use the term frame for a finite family that spans the ambient space, or equivalently has a positive lower frame bound.
This leads to the frame characterization stated in Theorem~\ref{thm:det}.

\begin{theorem}\label{thm:det}
Let $\Acal\in \C^{m\times m\times n}$ such that each frontal slice of its Fourier transform, $[\widehat{\Acal}]_{:,:,s}$, for $s\in [n]$, is diagonalizable as in \eqref{eq:diag},
and let $T \ge |\Lambda|$.
The signal $\Xcal$ is uniquely recovered from $\Psi_\Xcal$ if and only if, for each $j\in [p]$ and each $\lambda\in \Lambda$, the collection of vectors $\big\{ V_\lambda^{(i,k)} : (i,k)\in \Omega_j \big\}$ forms a frame for $\bigoplus_{s\in \mathcal{I}_\lambda}\ran\rb{P_{\lambda,s}}$, where
\begin{equation}\label{sensing1}
    V^{(i,k)}_\lambda =\bigoplus_{s\in \mathcal{I}_\lambda} \omega_n^{-(s-1)(k-1)} P_{\lambda,s} U_{s}^* e_i.
\end{equation}

\end{theorem}

Theorem~\ref{thm:det} gives the decoupled recovery criterion.
The temporal requirement is exactly $T\ge|\Lambda|$, the number of distinct eigenvalues across all frontal slices. After this threshold is reached, recoverability is determined by spatial frame conditions on the aggregated eigenspace blocks.
In particular, shared eigenvalues increase the spatial dimension $d_\lambda$ of the corresponding block, and therefore increase the sampling burden for that eigenvalue.

\begin{remark}[Reduction to the classical dynamical sampling model]
Theorem~\ref{thm:det} subsumes the classical dynamical sampling model within the proposed tensor framework.
Indeed, setting the third-mode dimension to $n=1$ reduces the tensor algebra to the standard matrix algebra: the $t$-product coincides with ordinary matrix multiplication, and the tensors $\Acal$ and $\Xcal$ become matrices $A\in\C^{m\times m}$ and $X\in\C^{m\times p}$, respectively. 
In this specialization, the evolution \eqref{evol1-intro} decouples across matrix columns as
\begin{equation*}
X_t(:,j)=A^t X(:,j) ,
\end{equation*}
so that the tensor dynamical sampling problem reduces to
$p$ independent classical dynamical sampling systems.

In this specialization, the Fourier representation along the third mode is trivial, since there is only one frontal slice; hence $\Acal=\widehat{\Acal}$.
Moreover, $\mathcal{I}_\lambda=\{1\}$ for every eigenvalue $\lambda$, and the vectors appearing in Theorem~\ref{thm:det} reduce to
\begin{equation*}
    V_\lambda^{(i,1)}=P_{\lambda,1}U_1^*e_i.
\end{equation*}
Thus, the   condition in Theorem~\ref{thm:det} becomes
\begin{equation*}
    \operatorname{span}\{P_{\lambda,1}U_1^*e_i:(i,1)\in\Omega_j\}
    =
    \operatorname{ran}(P_{\lambda,1}),
    \qquad \lambda\in\Lambda=\Lambda_1.
\end{equation*}
This is precisely the classical frame condition for dynamical sampling; see, for example, \cite[Theorem 2.3]{aldroubi2017dynamical}.
\end{remark}

\begin{corollary}
Suppose that no eigenvalue is shared by two distinct frontal slices of $\widehat{\Acal}$, that is, $\lambda\in\Lambda_s$ and $\lambda\in\Lambda_{s'}$ imply $s=s'$.
Then condition of Theorem~\ref{thm:det} is equivalent to the requirement that for each $j\in[p]$, $s\in[n]$, and $\lambda\in\Lambda_s$,
the set of vectors
\begin{equation*}
\cb{P_{\lambda,s}U_s^*e_i : i\in I_j}
\end{equation*}
form a frame for the eigenspace $\ran(P_{\lambda,s})$, where
\begin{equation*}
I_j=\cb{i\in[m] : (i,j,k)\in\Omega \text{ for some } k\in[n] }.
\end{equation*}
\end{corollary}
In this case, the recovery conditions decouple across frontal slices, and the sampling requirement depends only on the eigenspace structure within each slice.

\begin{corollary}
If the condition in Theorem~\ref{thm:det} holds, then for each
$j\in[p]$ the number of spatial samples must satisfy
\begin{equation*}
|\Omega_j| \ge
\max_{\lambda\in\Lambda}\, d_\lambda,
\end{equation*}
where $d_\lambda$ is given in \eqref{dim1} and it represents the total multiplicity of the eigenvalue $\lambda$
aggregated across all frontal slices of $\widehat{\Acal}$ that share this eigenvalue.
\end{corollary}

This lower bound shows how cross-slice eigenvalue sharing increases the spatial sampling burden. Indeed, when several frontal slices contribute to the same eigenvalue $\lambda$, the relevant block dimension becomes $d_\lambda=\sum_{s\in\mathcal I_\lambda}m_{\lambda,s}$, and each lateral slice must contain enough sampled spatial coordinates to cover the largest aggregated block.

The decoupled sensing vectors also explain how sampling failure is tied to the two tensor coordinates. For a fixed shared eigenvalue $\lambda$, each sampled coordinate $(i,k)$ contributes
\begin{equation*}
V_\lambda^{(i,k)}=\bigoplus_{s\in\mathcal I_\lambda}\omega_n^{-(s-1)(k-1)}P_{\lambda,s}U_s^*e_i .
\end{equation*}
Thus the mode-$1$ index $i$ changes the eigenspace components inside the slice blocks, while the mode-$3$ index $k$ changes the relative Fourier phases among the blocks indexed by $s\in\mathcal I_\lambda$. The next example makes this distinction concrete. It gives two finite masks for one shared-eigenvalue block: one fails because it misses a mode-$1$ direction, and the other fails because it has too few mode-$3$ phase patterns.

\begin{example}[Mode-resolved completion of a shared-eigenvalue block]\label{ex:mode-completion}
Fix one lateral slice $j$ and one shared eigenvalue $\lambda$. Let $m=7$, $n=9$, $\mathcal I_\lambda=\cb{1,2,3}$, and $m_{\lambda,1}=m_{\lambda,2}=m_{\lambda,3}=1$, so that $d_\lambda=3$. We identify $\bigoplus_{s\in\mathcal I_\lambda}\ran(P_{\lambda,s})$ with $\mathbb C^3$ in the order $s=1,2,3$.

First, we construct a mask that fails because it does not see enough mode-$1$ directions. Let $i_a=2$, $i_b=4$, and $i_c=6$. After identifying each one-dimensional space $\ran(P_{\lambda,s})$ with $\mathbb C$, assume that
\begin{equation*}
P_{\lambda,s}U_s^*e_{i_r}=
\begin{cases}
1, & (s,r)=(1,a),(2,b),(3,c),\\
0, & \text{otherwise}.
\end{cases}
\end{equation*}
Set
\begin{align*}
\Omega_j:=\cb{(2,3),(2,4),(2,6),(2,9),(4,2),(4,5),(4,6),(4,8)}.
\end{align*}
Then every vector $V_\lambda^{(i,k)}$ with $(i,k)\in\Omega_j$ lies in $\operatorname{span}\cb{(1,0,0),(0,1,0)}$, and hence the $\lambda$-block rank is $2$. If we add one sample at $(i_c,k_\star)=(6,8)$ and set
\begin{equation*}
\Omega_j^\sharp:=\Omega_j\cup\cb{(6,8)},
\end{equation*}
then
\begin{equation*}
V_\lambda^{(6,8)}=\rb{0,0,\omega_9^{-14}}.
\end{equation*}
Thus the augmented family has rank $3=d_\lambda$, so the blockwise frame condition is restored.

Second, we construct a mask that fails because it has too few mode-$3$ phase patterns. Assume instead that, for each selected row $i$,
\begin{equation*}
P_{\lambda,s}U_s^*e_i=1,\qquad s=1,2,3.
\end{equation*}
Then
\begin{equation*}
V_\lambda^{(i,k)}=\rb{1,\omega_9^{-(k-1)},\omega_9^{-2(k-1)}}.
\end{equation*}
Set $\Omega_j:=\cb{(1,3),(2,3),(3,3),(4,3),(6,3),(1,6),(3,6),(4,6),(5,6),(7,6)}$.
Only the two mode-$3$ indices $k_0=3$ and $k_1=6$ occur, so the phase rank is at most $2$. If we add one sample at $(i_\star,k_2)=(5,8)$ and set
\begin{equation*}
\Omega_j^\sharp:=\Omega_j\cup\cb{(5,8)},
\end{equation*}
then the three phase rows are
\begin{equation*}
\rb{1,\omega_9^{-2},\omega_9^{-4}},\qquad \rb{1,\omega_9^{-5},\omega_9^{-10}},\qquad \rb{1,\omega_9^{-7},\omega_9^{-14}}.
\end{equation*}
Since the nodes $\omega_9^{-2}$, $\omega_9^{-5}$, and $\omega_9^{-7}$ are distinct, the corresponding $3\times3$ Vandermonde matrix is nonsingular. Hence the phase rank increases from $2$ to $3=d_\lambda$, and the blockwise frame condition is restored.

Equivalently, both failures are detected by the deterministic Gram matrix
\begin{equation*}
G_{j,\lambda}(\Omega_j):=\sum_{(i,k)\in\Omega_j}V_\lambda^{(i,k)}\rb{V_\lambda^{(i,k)}}^*.
\end{equation*}
In each case, the initial mask gives $\rank\, G_{j,\lambda}(\Omega_j)=2$, whereas the augmented mask gives $\rank\, G_{j,\lambda}(\Omega_j^\sharp)=3=d_\lambda$.
\end{example}

Theorem~\ref{thm:det} and the corollaries above give the deterministic recovery picture in the diagonalizable setting. They identify the temporal injectivity threshold $T\ge|\Lambda|$, the aggregated eigenspace dimensions $d_\lambda$ that determine the spatial sampling burden, and the blockwise frame condition that detects whether a given sampling mask is sufficient. Example~\ref{ex:mode-completion} shows that this condition depends not only on the number of sampled entries, but also on how the samples are distributed across the mode-$1$ and mode-$3$ coordinates.

\subsection{Reconstruction viewpoints from the stacked and decoupled systems}
\label{subsec:reconstruction-viewpoints}

The deterministic analysis above identifies the spectral structure hidden inside the stacked recovery system. It also clarifies two reconstruction viewpoints that are useful in different regimes. The first works directly with the finite-horizon stacked systems and is the computational form used for the real-data Tensor DS experiments. The second uses the decoupled spectral representation and is useful for isolating the mechanisms tested in the synthetic experiments, especially temporal separation, blockwise spatial recovery, and Vandermonde conditioning.

The stacked formulation recovers the Fourier coefficient vector of each lateral slice directly from
\begin{equation*}
M_j(T)x(j)=\mathbf b_j(T),
\end{equation*}
where $M_j(T)$ and $\mathbf b_j(T)$ are defined in \eqref{eq:prelim-Mj}. In the noiseless full-rank case this system gives exact recovery. In numerical experiments, the same finite-horizon system is solved in the regularized least-squares form
\begin{equation*}
\widehat x(j)\in \operatorname*{arg\,min}_{x\in\C^{mn}} \norm{M_j(T)x-\mathbf b_j(T)}_2^2+\rho\norm{x}_2^2 .
\end{equation*}
The recovered vectors $\widehat x(j)$, $j=1,\ldots,p$, are then assembled into the Fourier-domain tensor and transformed back by the inverse DFT along the third mode. This is the reconstruction viewpoint used by the original three-dimensional Tensor DS implementation \cite{WangCaiHuang2024} and by the real-data experiments in this paper.

This stacked formulation has an important practical advantage. It only requires the finite-horizon matrices generated by the learned t-product dynamics. It does not require explicit diagonalization of the learned Fourier-domain frontal slices, nor a numerical grouping of eigenvalues that may be close but not exactly identical. Thus the stacked system is the natural computational object when the operator is learned from data and the available time horizon is short. The spectral results above should therefore be interpreted as revealing the internal geometry of this stacked system, rather than as requiring the real-data implementation to explicitly diagonalize the learned dynamics.

The decoupled formulation gives a second reconstruction viewpoint. Assume that the Fourier-domain frontal slices are diagonalizable and that $T\ge |\Lambda|$. For each sampled coordinate $(i,k)\in\Omega_j$, Proposition~\ref{prop:temp-regroup} gives
\begin{equation*}
{\bf y}^{(i,k)}=\mathsf V_T(\Lambda){\bf g}^{(i,k)}.
\end{equation*}
Since $\mathsf V_T(\Lambda)$ has full column rank in this regime, the temporal coefficients can be separated by
\begin{equation*}
\widehat{\bf g}^{(i,k)}=\mathsf V_T(\Lambda)^\dagger{\bf y}^{(i,k)} .
\end{equation*}
For each $\lambda\in\Lambda$, the corresponding component of the separated data then gives the spatial system
\begin{equation*}
\widehat g_\lambda(i,k)=\ab{{\bf v}_\lambda^{(i,k)},{\bf x}_\lambda^j}, \qquad (i,k)\in\Omega_j,
\end{equation*}
and the aggregated block can be recovered by
\begin{equation*}
\widehat{\bf x}_\lambda^j\in \operatorname*{arg\,min}_{{\bf x}\in\C^{d_\lambda}} \sum_{(i,k)\in\Omega_j}\vb{\widehat g_\lambda(i,k)-\ab{{\bf v}_\lambda^{(i,k)},{\bf x}}}^2 .
\end{equation*}
After this is done for all $\lambda\in\Lambda$ and all $j\in[p]$, the recovered spectral blocks are mapped back through the slice-wise eigenvector matrices and the inverse DFT.

This two-stage formulation is useful for interpreting the synthetic tests, where the spectral structure can be controlled by construction. It separates the two mechanisms in the theory. The Vandermonde solve isolates the temporal spectral coefficients, while the second-stage least-squares problems test the frame condition on the aggregated eigenspace blocks. This is also the viewpoint that leads directly to the randomized Gram matrices studied in Section~\ref{sec:rand} and to the temporal-conditioning experiment in Subsection~\ref{sec:synthetic}.

For learned real-data dynamics, however, the stacked formulation is the appropriate computational object. The decoupled formulation applies only after the temporal threshold $T\ge|\Lambda|$ has been reached, and it depends on a reliable spectral decomposition and eigenvalue grouping. Moreover, even when $T\ge|\Lambda|$, the inverse Vandermonde step may be poorly conditioned. The stability analysis in Section~\ref{sec:stability} quantifies this obstruction. Thus the decoupled formulation explains the spectral recovery mechanism, while the stacked formulation remains the practical reconstruction framework used in the real-data experiments.

\section{Random sampling robustness via Gram matrix concentration}\label{sec:rand}
In this section, we convert the deterministic injectivity condition in Theorem~\ref{thm:det} into a high-probability statement under independent random spatial sampling.
By Proposition~\ref{prop:temp-regroup}, once $T\ge |\Lambda|$, the temporal samples determine the coefficients
$\cb{g_\lambda(i,k)}_{\lambda\in\Lambda}$, and the recovery problem reduces to a spatial sensing problem on each aggregated $\lambda$-eigenspace block.
To analyze this condition probabilistically, we now pass to the coordinate realization of the block variables and sensing vectors appearing in Theorem~\ref{thm:det}.

For each $j\in[p]$ and $\lambda\in\Lambda$, let
$X_\lambda^j \in \bigoplus_{s\in\mathcal I_\lambda}\ran\rb{P_{\lambda,s}}$
denote the aggregated signal block, and let
$V_\lambda^{(i,k)} \in \bigoplus_{s\in\mathcal I_\lambda}\ran\rb{P_{\lambda,s}}$
be the sensing vectors from \eqref{sensing1}.
Under the coordinate identification induced by the bases
$U_{\lambda,s}$, $s\in\mathcal I_\lambda$, we write their coordinate representations as
${\bf x}_\lambda^j\in\C^{d_\lambda}$ and
${\bf v}_\lambda^{(i,k)}\in\C^{d_\lambda}$, respectively, namely
\begin{align}
{\bf x}_\lambda^j &:=
\bigoplus_{s\in\mathcal I_\lambda}
W_{\lambda,s}[\widehat{\Xcal}]_{:,j,s}
\in \C^{d_\lambda}, \label{sensing3} \\
{\bf v}_\lambda^{(i,k)} &:=
\frac{1}{n}
\bigoplus_{s\in\mathcal I_\lambda}
\wn^{-(s-1)(k-1)}\, U_{\lambda,s}^* e_i
\in \C^{d_\lambda}. \label{sensing2}
\end{align}
With this notation, the coefficient formula in Proposition~\ref{prop:temp-regroup} becomes
\begin{equation*}
g_\lambda(i,k)=\langle {\bf v}_\lambda^{(i,k)}, {\bf x}_\lambda^j \rangle, \qquad (i,k)\in\Omega_j,
\end{equation*}
Therefore, for fixed $j\in[p]$ and $\lambda\in\Lambda$, the frame condition in Theorem~\ref{thm:det} is equivalent to the positive definiteness of the associated Gram matrix
\begin{equation*}
    G_\lambda^j := n \sum_{(i,k)\in\Omega_j} {\bf v}_\lambda^{(i,k)}\rb{{\bf v}_\lambda^{(i,k)}}^*.
\end{equation*}
We now study this condition under an independent Bernoulli sampling model.

\subsection{Bernoulli sampling and population geometry}

We impose an independent Bernoulli sampling model where each sensor location $(i,j,k)$ is selected with probability $\alpha$. Specifically, let $\delta_{i,j,k}\sim\mathrm{Bernoulli}(\alpha)$ be i.i.d.\ random variables, and define the sampling set as
\begin{equation*}
    \Omega = \cb{ (i,j,k)\in [m]\times [p] \times [n] : \delta_{i,j,k}=1 }.
\end{equation*}
The  empirical Gram matrix can then be expressed as
\begin{equation*}
G_\lambda^j=n \sum_{i=1}^m \sum_{k=1}^n\delta_{i,j,k}
{\bf v}_\lambda^{(i,k)}({\bf v}_\lambda^{(i,k)})^*.
\end{equation*}
To analyze its stability, we define the population (deterministic) Gram matrix $\Sigma_\lambda$ as
\begin{equation*}
\Sigma_\lambda= n  \sum_{i=1}^m \sum_{k=1}^n
{\bf v}_\lambda^{(i,k)}({\bf v}_\lambda^{(i,k)})^*.
\end{equation*}
Then $\mathbb{E}[G_\lambda^j] = \alpha \Sigma_\lambda$.

\begin{lemma}\label{lem:diag}
Fix $\lambda\in\Lambda$ and let $\mathcal{I}_\lambda=\{s_1,\dots,s_r\}\subseteq[n]$.
Then $\Sigma_\lambda$ is $r\times r$ block diagonal matrix with different block sizes $m_{\lambda,s_i}$, $i=1,\ldots,r$.
More precisely, we have
\begin{equation*}
    \Sigma_\lambda = \mathrm{diag} \rb{ U_{\lambda,s_1}^* U_{\lambda,s_1}, \ldots, U_{\lambda,s_r}^* U_{\lambda,s_r} }.
\end{equation*}
\end{lemma}

Since each $U_{\lambda,s}$ has full column rank, the matrices $U_{\lambda,s}^* U_{\lambda,s}$ are positive definite for all $s\in\mathcal I_\lambda$. Hence, by Lemma~\ref{lem:diag}, $\Sigma_\lambda$ is positive definite on $\C^{d_\lambda}$. We therefore define
\begin{equation*}
    \gamma_{\lambda,\min}:=\lambda_{\min}(\Sigma_\lambda) = \min_{s\in \mathcal{I}_\lambda} \,\lambda_{\min}\rb{ U_{\lambda,s}^* U_{\lambda,s} } >0.
\end{equation*}

To convert the deterministic positivity of $\Sigma_\lambda$ into a high-probability lower bound for the random Gram matrix $G_\lambda^j$, two ingredients are needed. The first is the population conditioning parameter $\gamma_{\lambda,\min}$, which measures how well conditioned the aggregated eigenspace block is at the deterministic level. The second is a uniform bound on the size of a single random summand, which determines the concentration scale in the matrix Chernoff estimate. The next subsection introduces the corresponding incoherence parameter and derives the required spectral bound.

\subsection{Incoherence and Matrix Chernoff bound}

Define
\begin{equation*}
    \mu_{\lambda,s} := \frac{m}{m_{\lambda,s}}  \max_{i\in [m]} \norm{ U_{\lambda,s}^* e_i }_2^2
\end{equation*}
and $\mu_\lambda := \max_{s\in \mathcal{I}_\lambda} \mu_{\lambda,s}$.
Small $\mu_\lambda$ corresponds to spatial spreading.

Then
\begin{equation}\label{incoherence1}
R_\lambda:= n \max_{(i,k)}\litnorm{{\bf v}_\lambda^{(i,k)}}_2^2 = \frac{1}{n}\max_{i\in [m]} \sum_{s\in \mathcal{I}_\lambda} \norm{U_{\lambda,s}^*e_i}_2^2
\le\frac{\mu_\lambda d_\lambda}{mn}.
\end{equation}

Define
\begin{equation*}
X_{i,k}=n\, \delta_{i,j,k}
{\bf v}_\lambda^{(i,k)}
({\bf v}_\lambda^{(i,k)})^*.
\end{equation*}

Then $X_{i,k}\succeq0$ and the matrices $\{X_{i,k}\}$ are independent.
Moreover,
\begin{equation*}
\lambda_{\max}(X_{i,k})\le R_\lambda= n
\max_{(i,k)} \|{\bf v}_\lambda^{(i,k)}\|_2^2.
\end{equation*}

We also have
\begin{equation*}
\mathbb E[G_\lambda^j]= \alpha \Sigma_\lambda.
\end{equation*}

We are now in a position to apply the matrix Chernoff inequality to the sum $G_\lambda^j=\sum_{i,k} X_{i,k}$. The population lower bound is controlled by $\gamma_{\lambda,\min}$, while the size of a single random summand is controlled by $R_\lambda$. This gives a high-probability lower bound for the minimum eigenvalue of $G_\lambda^j$.

\begin{theorem}\label{thm:union}
Let $\eta,\varepsilon \in (0,1)$, and assume that $T\ge |\Lambda|$. If the sampling rate satisfies
\begin{equation}\label{eq:alpha_uniform}
    \alpha \ge \max_{\lambda \in \Lambda} \frac{3 R_\lambda}{\eta^2 \gamma_{\lambda, \min}} \log \left( \frac{d_\lambda p |\Lambda|}{\varepsilon} \right),
\end{equation}
then with probability at least $1-\varepsilon$, 
\begin{equation*}
    \lambda_{\min}(G_\lambda^j)\ge (1-\eta)\alpha \gamma_{\lambda,\min}
\end{equation*}
for all $j \in [p]$ and all $\lambda \in \Lambda$. In particular, 
the empirical Gram matrices $G_\lambda^j$ are non-singular, and hence the tensor signal $\Xcal$ is uniquely recoverable from the spatio-temporal samples $\Psi_\Xcal$.

Moreover, it is sufficient that
\begin{equation}\label{eq:suff}
    \alpha \ge \max_{\lambda \in \Lambda}
    \frac{3 \mu_\lambda d_\lambda}{\eta^2 mn \gamma_{\lambda,\min}}
    \log \rb{ \frac{d_\lambda p |\Lambda|}{\varepsilon} }.
\end{equation}
\end{theorem}

Theorem~\ref{thm:union} converts the deterministic frame condition in Theorem~\ref{thm:det} into an explicit sampling-rate requirement for random Bernoulli sampling. The bound has three main components. The quantity $d_\lambda$ is the dimension of the aggregated spectral block that must be captured, the incoherence parameter $\mu_\lambda$ measures how evenly that block is distributed across the spatial coordinates, and $\gamma_{\lambda,\min}$ measures the conditioning of the corresponding population sensing geometry. 
Thus the sufficient sampling rate deteriorates when the relevant aggregated blocks are large, poorly spread, or poorly conditioned. A brief discussion of the dependence of $\gamma_{\lambda,\min}$ on the choice of slice-wise eigenbases, and of the role of column normalization across slices, is given in Appendix~\ref{app:scaling}.

For the random sampling strategy to be practically meaningful, the sufficient lower bound on $\alpha$ should remain below one, ideally with $\alpha\ll1$. According to \eqref{eq:suff},
this is achievable when the sensing vectors are sufficiently spread across spatial coordinates, so that $\mu_\lambda$ remains small, ideally of order $\mathcal{O}(1)$. This motivates the following standard incoherence assumption, which is closely related to analogous hypotheses in compressed sensing and low-rank matrix recovery; see, for example, \cite{CandesRecht2009}.

\begin{assumption}[Eigenspace Incoherence]
For a fixed constant $\mu_0 \ge 1$ independent of the ambient dimensions $m$ and $n$, the normalized maximum sensing energy is bounded by $\mu_0$
\begin{equation*}
    \mu_{\lambda,s} = \frac{m}{m_{\lambda,s}}  \max_{i\in [m]} \norm{ U_{\lambda,s}^*e_i }_2^2 \le \mu_0 \quad \forall \lambda\in \Lambda,\ \forall s\in\mathcal I_\lambda.
\end{equation*}
\end{assumption}

The assumption above is abstract but natural: it requires the eigenspaces of the Fourier-domain dynamics to be sufficiently spread across the spatial coordinates. The following example provides a canonical benchmark in which this spreading is maximal. It also shows that, in such an ideal incoherent regime, the sufficient sampling rate in Theorem~\ref{thm:union} is governed by the largest aggregated eigenspace dimension $d_\lambda$, rather than by the full ambient dimension $mn$, and is therefore sharp up to logarithmic factors.

\begin{example}[DFT eigenbasis]
Consider the  case in which every Fourier-domain frontal slice is
diagonalized by the same unitary DFT matrix $F_m\in\C^{m\times m}$:
\begin{equation*}
    [\widehat{\Acal}]_{:,:,s}=F_mD_sF_m^*,
\qquad s\in[n],
\end{equation*}
where
\begin{equation*}
[F_m]_{a,b} = \frac{1}{\sqrt m}\exp\rb{-\frac{2\pi i}{m}(a-1)(b-1)},
\qquad a,b\in[m].
\end{equation*}
Thus $U_s=F_m$ for every slice $s$. For each eigenvalue
$\lambda\in\Lambda_s$, the matrix $U_{\lambda,s}$ is obtained by selecting
from $F_m$ the columns corresponding to the $\lambda$-eigenspace in slice $s$.  This model represents a maximally incoherent spatial geometry. Indeed, since
every entry of $F_m$ has modulus $m^{-1/2}$, for every coordinate $i\in[m]$,
\begin{equation*}
    \|U_{\lambda,s}^*e_i\|_2^2 = \frac{m_{\lambda,s}}{m}.
\end{equation*}
Consequently,
\begin{equation*}
    \mu_{\lambda,s} = \frac{m}{m_{\lambda,s}}
\max_{i\in[m]}\|U_{\lambda,s}^*e_i\|_2^2 =1, \qquad \mu_\lambda=1.
\end{equation*}
Thus the incoherence parameter takes its smallest possible value.

The population Gram matrix is also optimally conditioned. By Lemma~\ref{lem:diag}
and the orthonormality of the columns of $F_m$,
\begin{equation*}
    \Sigma_\lambda =
\mathrm{blkdiag}\big( I_{m_{\lambda,s_1}},
\ldots, I_{m_{\lambda,s_r}} \big) =
I_{d_\lambda}, \qquad \gamma_{\lambda,\min}=1,
\end{equation*}
where $\mathcal I_\lambda=\{s_1,\ldots,s_r\}$.

Finally, the maximal single-sample energy is
\begin{equation*}
    R_\lambda =
\frac{1}{n} \max_{i\in[m]} \sum_{s\in\mathcal I_\lambda} \|U_{\lambda,s}^*e_i\|_2^2 = \frac{1}{n} \sum_{s\in\mathcal I_\lambda} \frac{m_{\lambda,s}}{m}
= \frac{d_\lambda}{mn}.
\end{equation*}
Therefore, the uniform recovery guarantee in Theorem~\ref{thm:union} becomes
\begin{equation*}
    \alpha \ge \frac{3}{\eta^2 mn} \max_{\lambda\in\Lambda} d_\lambda \log\!\rb{\frac{d_\lambda p|\Lambda|}{\varepsilon}}.
\end{equation*}

Thus, in this ideal incoherent regime, the sufficient sampling rate in Theorem~\ref{thm:union} scales like
\begin{equation*}
\alpha \asymp \frac{1}{mn}\max_{\lambda\in\Lambda} d_\lambda
\end{equation*}
up to logarithmic factors. In particular, the recovery threshold is governed by the largest aggregated spectral block, rather than by the full ambient dimension $mn$.

This scaling is sharp up to logarithmic factors. Indeed, by the deterministic necessary condition in the decoupled formulation, one must have
\begin{equation*}
|\Omega_j| \ge \max_{\lambda\in\Lambda} d_\lambda
\end{equation*}
for every $j\in[p]$. Under Bernoulli sampling with rate $\alpha$, the expected number of samples in one lateral slice is $\alpha mn$, so any recovery guarantee must require, at least at the level of scaling,
\begin{equation*}
\alpha \gtrsim \frac{1}{mn}\max_{\lambda\in\Lambda} d_\lambda.
\end{equation*}
Hence Theorem~\ref{thm:union} matches the deterministic lower bound up to the logarithmic factor in this canonical Fourier-diagonalizable model.
\end{example}

For comparison, randomized dynamical sampling bounds such as \cite{huang2024robust} are also coherence-based. In the graph diffusion setting, the sampling complexity is governed by a spectral graph weighted coherence together with the intrinsic bandlimit of the signal. Thus the present result should not be viewed as a departure from the general coherence-based philosophy of randomized dynamical sampling.

The difference is where the random embedding is applied. A direct vectorized treatment of the tensor system would lead to a global coherence condition for the eigenspaces of $\textrm{bcirc}\rb{\Acal}$, or equivalently for a lifted space--time sampling matrix. Such a formulation is valid, but it does not separate the temporal Vandermonde mechanism from the spatial sensing geometry inside each shared-eigenvalue block. In contrast, Theorem~\ref{thm:union} first uses the temporal decoupling to separate the distinct eigenvalues in $\Lambda$, and then applies concentration to the empirical Gram matrices on the aggregated blocks
\begin{equation*}
\bigoplus_{s\in\mathcal I_\lambda}\ran\rb{P_{\lambda,s}}.
\end{equation*}
This is why the sufficient sampling rate is governed by $\max_{\lambda\in\Lambda}d_\lambda$ rather than by an ambient-dimensional or global lifted surrogate.

Consequently, in the canonical incoherent Fourier-diagonalizable model above, Theorem~\ref{thm:union} gives
\begin{equation*}
\alpha mn \asymp \max_{\lambda\in\Lambda}d_\lambda
\end{equation*}
up to logarithmic factors, matching the deterministic blockwise lower bound $|\Omega_j|\ge \max_{\lambda\in\Lambda}d_\lambda$ at the level of scaling. The phase transition experiments in Subsection~\ref{sec:synthetic}, especially Figure~\ref{fig:phase_transition_three}, provide numerical evidence that this scaling is sharp up to logarithmic factors in the incoherent regime. The empirical recovery boundary is closely tracked by the theoretical curve derived from Theorem~\ref{thm:union}, suggesting that the dominant bottleneck is the largest aggregated spectral block rather than a generic ambient-dimensional count.

\section{Recovery stability and temporal conditioning}\label{sec:stability}
For the stability discussion, we work with the coordinate realizations ${\bf x}_\lambda^j$ and ${\bf v}_\lambda^{(i,k)}$ of the aggregated eigenspace component \eqref{sensing3} and the sensing vectors \eqref{sensing2}, respectively.

Throughout this subsection, write $R:=|\Lambda|$.
For noisy temporal data $\widetilde{\bf y}^{(i,k)}$, define
\begin{equation*}
\widetilde{\bf g}^{(i,k)} := \underset{{\bf g}\in\C^R}{\textrm{arg min}}
\norm{\widetilde{\bf y}^{(i,k)}-\mathsf V_T(\Lambda){\bf g}}_2 = \mathsf{V}_T(\Lambda)^\dagger\widetilde{\bf y}^{(i,k)}.
\end{equation*}
For a fixed $\lambda\in\Lambda$, the least-squares reconstruction $\widehat{\bf x}_\lambda^j$ below is obtained by solving
\begin{equation*}
\widetilde{\bf x}_\lambda^j \in \underset{{\bf x}\in\C^{d_\lambda}}{\textrm{arg min}} \sum_{(i,k)\in\Omega_j} \vb{ \widetilde g_\lambda(i,k)-\ab{{\bf v}_\lambda^{(i,k)},{\bf x}} }^2,
\end{equation*}
where $\widetilde g_\lambda(i,k)$ denotes the $\lambda$-component of $\widetilde{\bf g}^{(i,k)}$, with the ordering fixed in \eqref{yg-system}.

\begin{proposition}\label{prop:temp-stability}
Fix $j\in[p]$ and $\lambda\in\Lambda$, and assume that $T\ge |\Lambda|$. Suppose that the family
\begin{equation*}
\big\{ {\bf v}_\lambda^{(i,k)} : (i,k)\in\Omega_j \big\}
\end{equation*}
has lower frame bound $C_\lambda^j > 0$. 

Let the temporal data be corrupted by additive noise:
\begin{equation*}
\widetilde{\bf y}^{(i,k)} = {\bf y}^{(i,k)} + {\boldsymbol\eta}^{(i,k)}, \qquad (i,k)\in\Omega_j.
\end{equation*}
Then least-squares reconstruction $\tilde{\bf x}_\lambda^j$ obtained from the noisy temporal data satisfies
\begin{equation*}
    \norm{ \widetilde{\bf x}_\lambda^j - {\bf x}_\lambda^j }_2 \le (C_\lambda^j)^{-1/2} 
    \frac{1}{\sigma_{\min}\rb{ \mathsf{V}_T(\Lambda) }}  \rb{ \sum_{(i,k)\in\Omega_j} \norm{ {\boldsymbol\eta}^{(i,k)} }_2^2 }^{1/2}.
\end{equation*}
\end{proposition}

Proposition~\ref{prop:temp-stability} gives a two-factor stability estimate. The spatial factor is determined by the lower frame bound $C_\lambda^j$, while the temporal factor is determined by the smallest singular value of $\mathsf{V}_T(\Lambda)$. Thus exact recovery in Theorem~\ref{thm:det} does not by itself guarantee stable recovery. The Vandermonde step may still amplify noise even after the temporal injectivity threshold has been reached.

The role of the Vandermonde factor can be compared with the space-time lifting viewpoint in dynamical sampling \cite{aldroubi2021}. More specifically, in randomized graph dynamical sampling, the temporal profile $\rb{1,\lambda,\ldots,\lambda^{T-1}}$ is normalized by its $\ell^2$ norm before applying a coherence estimate \cite{huang2024robust}.
 In the present notation this norm is
\begin{equation}\label{eq:phiT-def}
\phi_T(\rho):=\rb{\sum_{t=0}^{T-1}\rho^{2t}}^{1/2}=
\begin{cases}
\rb{\dfrac{1-\rho^{2T}}{1-\rho^2}}^{1/2}, & \rho\ne 1,\\
\sqrt{T}, & \rho=1.
\end{cases}
\end{equation}
In a lifted-space sampling analysis, such a factor is used to normalize the temporal columns before applying a coherence estimate. Here the situation is different. We must invert the unnormalized Vandermonde matrix $\mathsf V_T(\Lambda)$ in order to separate the coefficients associated with distinct eigenvalues. 
Hence unequal temporal column norms become a source of noise amplification rather than merely a normalization factor.

To make this obstruction explicit, write
\begin{equation*}
D_T:=\textrm{diag}\rb{\phi_T(|\lambda_1|),\ldots,\phi_T(|\lambda_R|)},\qquad \widetilde{\mathsf V}_T:=\mathsf V_T(\Lambda)D_T^{-1}.
\end{equation*}
Then $\widetilde{\mathsf V}_T$ has unit-norm columns and
\begin{equation*}
\mathsf V_T(\Lambda)=\widetilde{\mathsf V}_T D_T.
\end{equation*}
Thus the conditioning of $\mathsf V_T(\Lambda)$ contains at least two effects: the angular conditioning of the normalized temporal profiles, encoded by $\widetilde{\mathsf V}_T$, and the imbalance of the column norms, encoded by $D_T$. 
The estimate below isolates the second effect. It does not characterize the full conditioning of the Vandermonde matrix. Instead, it shows that, outside the stable-spectrum regime described next, imbalance in the magnitudes of the eigenvalues alone can force temporal instability as $T$ grows.

Write
\begin{equation}\label{eq:rho-pm}
\rho_+:=\max_{\lambda\in\Lambda}\vb{\lambda},\qquad \rho_-:=\min_{\lambda\in\Lambda}\vb{\lambda}.
\end{equation}
In the stable-spectrum regime, temporal oversampling does not by itself create this Vandermonde blow-up. Indeed, if
\begin{equation*}
\rho_+ := \max_{\lambda\in\Lambda}\vb{\lambda}<1,
\end{equation*}
then the Vandermonde Gram matrices satisfy
\begin{equation*}
\mathsf V_T(\Lambda)^*\mathsf V_T(\Lambda)=\rb{\sum_{t=0}^{T-1}\overline{\lambda_a}^{\,t}\lambda_b^t}_{a,b=1}^{|\Lambda|}
\end{equation*}
and converge to
\begin{equation*}
\rb{\frac{1}{1-\overline{\lambda_a}\lambda_b}}_{a,b=1}^{|\Lambda|}.
\end{equation*}
Moreover, the difference between two consecutive Gram matrices is positive semidefinite, since it is the rank-one contribution from the new time row. Since the eigenvalues are distinct, the first $|\Lambda|$ rows already form a nonsingular square Vandermonde matrix. Hence $\kappa_2\rb{\mathsf V_T(\Lambda)}$ is uniformly bounded for $T\ge |\Lambda|$. Thus, when $\rho_+<1$, increasing $T$ does not cause the temporal-conditioning deterioration considered below. Ill-conditioning may still occur if the eigenvalues are nearly colliding, but it is not caused by temporal oversampling.

\begin{theorem}\label{thm:temp-conditioning}
Assume that $\rho_-<\rho_+$. For $T\ge |\Lambda|$, define
\begin{equation*}
    \Phi_T := \frac{\phi_T(\rho_-)}{\phi_T(\rho_+)}.
\end{equation*}
Then the following hold:
\begin{enumerate}[label={\rm(\roman*)}]
    \item For every $T\ge |\Lambda|$,
    \begin{equation*}
        \frac{1}{\kappa_2\rb{ \mathsf{V}_T(\Lambda) }} \le \Phi_T.
    \end{equation*}
    
    \item The sequence $\cb{ \Phi_T }_{T\ge R}$ is strictly decreasing in $T$.

    \item If $\rho_+ \ge 1$, then
    \begin{equation*}
        \Phi_T \to 0 \qquad\text{as } T\to\infty,
    \end{equation*}
\end{enumerate}
where $\kappa_2(A):=\sigma_{\max}(A)/\sigma_{\min}(A)$ denotes the standard $2$-norm condition number of $A$.
\end{theorem}

Theorem~\ref{thm:temp-conditioning} shows that temporal oversampling is not automatically beneficial.
When the spectral magnitudes are unbalanced, the columns of $\mathsf V_T(\Lambda)$ become increasingly uneven in norm, and this forces deterioration of the Vandermonde conditioning.

\begin{corollary}\label{coro:explicit}
Assume that $0\le \rho_-<1<\rho_+$.
Let $\tau>0$. Then any temporal depth $T\ge |\Lambda|$ satisfying
\begin{equation}\label{eq:Ttau}
\rho_+^{T-1}\ge \frac{1}{\tau\sqrt{1-\rho_-^2}}
\end{equation}
guarantees
\begin{equation*}
\frac{1}{\kappa_2\rb{\mathsf V_T(\Lambda)}}\le \tau.
\end{equation*}
Moreover, a more transparent sufficient condition for \eqref{eq:Ttau} is
\begin{equation}\label{eq:Ttau-rough}
T\ge \max\cb{ |\Lambda|,\, 1+\frac{\rho_+}{\rho_+-1}
\log\rb{\frac{1}{\tau\sqrt{1-\rho_-^2}}} }.
\end{equation}
In particular, the required temporal depth grows at most logarithmically in $1/\tau$ and becomes larger as $\rho_+$ approaches $1$ from above.
\end{corollary}

Corollary~\ref{coro:explicit} makes the preceding obstruction explicit in the mixed-modulus case $0\le \rho_-<1<\rho_+$.
In this regime,
\begin{equation*}
\Phi_T \le \frac{1}{\sqrt{1-\rho_-^2}}\,\rho_+^{-(T-1)}.
\end{equation*}
Thus the upper bound for $1/\kappa_2\rb{\mathsf V_T(\Lambda)}$ decays at least geometrically with $T$.
The rough bound \eqref{eq:Ttau-rough} gives the corresponding temporal scale: the depth needed to certify $1/\kappa_2\rb{\mathsf V_T(\Lambda)}\le\tau$ grows logarithmically in $1/\tau$ and becomes larger as $\rho_+$ approaches $1$.

This conditioning mechanism is consistent with the reconstruction viewpoint described in Subsection~\ref{subsec:reconstruction-viewpoints}. The decoupled two-stage formulation is well suited for controlled spectral diagnostics, where the Vandermonde factor can be isolated, whereas the real-data experiments use the finite-horizon stacked Tensor DS system to avoid an explicit and potentially ill-conditioned spectral separation step.

\section{Numerical Experiments}\label{sec:experiment}

We present numerical experiments designed to test the theoretical and computational claims of the paper. The experiments have two complementary purposes. First, synthetic tensor dynamical systems are used to validate the spectral-block randomized sampling threshold in Theorem~\ref{thm:union} and the Vandermonde-conditioning mechanism in Proposition~\ref{prop:temp-stability} and Theorem~\ref{thm:temp-conditioning}. Second, real spatio-temporal datasets are used to evaluate the proposed Tensor DS method against low-rank, tensor-completion, graph-based, and dynamical sampling baselines. Thus the synthetic experiments are designed to validate the theory, while the real-data experiments test whether the same tensor dynamical modeling principles lead to accurate reconstruction in practical settings.

\subsection{Synthetic phase transitions and temporal stability}\label{sec:synthetic}

We first consider synthetic tensor dynamical systems in the diagonalizable setting. These experiments isolate two numerical consequences of the theory. The phase transition experiments test whether the empirical recovery boundary follows the randomized sampling bound in Theorem~\ref{thm:union}. 
The temporal stability experiment tests a different issue, namely the gap between identifiability and numerical conditioning. Theorem~\ref{thm:det} shows that temporal separation is possible once $T\ge |\Lambda|$, while Proposition~\ref{prop:temp-stability} shows that stable recovery also depends on the Vandermonde factor $\sigma_{\min}\rb{\mathsf V_T(\Lambda)}^{-1}$. Thus this subsection tests both the sampling threshold predicted by the randomized theory and the temporal conditioning mechanism developed in the stability analysis.

\medskip \noindent
\textbf{Phase transition experiments.}
To validate the sampling threshold predicted by the theory, we performed phase transition experiments for the random sampling model. 
The experiments are written in the tensor notation of the paper, with signal tensors $\Xcal\in\C^{m\times p\times n}$ and tensor operators $\Acal\in\C^{m\times m\times n}$. Thus $m$ is the first-mode dimension, $p$ is the second-mode dimension (number of lateral slices), and $n$ is the third-mode dimension.
The goal is to examine empirically how the recovery success probability changes with the sampling rate $\alpha$ and with the structural parameters of the problem.

For each grid point in the phase diagrams, we generated $40$ independent trials. A trial was declared successful if the relative reconstruction error was smaller than $10^{-6}$. The color value at each grid point is the empirical success probability, obtained by averaging these $40$ binary success indicators, and therefore lies between $0$ and $1$.

The three phase-transition panels are generated as follows.
\begin{itemize}[leftmargin=1.5em]
    \item \textbf{$m$-sweep.} We vary $m=4,\ldots,30$ and the sampling rate $\alpha$, while fixing $p=4$, $n=5$, and $d_\lambda=10$. We use $|\Lambda|=10$ for $m=4,\ldots,20$ and $|\Lambda|=15$ for $m=21,\ldots,30$.

    \item \textbf{$n$-sweep.} We vary $n=4,\ldots,30$ and the sampling rate $\alpha$, while fixing $p=4$, $m=5$, and $d_\lambda=10$. We use $|\Lambda|=10$ for $n=4,\ldots,20$ and $|\Lambda|=15$ for $n=21,\ldots,30$.

    \item \textbf{$d_\lambda$-sweep.} We fix $m=10$, $p=4$, $n=10$, $|\Lambda|=35$, and vary $d_\lambda=3,\ldots,14$. For each value of $d_\lambda$, the eigenvalue assignment is generated so that one target eigenvalue has aggregated multiplicity $d_\lambda$, while the remaining eigenvalue positions are assigned to the other distinct eigenvalues. This keeps the number of distinct eigenvalues fixed and isolates the effect of the aggregated block size.
\end{itemize}
In all three panels, the temporal depth is chosen as $T=|\Lambda|+3$, so that the temporal injectivity threshold $T\ge|\Lambda|$ is satisfied with mild oversampling.
Here $d_\lambda$ is the aggregated eigenspace dimension defined in \eqref{dim1}, and $|\Lambda|$ is the number of distinct eigenvalues appearing across the Fourier-domain frontal slices.

The eigenvalues used in these phase transition experiments satisfy $\rho_+<1$. Thus, even when $T>|\Lambda|$, the phase diagrams isolate the randomized spatial sampling threshold from the temporal Vandermonde blow-up discussed in Section~\ref{sec:stability}. The complementary regime $1<\rho_+$ is tested separately in the temporal stability experiment below.

\begin{figure*}[ht!]
    \centering
    \includegraphics[width=0.30\textwidth]{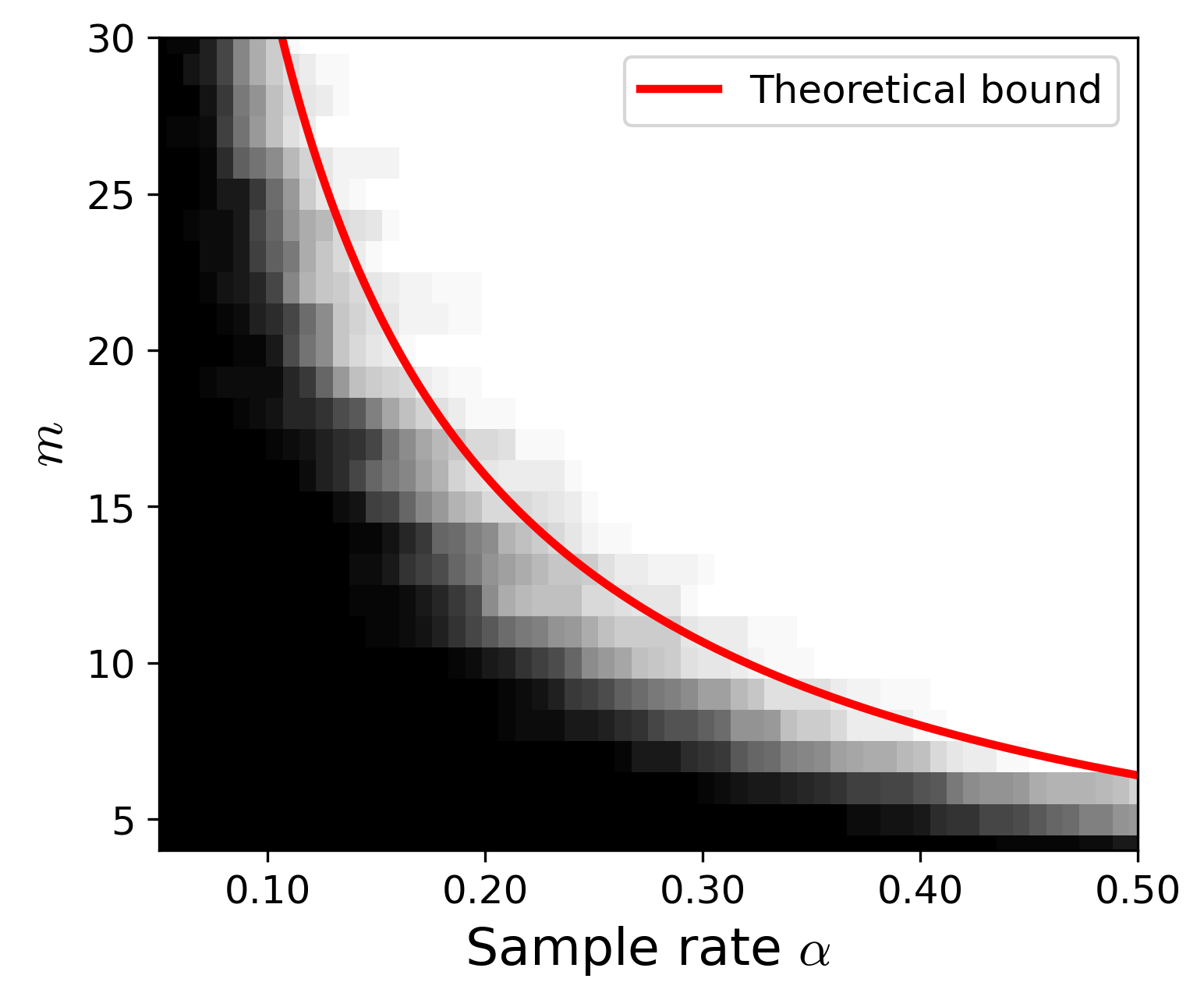}
    \hfill
    \includegraphics[width=0.30\textwidth]{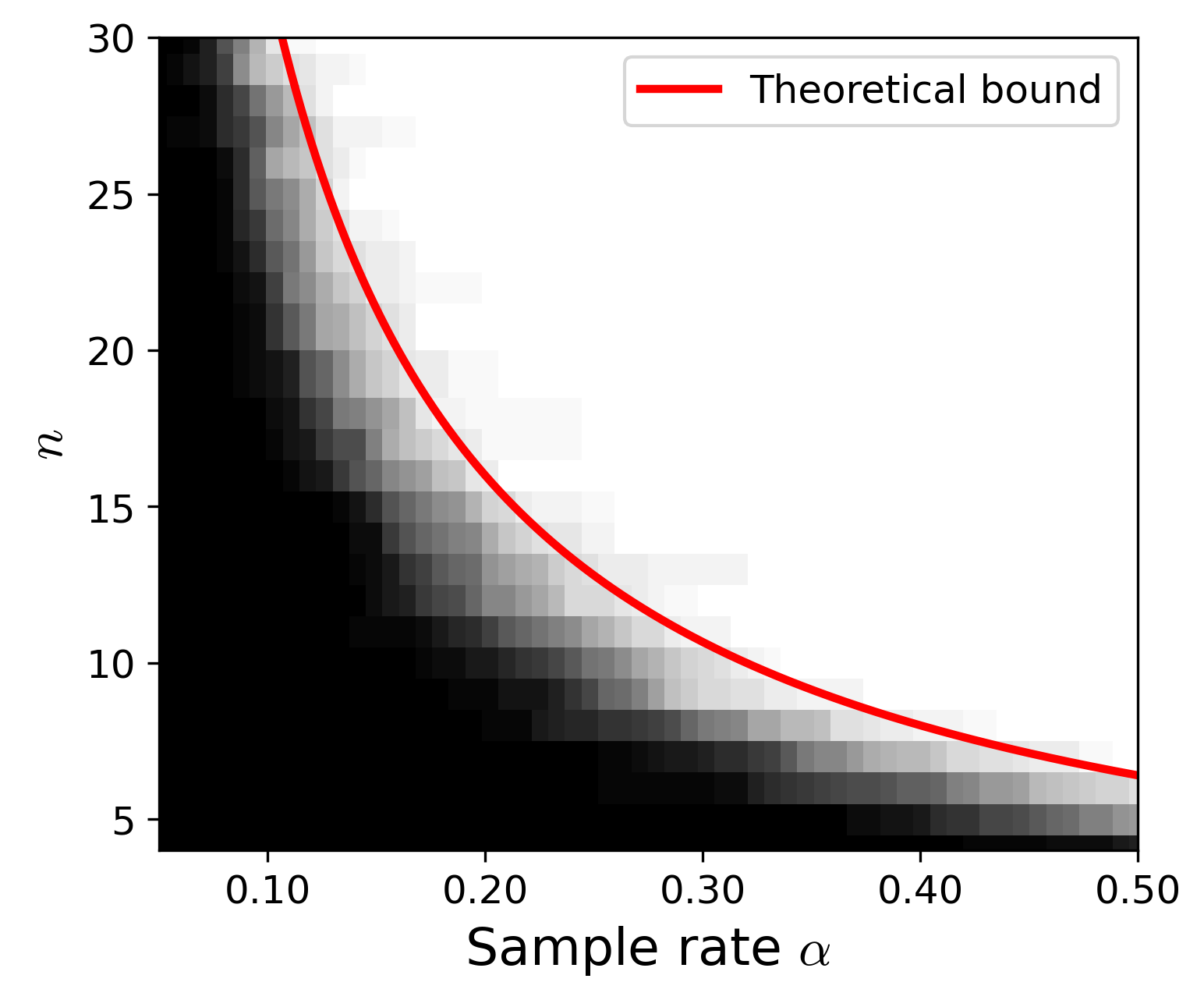}
    \hfill
    \includegraphics[width=0.30\textwidth]{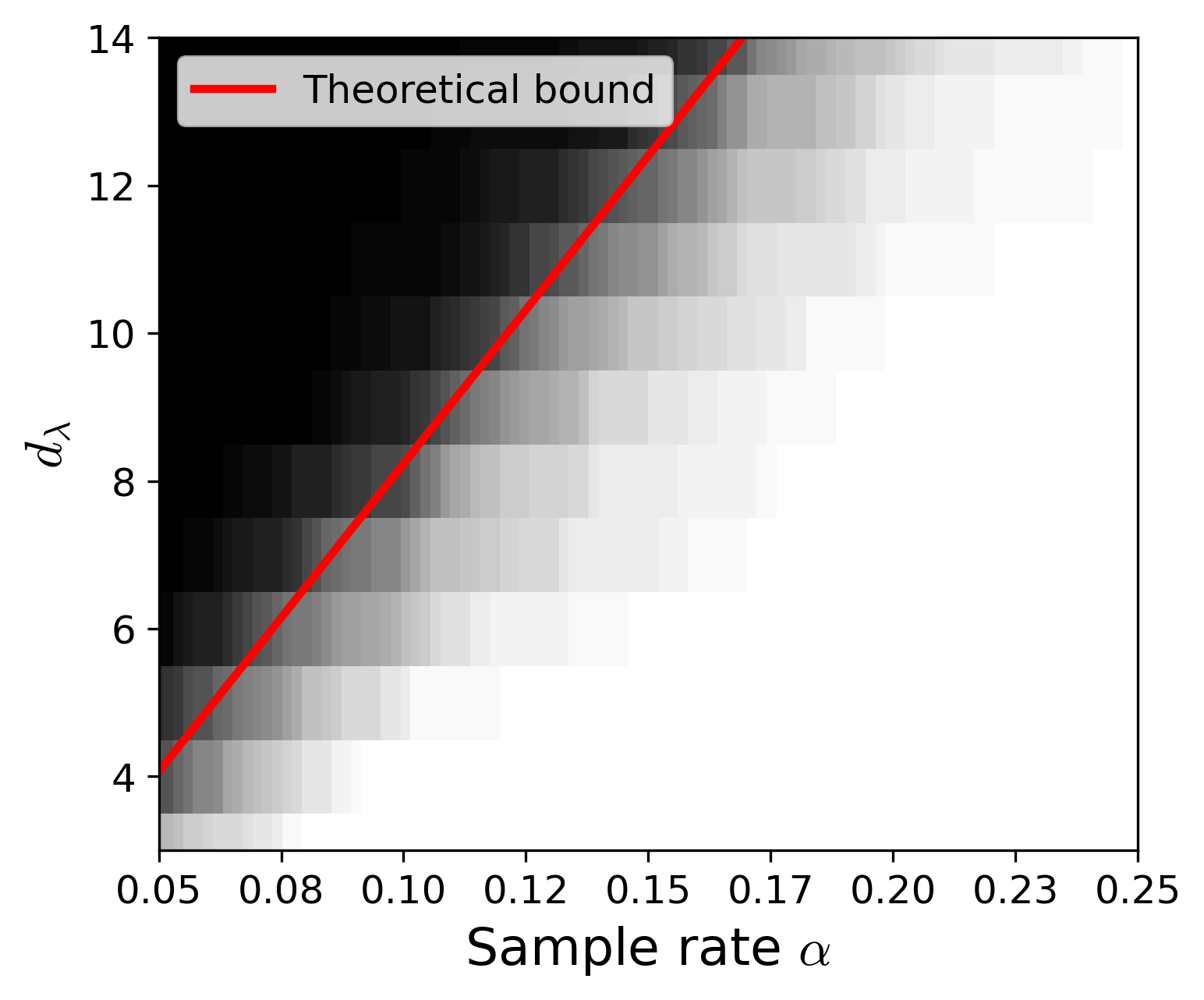}

    \vspace{0.5em}
    \includegraphics[width=0.42\textwidth]{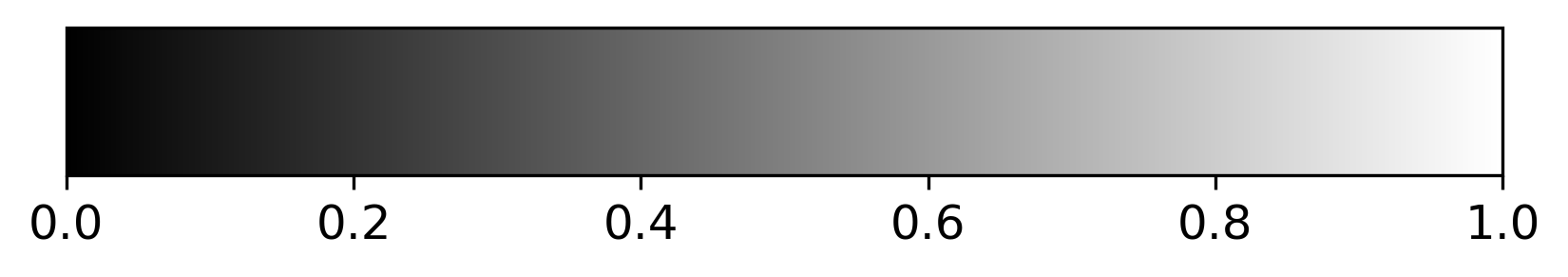}

    \caption{Phase transition diagrams for the synthetic random sampling experiments. Each grid point reports the empirical recovery probability over $40$ independent trials, with success declared when the relative reconstruction error is below $10^{-6}$.
    The left panel varies the first-mode dimension $m$, the middle panel varies the third-mode dimension $n$, and the right panel varies the aggregated spectral block dimension $d_\lambda$. The red curves indicate the theoretical sampling trends predicted by 
Theorem~\ref{thm:union}.
Across all three panels, the empirical phase-transition boundaries closely
follow the predicted trends: the critical sampling rate decreases as $m$
or $n$ increases and increases as $d_{\lambda}$ grows.
This agreement provides numerical evidence consistent with the theoretical
finding that the leading-order recovery threshold scales as
$\alpha_{\mathrm{crit}} \asymp d_{\lambda}/(mn)$.}
    \label{fig:phase_transition_three}
\end{figure*}

Figure~\ref{fig:phase_transition_three} summarizes the empirical
phase-transition behavior with respect to the sampling rate $\alpha$ and the
structural parameters $m$, $n$, and $d_\lambda$. In all three panels, a clear
transition from failure to successful recovery is observed: below the
transition boundary, the empirical success probability remains close to zero,
whereas above it, the success probability rapidly approaches one. More
importantly, the empirical transition boundaries closely follow the theoretical
trends represented by the red curves. In particular, the critical sampling
rate decreases as either $m$ or $n$ increases and increases as the aggregated
eigenspace dimension $d_\lambda$ grows. 
These numerical trends are consistent with the sampling-rate dependence
predicted by Theorem~\ref{thm:union}, namely,
\begin{equation*}
    \alpha_{\mathrm{crit}}
    \asymp
    \frac{d_\lambda}{mn},
\end{equation*}
up to multiplicative constants and logarithmic factors. The agreement between
the empirical phase-transition boundaries and the theoretical curves provides
numerical evidence that the predicted spectral-block scaling captures the
dependence of the recovery threshold on the problem dimensions. In particular,
the experiments support the theoretical finding that the dominant spatial
sampling burden is governed by the aggregated eigenspace dimension
$d_\lambda$, rather than by a generic ambient-dimensional count.

\medskip \noindent
\textbf{Temporal stability experiment.}
We next examine the effect of the temporal depth $T$ on numerical stability. In the decoupled formulation, the temporal Vandermonde matrix $\mathsf V_T(\Lambda)$ first becomes injective when $T\ge |\Lambda|$. This injectivity threshold is necessary for separating the coefficients associated with distinct eigenvalues, but it does not guarantee numerical stability. Proposition~\ref{prop:temp-stability} shows that the reconstruction error contains the temporal factor $\sigma_{\min}\rb{\mathsf V_T(\Lambda)}^{-1}$, while Theorem~\ref{thm:temp-conditioning} shows that imbalance among the eigenvalue magnitudes can force the conditioning of $\mathsf V_T(\Lambda)$ to deteriorate as $T$ grows. 
Similar conditioning effects were also observed numerically in the earlier work \cite{WangCaiHuang2024}.

In the experiment, we use a synthetic diagonalizable t-product system with
\begin{equation*}
(m,p,n)=(10,4,3), \qquad |\Lambda|=5,\qquad \max_{\lambda\in\Lambda}d_\lambda=9.
\end{equation*}
The eigenvalues are chosen in the mixed-modulus regime with
\begin{equation*}
\rho_- =0.45,\qquad \rho_+ =1.95,
\end{equation*}
so that the temporal Vandermonde matrix is identifiable once $T\ge |\Lambda|$ but may become increasingly ill-conditioned as $T$ grows. We fix the sampling rate at $\alpha=0.7$ and use a fixed Bernoulli spatial mask across all time samples. For each value of $T$, the initial tensor is recovered by the decoupled two-stage procedure described in Subsection~\ref{subsec:reconstruction-viewpoints}. 
The reported reconstruction errors are the median relative errors over $20$ repeated noise realizations for each noise level.

We also mark the sufficient temporal scale predicted by Corollary~\ref{coro:explicit}. For the tolerance $\tau=10^{-4}$ used in the experiment, define
\begin{equation*}
T_\tau:=\min\cb{T\ge |\Lambda|:\rho_+^{T-1}\ge \frac{1}{\tau\sqrt{1-\rho_-^2}}}.
\end{equation*}
Then Corollary~\ref{coro:explicit} guarantees that every $T\ge T_\tau$ satisfies
\begin{equation*}
\frac{1}{\kappa_2\rb{\mathsf V_T(\Lambda)}}\le \tau.
\end{equation*}
Thus $T_\tau$ is not the injectivity threshold $T\ge |\Lambda|$, but a certified conditioning scale beyond which the reciprocal condition number is guaranteed to be at most $\tau$.

\begin{figure*}[ht!]
    \centering
    \includegraphics[width=0.49\textwidth]{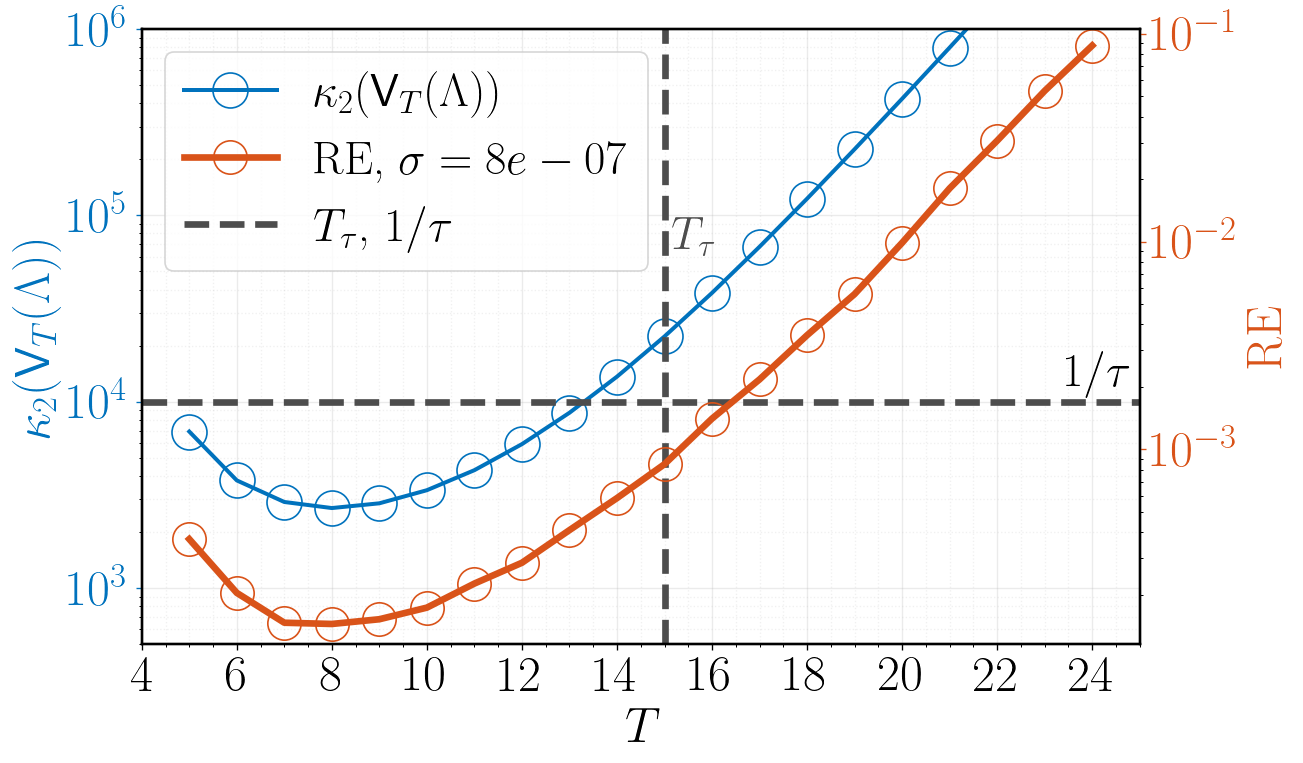}
    \hfill
    \includegraphics[width=0.49\textwidth]{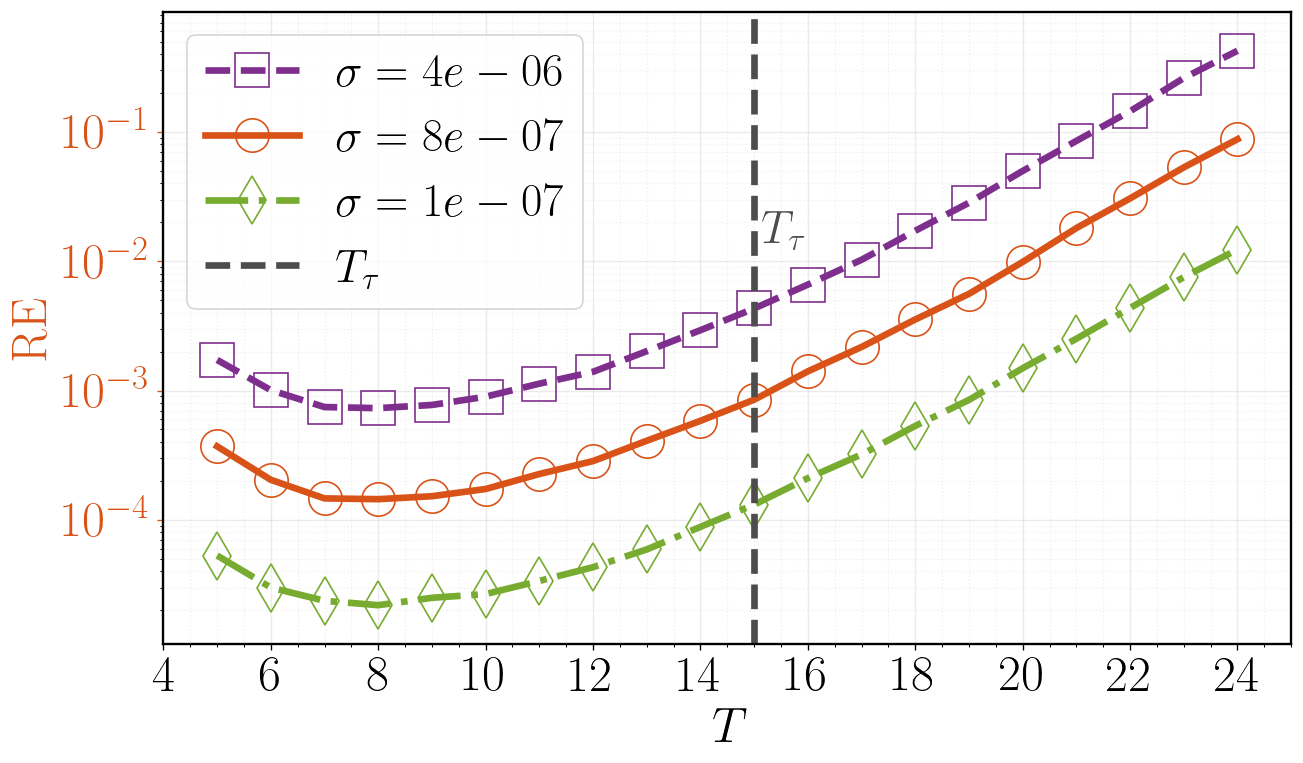}
    \caption{Vandermonde conditioning and two-stage recovery error as the temporal depth $T$ increases. The left panel shows the condition number $\kappa_2\rb{\mathsf V_T(\Lambda)}$ together with the representative relative reconstruction error for $\sigma=8\times10^{-7}$, while the right panel shows the relative reconstruction error of the decoupled two-stage recovery under several scaled additive Gaussian noise levels. The horizontal dashed line marks $1/\tau$, and the vertical dashed line marks the certified conditioning scale $T_\tau$ from Corollary~\ref{coro:explicit}, with $\tau=10^{-4}$. Although $T\ge|\Lambda|$ is sufficient for temporal identifiability, increasing $T$ can worsen numerical conditioning in the mixed-modulus regime $0\le\rho_-<1<\rho_+$, and this deterioration is reflected in the reconstruction error.}
    \label{fig:vandermonde_cond}
\end{figure*}

Figure~\ref{fig:vandermonde_cond} illustrates the two competing effects of increasing the temporal depth $T$. Once the injectivity threshold $T\ge|\Lambda|$ is reached, additional temporal samples can improve the two-stage reconstruction for moderate values of $T$. For larger $T$, however, the unnormalized Vandermonde columns become increasingly imbalanced in norm in the mixed-modulus regime, and the condition number $\kappa_2\rb{\mathsf V_T(\Lambda)}$ grows rapidly. The resulting increase in the reconstruction error is consistent with Proposition~\ref{prop:temp-stability} and Theorem~\ref{thm:temp-conditioning}. Thus temporal oversampling can improve identifiability, but it may degrade stability when the Vandermonde separation step is ill-conditioned.

\subsection{Real-data reconstruction setup}

We next evaluate the proposed Tensor DS method on real spatio-temporal data. Unlike the synthetic phase transition experiments, these experiments do not assume an exactly diagonalizable generative model with known spectral blocks. Instead, they test whether the tensor dynamical structure learned from training data gives stable and accurate reconstruction from partial observations.
This subsection describes the datasets, sampling protocol, baselines, and evaluation metrics used in both real-data experiments. The dataset-specific results are then reported separately in the following two subsections.

\medskip \noindent
\textbf{Datasets and preprocessing.}
We consider two real datasets: a monthly sea surface temperature (SST) dataset from the NOAA Optimum Interpolation (OI) SST V2 product, obtained through the NOAA Physical Sciences Laboratory (PSL) archive \cite{reynolds2002improved}, 
and a velocity dataset extracted from the forced isotropic turbulence dataset of the Johns Hopkins Turbulence Database (JHTDB) \cite{li2008public,perlman2007data}.

For the SST experiment, we crop a spatial subregion of size $60\times 60$ from the original field, using the first monthly record as the starting time and the 1-based array indices $61{:}120$ in latitude and $121{:}180$ in longitude. This yields a sequence of monthly SST snapshots.
The first $60$ months are used for training, and the subsequent six months are used as the dynamical-sampling and evaluation horizon.
The data are normalized to $[0,1]$ by min-max scaling.

For the velocity experiment, each frame is originally a three-dimensional field with three velocity components from the JHTDB forced isotropic turbulence dataset. We first convert the three components into velocity magnitude and then extract a fixed two-dimensional slice to form the spatio-temporal field used in the experiment.

\medskip \noindent
\textbf{Sampling protocol.} 
For both datasets, the sampling rate is varied over
\begin{equation*}
\alpha\in\cb{0.2,0.3,\ldots,0.9},
\end{equation*}
and reconstruction performance is evaluated under multiple independent random masks.

This sampling protocol is consistent with the independent spatial sampling model used in Section~\ref{sec:rand}, although the real datasets need not satisfy the exact spectral assumptions imposed in the synthetic experiments. Thus the real-data experiments should be interpreted as algorithmic validation rather than as direct tests of the exact hypotheses of Theorem~\ref{thm:union}.

\medskip \noindent
\textbf{Compared methods.}
We compare the proposed Tensor DS method with the following five baselines.

\begin{itemize}[leftmargin=1.5em]
\item \textbf{POD-LD.}
POD-LD is a POD/SVD-based low-rank reconstruction baseline with learned latent dynamics. The spatio-temporal data are unfolded into a matrix whose columns are vectorized spatial snapshots. From the fully observed training snapshots, we remove the temporal mean and compute a rank-$r_{\mathrm{mc}}$ POD/SVD basis. A first-order linear transition is then fitted in the latent coordinates. Given fixed-mask observations over the testing horizon, the initial latent state is recovered by ridge-regularized least squares and then propagated by the learned latent transition. This baseline follows the common POD-based strategy of estimating low-dimensional coefficients from incomplete observations \cite{everson1995,willcox2006,peherstorfer2020}, but is adapted here by coupling the POD coefficients through a learned latent linear dynamics.

\item \textbf{Tucker.}
This is a Tucker-type reduced dynamical baseline based on low multilinear-rank factorization \cite{tucker1966,KoldaBader2009}. The spatial factor matrices are learned from the training data, the snapshots are projected onto a low-dimensional Tucker core space, and a linear transition is fitted between consecutive core coefficients. The initial core coefficient is then estimated from the fixed-mask multi-time observations and propagated by the learned reduced core dynamics.

\item \textbf{GR.}
GR is a graph-regularized interpolation baseline for graph-structured data \cite{narang2013}. Each spatial snapshot is viewed as a graph signal on the spatial grid graph, and missing entries are reconstructed by balancing agreement with the sampled locations and smoothness with respect to the graph Laplacian. This baseline exploits spatial graph regularity, but it does not learn a tensor-valued system operator.

\item \textbf{Vector DS.}
Vector DS is used as a fully vectorized homogeneous dynamical sampling baseline, motivated by the graph dynamical sampling framework of \cite{gong2026randomized}. In contrast to the graph-spectral affine formulation in that work, our implementation does not impose graph-bandlimitedness or recover an additional source term. Each snapshot is reshaped into a vector, a full-state linear evolution operator is learned from consecutive training snapshots, and the initial testing state is recovered from the fixed-mask multi-time observations by ridge-regularized least squares before being propagated by the learned vector operator. Thus this baseline isolates the effect of vectorizing the spatial field and learning a high-dimensional vector evolution model, without preserving the tensor geometry.

\item \textbf{SRS.}
SRS is a static random sampling baseline for bandlimited graph signals \cite{puy2018}. Each testing frame is reconstructed independently from sampled graph nodes using a graph-spectral reconstruction with a training-selected spectral cutoff and regularization parameter. Unlike Tensor DS and the dynamical baselines, SRS does not fit a multi-time dynamical observation system for the initial state.
\end{itemize}

These baselines separate the effects of vectorized low-rank modeling, low multilinear-rank tensor modeling, graph smoothness, full vectorized dynamics, and graph-spectral static interpolation. In particular, Vector DS serves as the closest fully vectorized non-tensor dynamical sampling counterpart. 
Tensor DS differs from all of them by learning and exploiting a t-product system operator, thereby preserving the tensor-valued dynamical structure rather than reducing the data to a matrix subspace, a Tucker-core model, or a full vectorized model.

\medskip \noindent
\textbf{Learning the t-product system operator.}
For both real-data experiments, the Tensor DS method first learns a t-product system operator $\mathcal{A}$ for the evolution model \eqref{evol1-intro}. Using the Fourier-domain representation reviewed in Section~\ref{sec:prelim}, we fit each frequency-slice operator from consecutive training snapshots by a truncated DMD-type least-squares estimate \cite{schmid2010dynamic,tu2014dynamic}. The truncation parameter $r_{\mathrm{tdmd}}$ specifies the SVD rank used in this slice-wise fit, and the learned slice operators are transformed back by the inverse DFT before reconstruction.
The resulting operator is then inserted into the finite-horizon stacked Tensor DS system described in Subsection~\ref{subsec:reconstruction-viewpoints}, with the ridge parameter listed in the dataset-specific tables.

\medskip \noindent
\textbf{Evaluation metrics.}
We evaluate reconstruction quality using relative error (RE), signal-to-noise ratio (SNR), mean absolute error (MAE), and structural similarity (SSIM). Let $x_{\mathrm{true}}\in\mathbb{R}^N$ denote the ground-truth signal and $x_{\mathrm{rec}}\in\mathbb{R}^N$ denote the reconstructed signal. The relative error is defined by
\begin{equation*}
    \mathrm{RE} :=
\frac{\|x_{\mathrm{rec}}-x_{\mathrm{true}}\|_2}{\|x_{\mathrm{true}}\|_2},
\end{equation*}
and the reconstruction SNR is defined by
\begin{equation*}
    \mathrm{SNR} := 20\log_{10}
\frac{\|x_{\mathrm{true}}\|_2}{\|x_{\mathrm{rec}}-x_{\mathrm{true}}\|_2},
\end{equation*}
In the numerical implementation, a small value $\varepsilon=10^{-12}$ is added to the denominator only to avoid division by zero.

The mean absolute error is defined as
\begin{equation*}
    \mathrm{MAE} =
\frac{1}{N}\sum_{i=1}^{N}
\left|x_{\mathrm{rec},i}-x_{\mathrm{true},i}\right|.
\end{equation*}

For image-like data, we also report the structural similarity index (SSIM). Given two images $X_{\mathrm{rec}}$ and $X_{\mathrm{true}}$, SSIM is defined by
\begin{equation*}
    \mathrm{SSIM}(X_{\mathrm{rec}},X_{\mathrm{true}})
= \frac{
(2\mu_{\mathrm{rec}}\mu_{\mathrm{true}}+C_1)
(2\sigma_{\mathrm{rec,true}}+C_2)
}{
(\mu_{\mathrm{rec}}^2+\mu_{\mathrm{true}}^2+C_1)
(\sigma_{\mathrm{rec}}^2+\sigma_{\mathrm{true}}^2+C_2) },
\end{equation*}
where $\mu_{\mathrm{rec}}$ and $\mu_{\mathrm{true}}$ are the image means, $\sigma_{\mathrm{rec}}^2$ and $\sigma_{\mathrm{true}}^2$ are the variances, and $\sigma_{\mathrm{rec,true}}$ is the covariance between the reconstructed and true images. The constants $C_1$ and $C_2$ are used to stabilize the ratio.

\subsection{Sea surface temperature reconstruction}
\

\medskip \noindent
\textbf{Experimental setup and ROI interpretation.}
Table~\ref{tab:sst-hyperparameters} summarizes the experimental settings used in the SST experiments.

\begin{table}[ht!]
\centering
\caption{Experimental settings for the SST experiments.}
\label{tab:sst-hyperparameters}
\begin{tabular}{lll}
\toprule
Category & Setting & Value \\
\midrule
\midrule
\multirow{5}{*}{Dataset} 
& Spatial crop indices & lat $61{:}120$, lon $121{:}180$ \\
& Spatial crop size & $60 \times 60$ \\
& Training months & $1\text{-}60$ \\
& Dynamical-sampling / evaluation months & $61\text{-}66$ \\
& Normalization & min--max scaling to $[0,1]$ \\
\midrule
\multirow{4}{*}{Sampling} 
& Sampling rates & $0.2,0.3,\dots,0.9$ \\
& Number of trials & $10$ \\
& Sampling mask & fixed across the observed months \\
& Noise level $\sigma$ & $0$ \\
\midrule
\multirow{3}{*}{Tensor DS} 
& Tensor size $(m,p,n)$ & $(60,1,60)$ \\
& SVD rank cap $r_{\mathrm{tdmd}}$ & $55$\\
& Least-squares ridge parameter & $10^{-6}$ \\
\midrule
\multirow{2}{*}{POD-LD} 
& Low-rank truncation $r_{\mathrm{mc}}$ & $10$ \\
& Ridge parameter & $10^{-6}$ \\
\midrule
\multirow{2}{*}{Tucker} 
& Tucker ranks & $(14,14)$ \\
& Ridge parameter & $10^{-6}$ \\
\midrule
\multirow{3}{*}{GR} 
& Graph type & $60 \times 60$ grid graph \\
& \multirow{2}{*}{Regularization parameters} & spatial Laplacian $10^{-2}$\\ 
&&temporal smoothing $5\times10^{-2}$ \\
\midrule
\multirow{2}{*}{Vector DS} 
& Learning ridge parameter & $10^{-4}$ \\
& Reconstruction ridge parameter & $10^{-3}$ \\
\midrule
\multirow{5}{*}{SRS}
& \multirow{2}{*}{Spectral cutoff rule} & energy ratio $0.995$\\
&& max $300$, min $100$ modes \\
& Effective spectral cutoff & $100$ modes \\
& Reconstruction parameter & selected from training data \\
& Selected value & $10^{-5}$ \\
\bottomrule
\end{tabular}
\end{table}

For the tensor dynamical sampling method, each monthly SST frame is represented as a third-order tensor of size $60 \times 1 \times 60$. The t-product system operator is learned from the first $60$ months using the slice-wise truncated DMD-type procedure described above.

The subsequent six months, Months $61$--$66$, form the dynamical-sampling and evaluation horizon. For each random trial, a fixed spatial mask is applied to all six months. Tensor DS uses these masked multi-time observations to recover the initial testing state at Month $61$, and the learned t-product system operator is then applied recursively to reconstruct the remaining months in the same horizon. No observations beyond Month $66$ are used in this experiment.

The remaining rank and regularization parameters in Table~\ref{tab:sst-hyperparameters} specify the reduced models used by the competing baselines. For POD-LD, $r_{\mathrm{mc}}$ is the dimension of the learned latent subspace. For Tucker, the listed ranks are the two spatial Tucker ranks used to form the reduced core dynamics. The graph-based baselines, GR and SRS, use the same $60\times 60$ grid graph. Vector DS learns a full vectorized dynamical operator from the training snapshots and reconstructs the initial testing state by ridge-regularized multi-time least squares. SRS performs framewise static graph-spectral reconstruction with the spectral cutoff and regularization parameter selected from the training data.

The SST domain used in this experiment spans approximately $29.5^\circ\mathrm{N}$--$29.5^\circ\mathrm{S}$ and $120.5^\circ\mathrm{E}$--$179.5^\circ\mathrm{E}$, covering a broad portion of the tropical western Pacific. This region overlaps the Indo-Pacific warm pool, which is one of the warmest oceanic regions on Earth and plays a major role in global oceanographic and climatic processes \cite{dedeckker2016warmpool}. The two highlighted regions of interest (ROIs) are therefore not arbitrary visual patches. The first ROI, located over approximately $6.5^\circ\mathrm{N}$--$23.5^\circ\mathrm{S}$ and $123.5^\circ\mathrm{E}$--$153.5^\circ\mathrm{E}$, lies on the western-to-southern flank of the warm-pool region, where the SST field exhibits a pronounced warm band and stronger thermal gradients. The second ROI, located over approximately $5.5^\circ\mathrm{N}$--$12.5^\circ\mathrm{S}$ and $160.5^\circ\mathrm{E}$--$178.5^\circ\mathrm{E}$, lies close to the reported heat-center region of the western Pacific warm pool, whose climatological mean position is near $0.4^\circ\mathrm{S},169.0^\circ\mathrm{E}$ \cite{hu2012heatcenter}. These ROIs provide physically meaningful local windows for comparing whether a method preserves warm-band structure, thermal contrast, and localized SST gradients.

\medskip \noindent
\textbf{Quantitative comparison.}

\begin{figure}[ht!]
    \centering
    \includegraphics[width=0.32\textwidth]{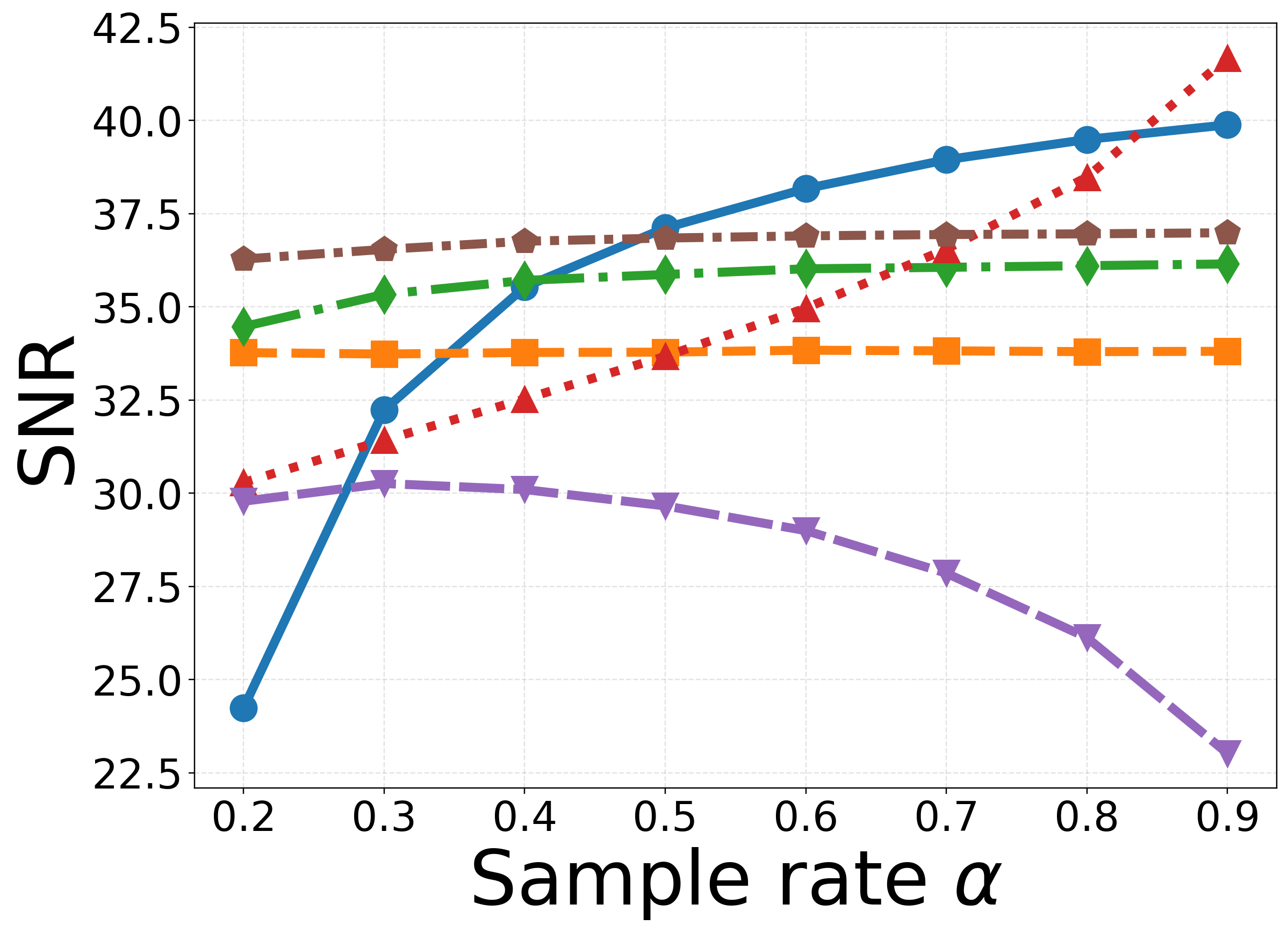}
    \includegraphics[width=0.32\textwidth]{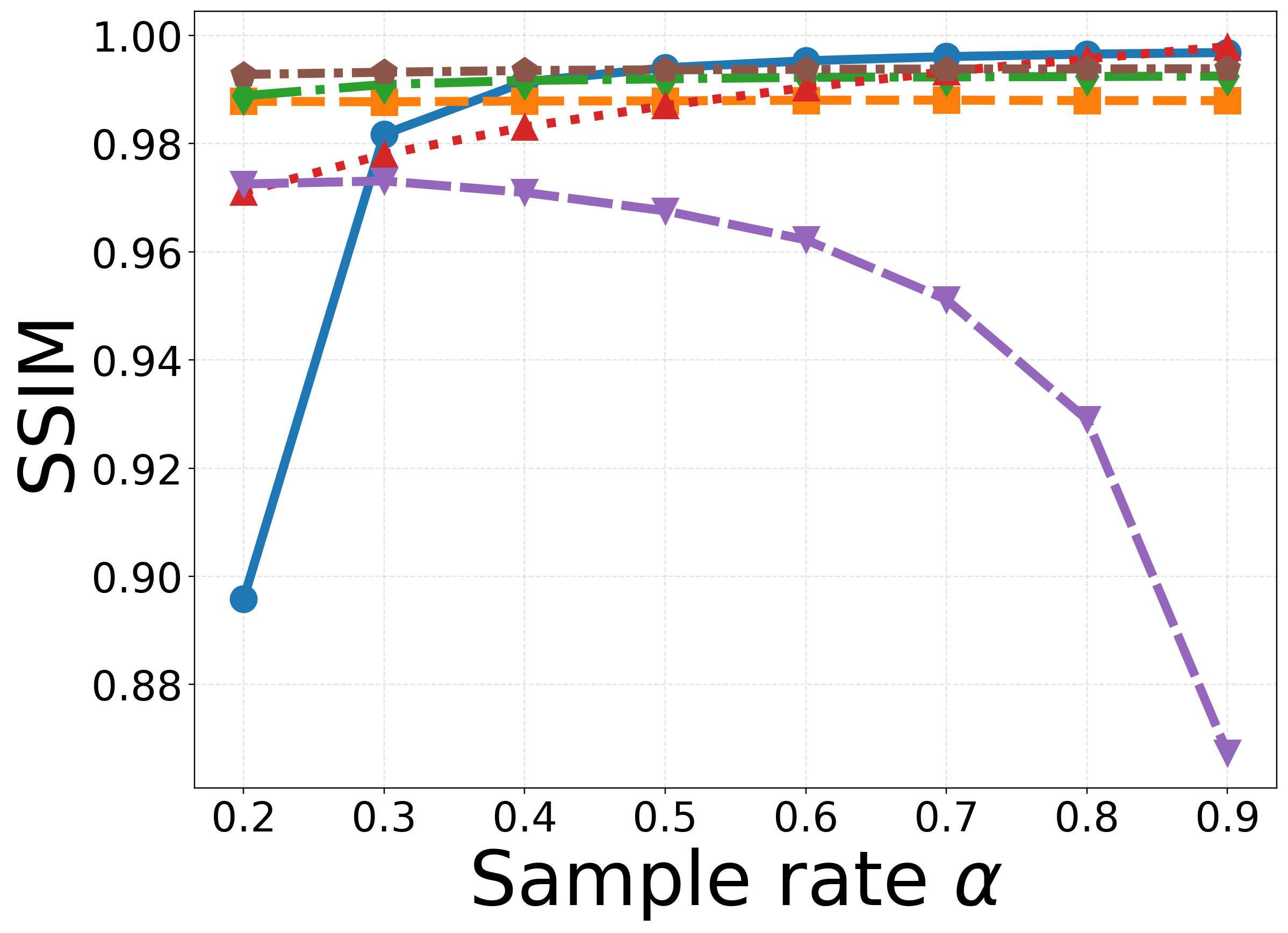}
    \includegraphics[width=0.32\textwidth]{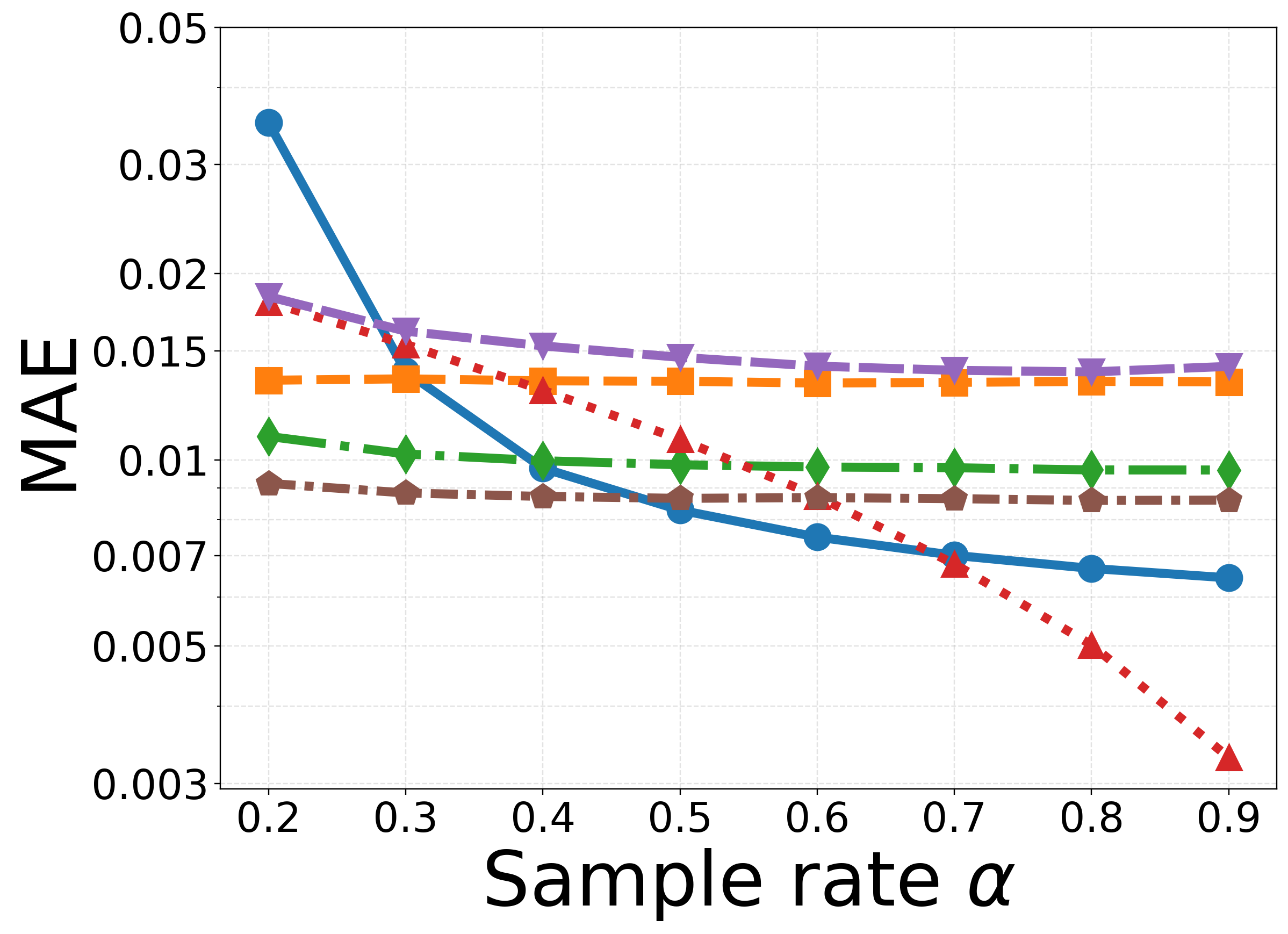}

  \vspace{-0.2in}
    \includegraphics[width=0.88\textwidth]{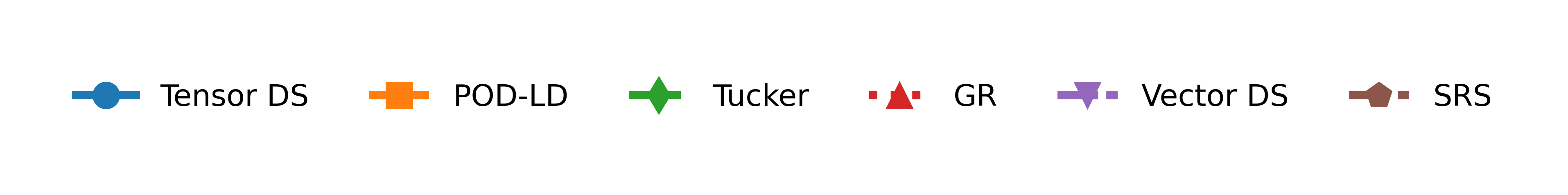}
    \caption{Initial-state metric comparison on the SST dataset. From left to right: SNR, SSIM, and MAE versus sample rate $\alpha$. Higher SNR and SSIM and lower MAE indicate better reconstruction. Tensor DS improves rapidly as $\alpha$ increases and gives strong initial-state recovery in the moderate-to-high sampling regime. SRS is also competitive on this smooth field, while the full Vector DS baseline is less effective for the SST initial state. The shared legend is shown separately below.}
    \label{fig:sst_initial_metrics}
\end{figure}

\begin{figure}[ht!]
    \centering
    \includegraphics[width=0.32\textwidth]{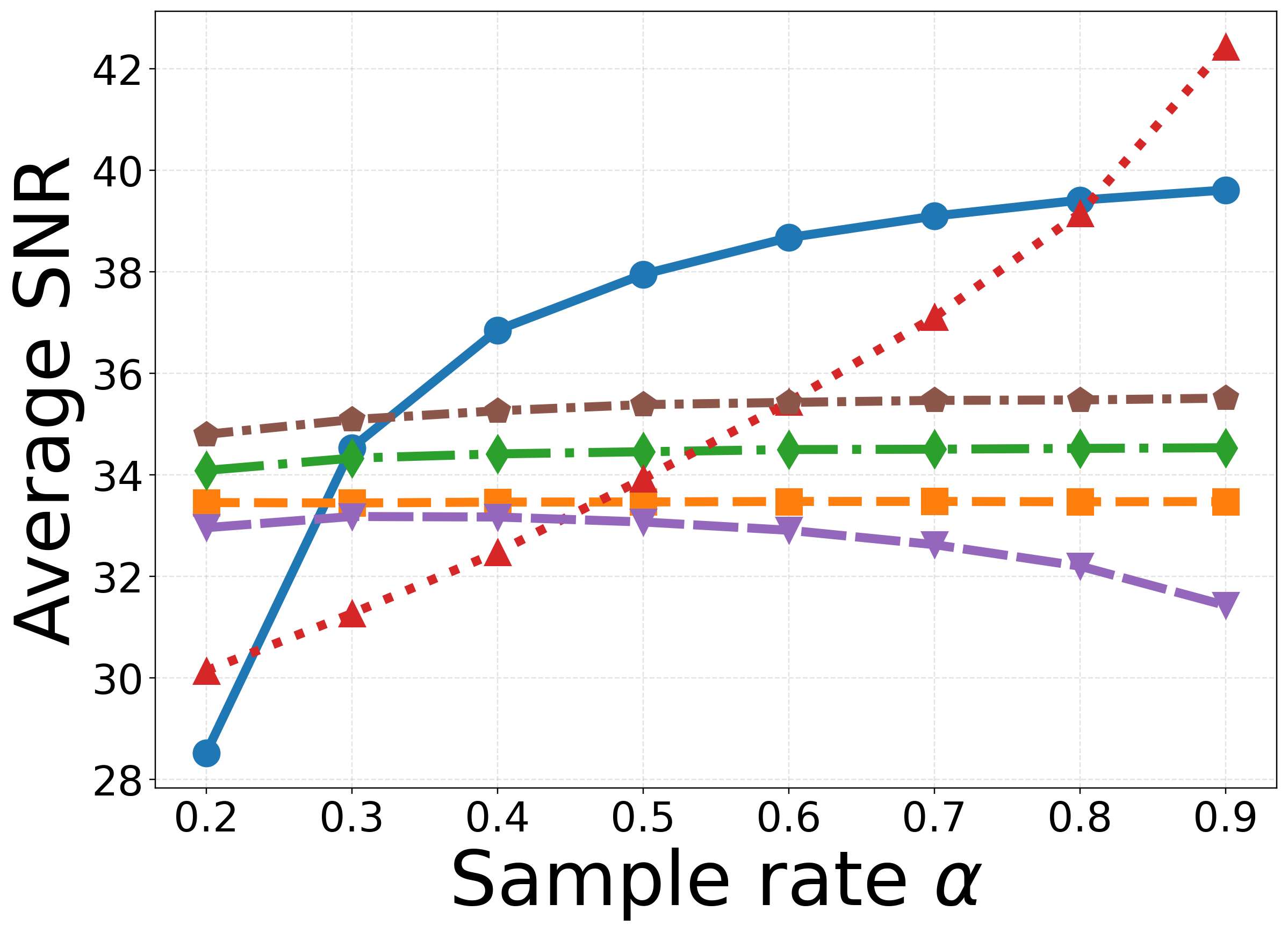}
    \includegraphics[width=0.32\textwidth]{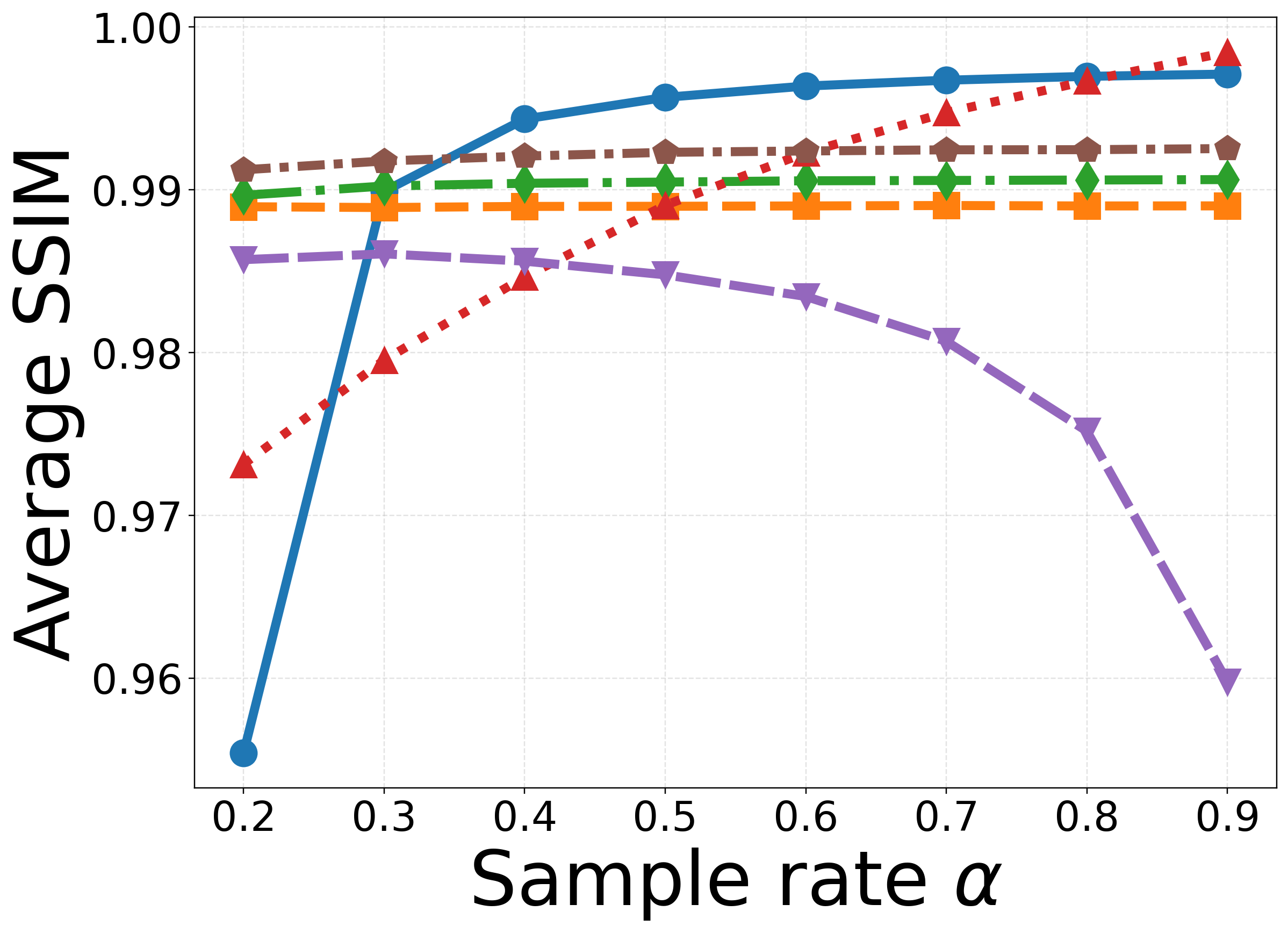}
    \includegraphics[width=0.32\textwidth]{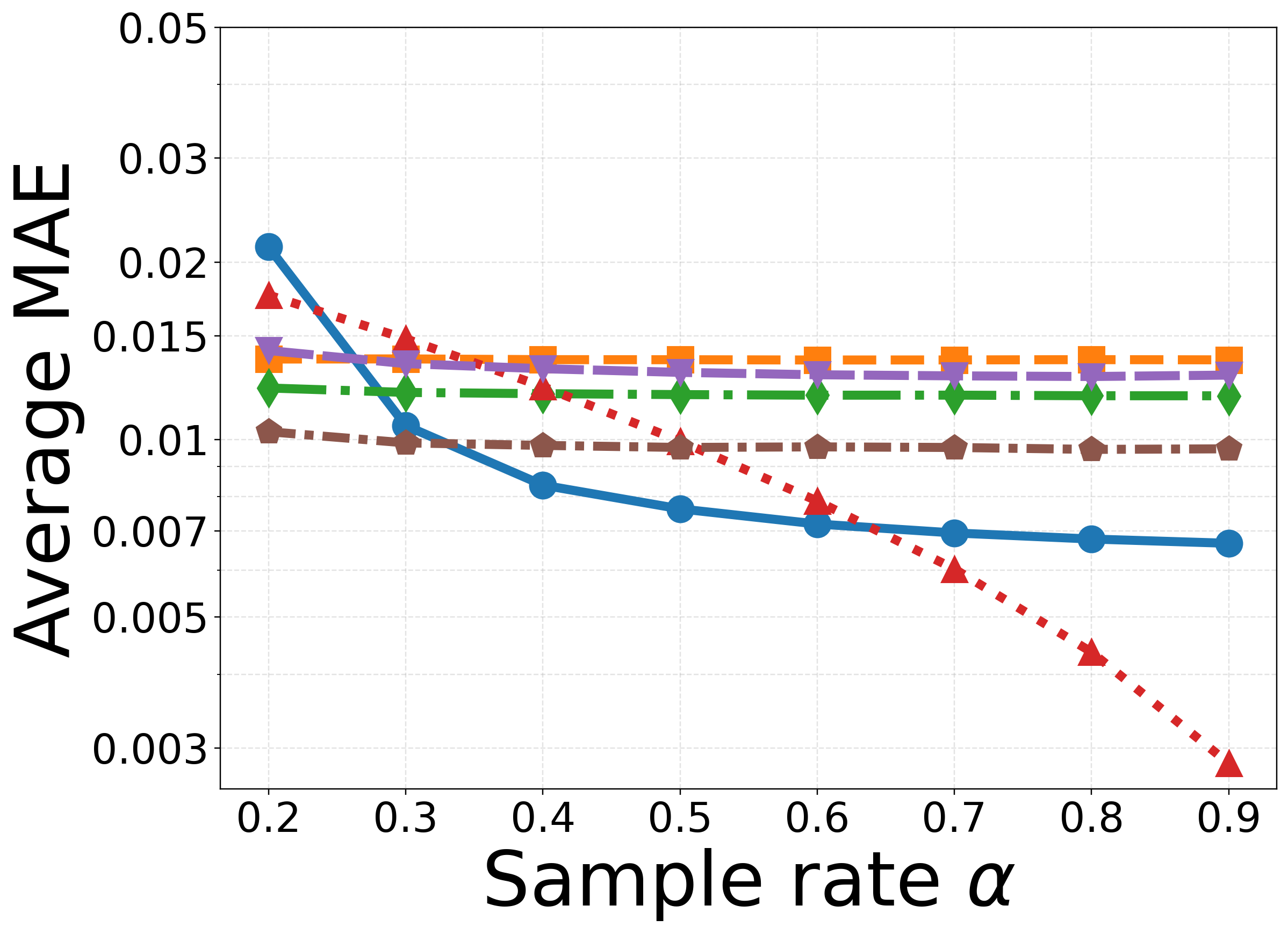}

    \vspace{-0.2in}
    \includegraphics[width=0.88\textwidth]{Figures/Fig4-5/sst_legend_only.png}

    \caption{Average reconstruction performance over the first six testing months on the SST dataset. From left to right: averaged SNR, averaged SSIM, and averaged MAE versus sample rate $\alpha$. The advantage of Tensor DS persists over the reconstruction horizon. Once enough spatial samples are available, t-product system model maintains better average reconstruction quality and lower average error than the competing baselines. The shared legend is shown separately below.}
    \label{fig:sst_6month_avg_metrics}
\end{figure}
Figure~\ref{fig:sst_initial_metrics} shows the initial-state reconstruction quality on the SST dataset. Tensor DS is not uniformly dominant at the lowest reported sampling rate, but it improves rapidly as $\alpha$ increases and becomes one of the strongest performers in SNR, SSIM, and MAE in the moderate-to-high sampling regime. This behavior is consistent with the role of spatial sampling in the recovery theory. Once the observed entries provide enough information for the relevant tensor recovery blocks, the learned tensor dynamics can be exploited effectively for initial-state reconstruction. The competitive performance of SRS reflects that the SST field is spatially smooth and has a strong low-frequency graph-spectral component.

Figure~\ref{fig:sst_6month_avg_metrics} further reports the average performance over the first six testing months. The same pattern persists over the temporal reconstruction horizon. Tensor DS maintains strong average accuracy once the sampling rate is sufficiently large, whereas the competing low-rank and graph-based baselines either saturate or oversmooth the evolving SST field. The full Vector DS baseline is less competitive on SST, which indicates that vectorization followed by a high-dimensional homogeneous linear fit does not by itself provide the same structural constraint as the t-product model. 
Thus the advantage of Tensor DS is not limited to a single initial frame, but reflects a more faithful learned dynamical representation of the spatio-temporal SST evolution.

For the fixed-rate monthly error table and the qualitative visualization below, we use the representative sampling rate $\alpha=0.6$. The quantitative curves in Figures~\ref{fig:sst_initial_metrics} and~\ref{fig:sst_6month_avg_metrics} are obtained by averaging over the random-mask trials for each reported sampling rate.

\begin{table}[ht!]
\centering
\caption{SST monthly RE at $\alpha=0.6$ with $T=6$.}
\label{tab:sst_re_t6_alpha06}
\setlength{\tabcolsep}{4pt}
\resizebox{\linewidth}{!}{
\begin{tabular}{c|cccccc}
\toprule
Step & Tensor DS & POD-LD & Tucker & GR & Vector DS & SRS \\
\midrule
\midrule
1 & $\mathbf{1.23\times10^{-2}}$ & $2.03\times10^{-2}$ & $1.58\times10^{-2}$ & $1.79\times10^{-2}$ & $3.55\times10^{-2}$ & $\underline{1.43\times10^{-2}}$ \\
2 & $\mathbf{1.07\times10^{-2}}$ & $2.09\times10^{-2}$ & $1.91\times10^{-2}$ & $\underline{1.46\times10^{-2}}$ & $1.87\times10^{-2}$ & $1.66\times10^{-2}$ \\
3 & $\mathbf{1.21\times10^{-2}}$ & $2.37\times10^{-2}$ & $2.08\times10^{-2}$ & $2.05\times10^{-2}$ & $2.02\times10^{-2}$ & $\underline{1.93\times10^{-2}}$ \\
4 & $\mathbf{1.16\times10^{-2}}$ & $2.00\times10^{-2}$ & $2.00\times10^{-2}$ & $\underline{1.53\times10^{-2}}$ & $1.95\times10^{-2}$ & $1.79\times10^{-2}$ \\
5 & $\underline{1.18\times10^{-2}}$ & $2.32\times10^{-2}$ & $1.98\times10^{-2}$ & $\mathbf{1.02\times10^{-2}}$ & $1.95\times10^{-2}$ & $1.69\times10^{-2}$ \\
6 & $\mathbf{1.19\times10^{-2}}$ & $2.79\times10^{-2}$ & $1.75\times10^{-2}$ & $2.31\times10^{-2}$ & $1.75\times10^{-2}$ & $\underline{1.43\times10^{-2}}$ \\
\bottomrule
\end{tabular}
}
\par\smallskip
\footnotesize Bold and underlined entries denote the smallest and second-smallest RE values in each row, respectively.
\end{table}

Table~\ref{tab:sst_re_t6_alpha06} reports the trial-averaged monthly relative errors at $\alpha=0.6$ over the six-month horizon. Tensor DS gives the smallest error in five of the six months and is second-best in the remaining month. SRS is often the closest competitor, reflecting the effectiveness of graph-spectral interpolation for the dominant smooth SST component. GR gives the smallest error at Month $65$ but is less consistent across the full horizon, while Vector DS is less effective for this dataset. The qualitative comparison below further shows that a small global RE does not always imply equally accurate preservation of localized thermal contrast in the highlighted ROIs.

\medskip \noindent
\textbf{Qualitative comparison.}

\begin{figure}[ht!]
    \centering
    \includegraphics[width=1\linewidth]{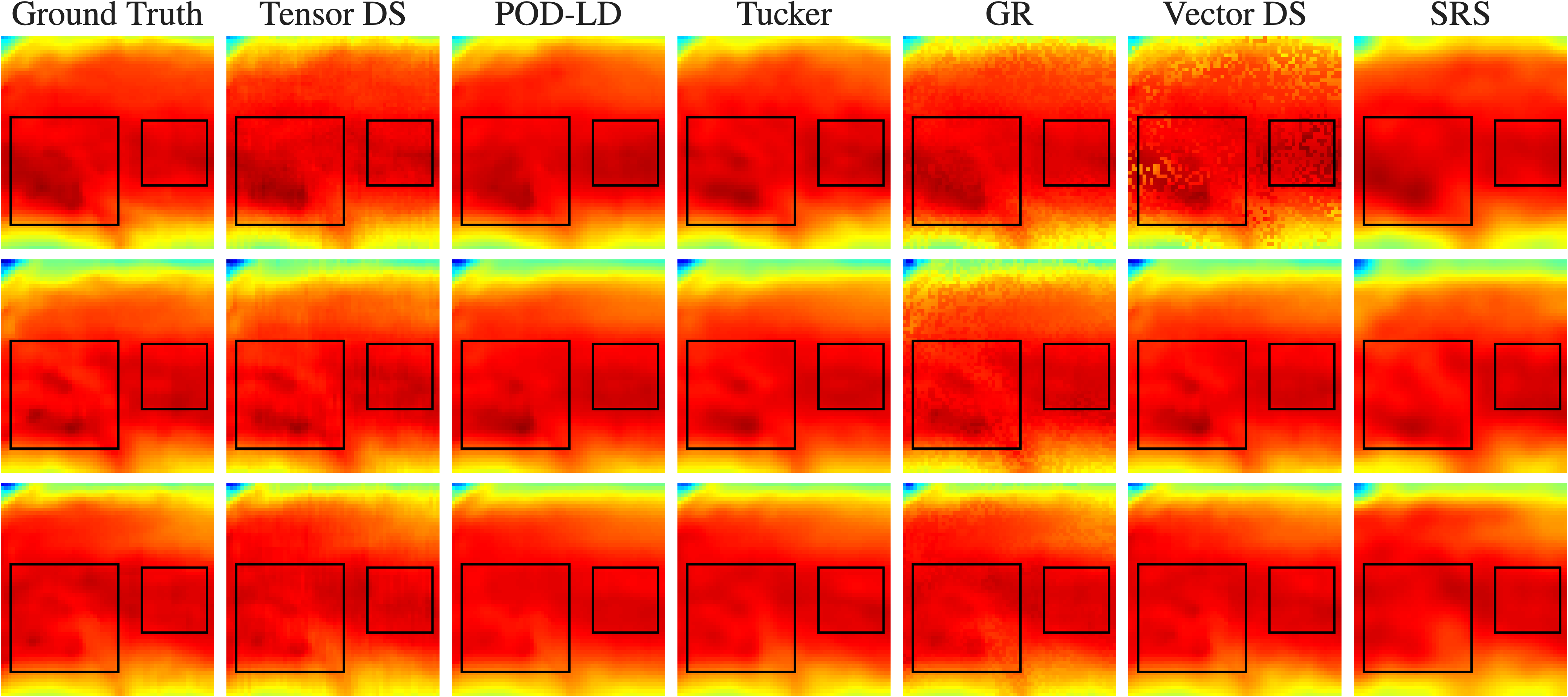}
    \caption{Visual comparison on SST data from Months 61, 63 and 65, with highlighted regions of interest. The selected domain lies in the tropical western Pacific and overlaps the Indo-Pacific warm pool. The two boxed ROIs highlight a warm-pool flank region and a heat-center-adjacent region of the western Pacific warm pool, where warm-band structure, localized SST gradients, and thermal contrast are more pronounced. Tensor DS faithfully preserves the warm-band geometry and local ROI contrast. SRS captures the large-scale smooth temperature field well, but less sharply preserves localized thermal contrast and warm-band boundary detail. POD-LD and Tucker lose some local thermal detail, while GR and Vector DS show visible local artifacts in the highlighted regions.}
    \label{fig:sst_visual_compare_box}
\end{figure}

Figure~\ref{fig:sst_visual_compare_box} presents the qualitative reconstruction results from Months 61, 63, and 65. Tensor DS most faithfully reproduces the large-scale warm-band pattern and preserves the smooth thermal transitions observed in the ground truth. POD-LD and Tucker recover the overall temperature distribution but tend to blur local details. Among the graph-based baselines, SRS captures the dominant smooth SST component well, which is consistent with its competitive RE values in Table~\ref{tab:sst_re_t6_alpha06}, but the ROI views show some loss of localized thermal contrast and warm-band boundary detail. GR shows local speckled artifacts and may distort thermal contrast in the highlighted regions. Vector DS also shows local artifacts near the boxed regions, especially at the initial frame. The boxed ROIs further highlight the local differences. Tensor DS better preserves the continuity and shape of the warm structures, especially in the warm-pool-related regions where thermal gradients are more pronounced.

\subsection{Turbulence velocity reconstruction}
\

\medskip \noindent
\textbf{Experimental setup and ROI interpretation.}
Table~\ref{tab:velocity-hyperparameters} summarizes the experimental settings used in the velocity experiments. 
The overall sampling protocol, compared methods, operator-learning procedure, and metrics are the same as in the common real-data setup above. The table records the dataset-specific dimensions and the method parameters used for this dataset.

\begin{table}[ht!]
\centering
\caption{Experimental settings for the turbulence velocity experiments.}
\label{tab:velocity-hyperparameters}
\begin{tabular}{lll}
\toprule
Category & Setting & Value \\
\midrule
\midrule
\multirow{6}{*}{Dataset} 
& Scalar field & velocity magnitude \\
& 2D slice & third-coordinate index $16$ \\
& Frame size & $32 \times 112$ \\
& Training frames & $1\text{-}60$ \\
& Dynamical-sampling / evaluation frames & $61\text{-}66$ \\
& Normalization & min--max scaling to $[0,1]$ \\
\midrule
\multirow{4}{*}{Sampling} 
& Sampling rates & $0.2,0.3,\dots,0.9$ \\
& Number of trials & $10$ \\
& Sampling mask & fixed across the observed frames \\
& Noise level $\sigma$ & $0$ \\
\midrule
\multirow{3}{*}{Tensor DS} 
& Tensor size $(m,p,n)$ & $(32,1,112)$ \\
& SVD rank cap $r_{\mathrm{tdmd}}$ & $60$ \\
& Ridge parameter & $10^{-6}$ \\
\midrule
\multirow{2}{*}{POD-LD} 
& Low-rank truncation $r_{\mathrm{mc}}$ & $10$ \\
& Ridge parameter & $10^{-6}$ \\
\midrule
\multirow{2}{*}{Tucker} 
& Tucker ranks & $(14,14)$ \\
& Ridge parameter & $10^{-6}$ \\
\midrule
\multirow{3}{*}{GR} 
& Graph type & $32 \times 112$ grid graph \\
& \multirow{2}{*}{Regularization parameters} & spatial Laplacian $10^{-2}$\\ 
&&temporal smoothing $5\times10^{-2}$ \\
\midrule
\multirow{2}{*}{Vector DS} 
& Learning ridge parameter & $10^{-4}$ \\
& Reconstruction ridge parameter & $10^{-3}$ \\
\midrule
\multirow{5}{*}{SRS}
& \multirow{2}{*}{Spectral cutoff rule} & energy ratio $0.995$\\
&& max $300$, min $100$ modes \\
& Effective spectral cutoff & $100$ modes \\
& Reconstruction parameter & selected from training data \\
& Selected value & $10^{-5}$ \\
\bottomrule
\end{tabular}
\end{table}

We also evaluate the methods on a velocity dataset extracted from the JHTDB forced isotropic turbulence dataset \cite{li2008public,perlman2007data}. Each raw frame contains three velocity components on a three-dimensional spatial grid. We convert the velocity vector to its magnitude and take the two-dimensional slice at the third-coordinate index $16$. This index refers to the original three-dimensional turbulence field, not to the third mode of the tensor representation. The resulting sequence consists of $32\times112$ velocity-magnitude frames. The first $60$ frames are used for training, and Frames $61$--$66$ form the dynamical-sampling and evaluation horizon.

For Tensor DS, each velocity-magnitude frame is represented as a third-order tensor of size $32\times1\times112$. The t-product system operator is learned from the training frames using the same slice-wise truncated DMD-type procedure described in the common setup. For each random trial, a fixed spatial mask is applied to all six observed frames. Tensor DS uses these masked multi-time observations to recover the initial testing state at Frame $61$, and the learned t-product system operator is then applied recursively to reconstruct the remaining frames in the same horizon.

The highlighted ROIs are selected from local patches where the velocity-magnitude field exhibits stronger instantaneous spatial variation and richer fine-scale texture. In homogeneous isotropic turbulence, such localized complexity is commonly associated with intermittency and intense vortical structures that appear as filamentary or worm-like regions in the flow \cite{she1990intermittent,jimenez1993structure,ghira2022characteristics,vela2021synchronisation}. These ROIs therefore provide informative local windows for comparing reconstruction quality.
In contrast with the SST experiment, which mainly tests recovery of a smooth geophysical field with coherent thermal bands, the velocity experiment tests whether the same tensor dynamical model can preserve sharper transitions and localized texture in a more irregular turbulent field.

\medskip \noindent
\textbf{Quantitative comparison.}

\begin{figure*}[ht!]
\centering
\includegraphics[width=0.32\textwidth]{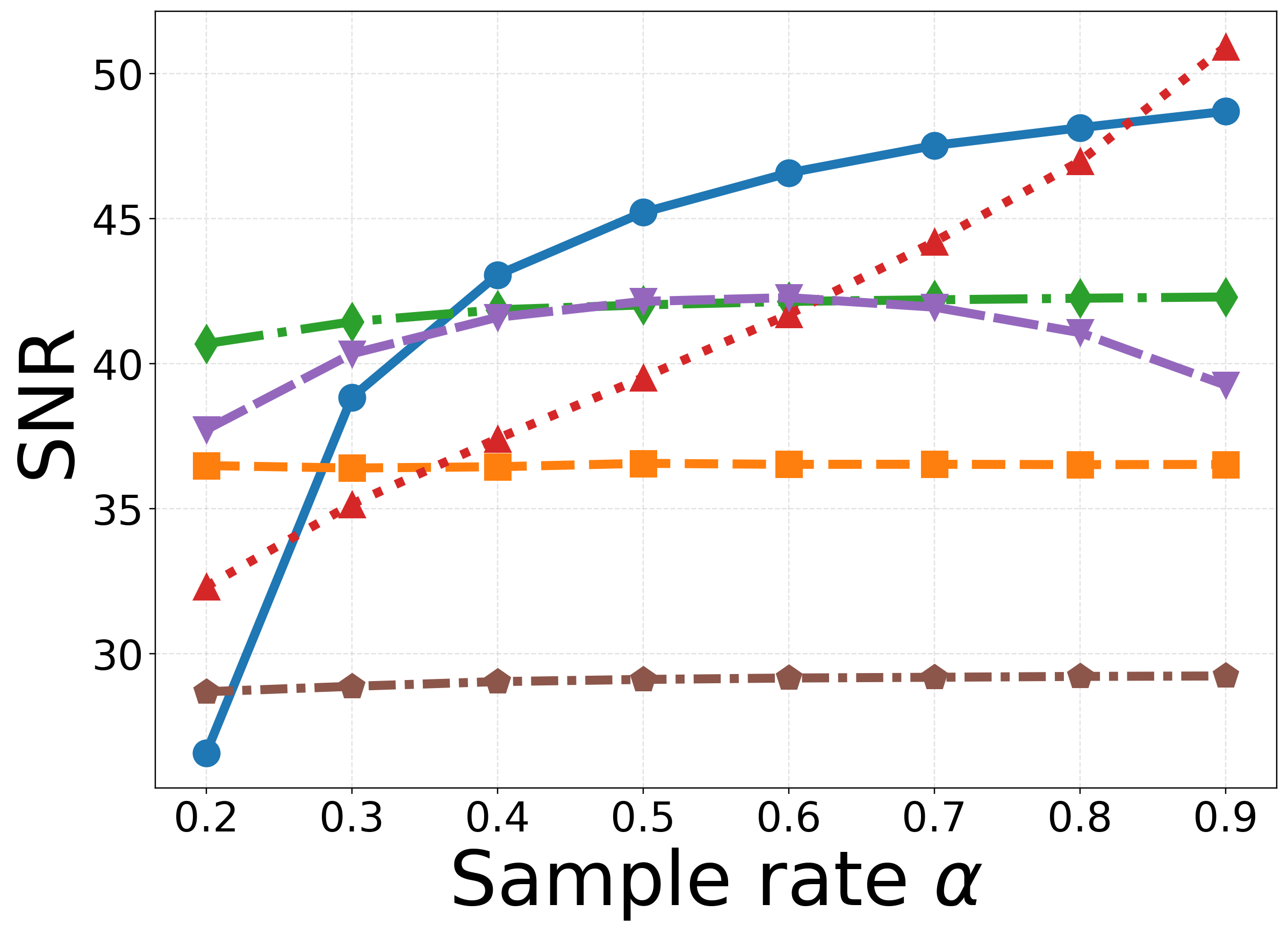}
\includegraphics[width=0.32\textwidth]{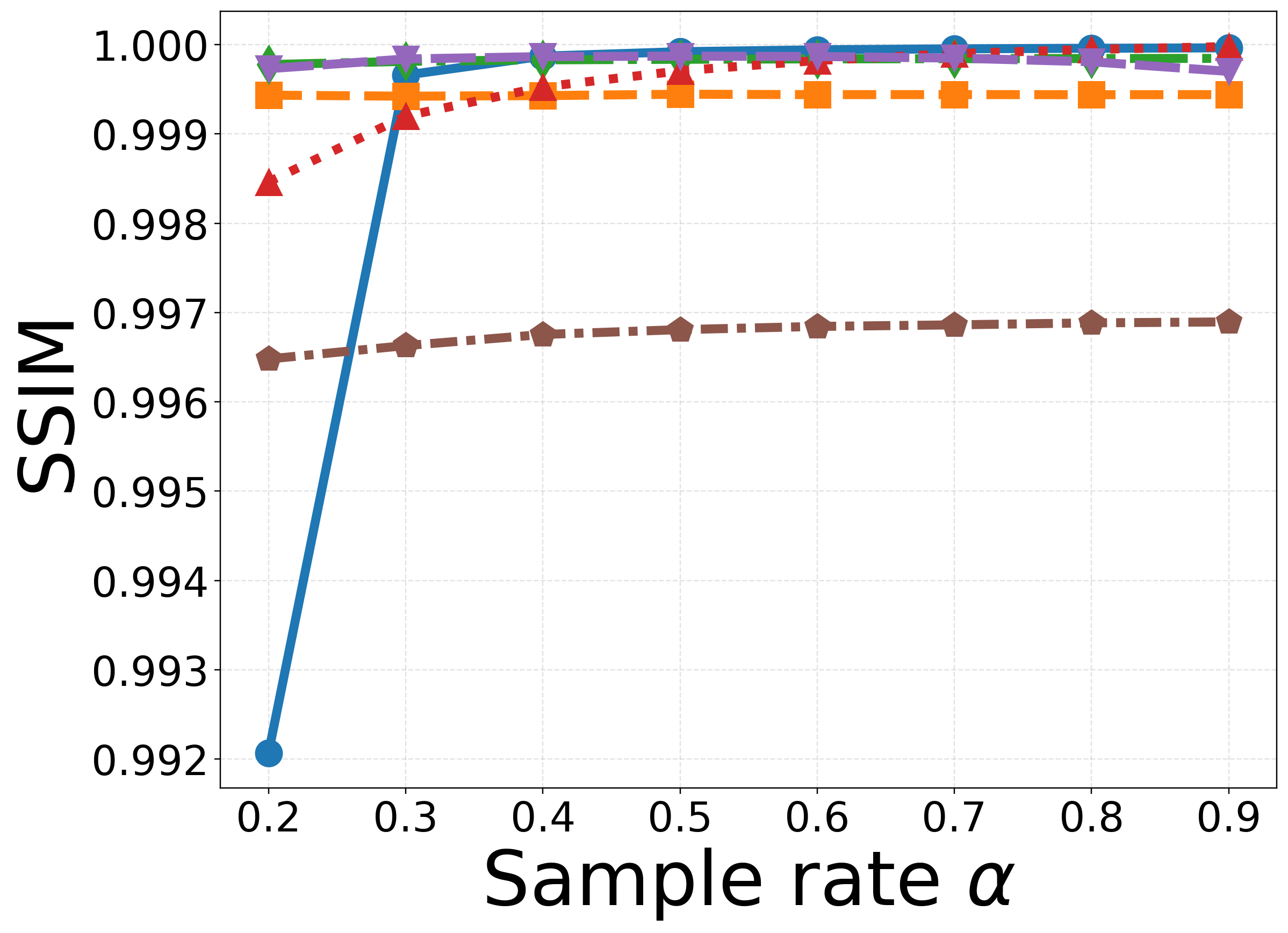}
\includegraphics[width=0.32\textwidth]{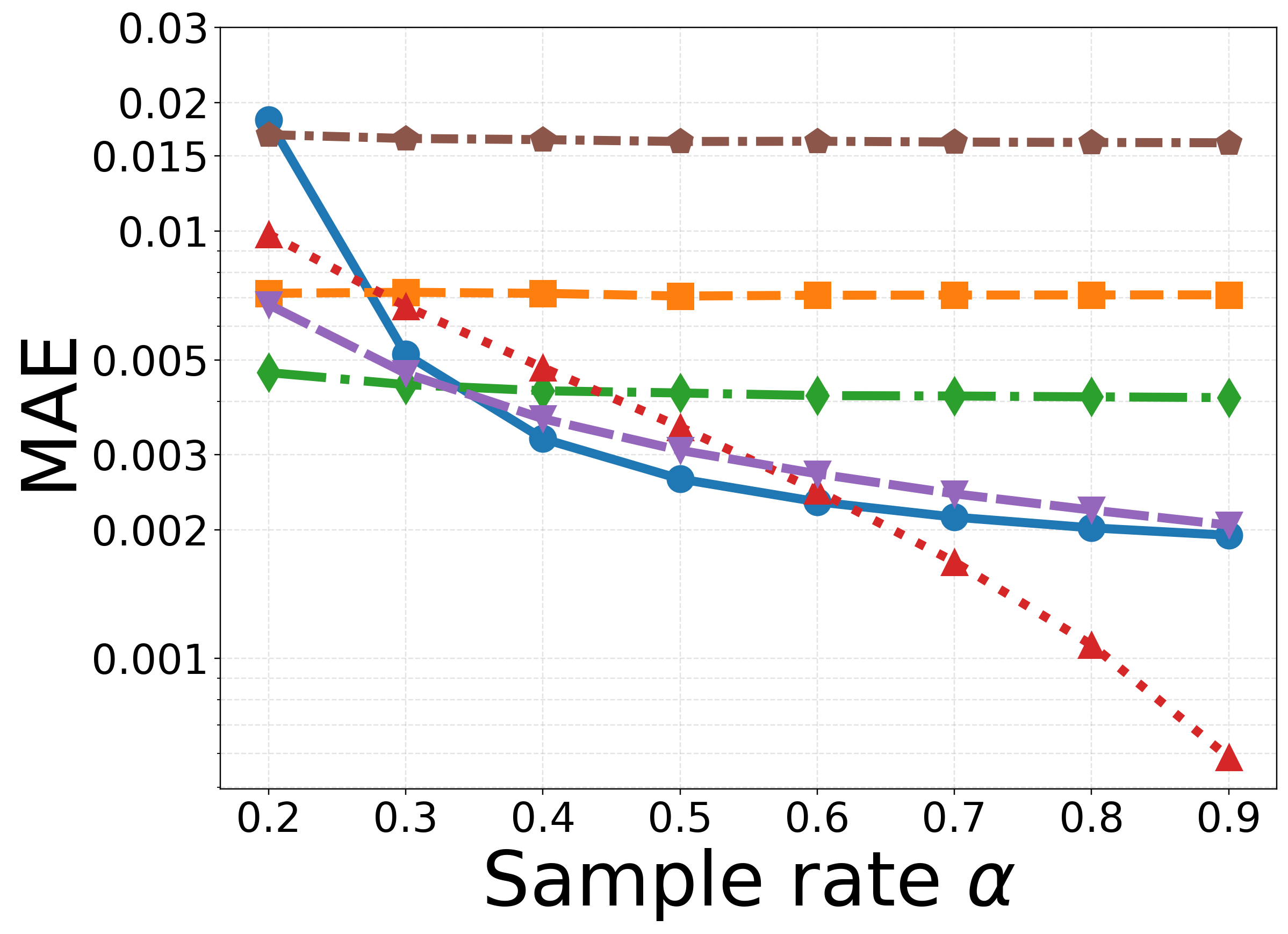}

  \vspace{-0.2in}
\includegraphics[width=0.88\textwidth]{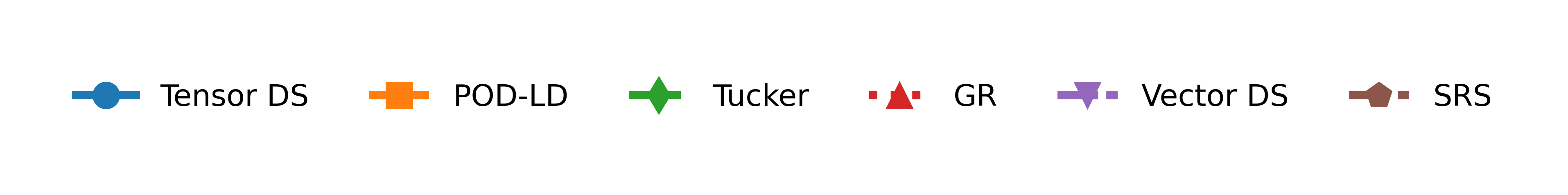}
\caption{Initial-state metric comparison on the velocity dataset. From left to right: SNR, SSIM, and MAE versus sample rate $\alpha$. The MAE panel is displayed on a logarithmic scale. Tensor DS gives the strongest overall initial-state recovery across the reported sampling range, showing that the t-product system model remains effective for a velocity field with sharper local variations than the SST data.}
\label{fig:velocity_initial_metrics}
\end{figure*}

\begin{figure*}[ht!]
\centering
\includegraphics[width=0.32\textwidth]{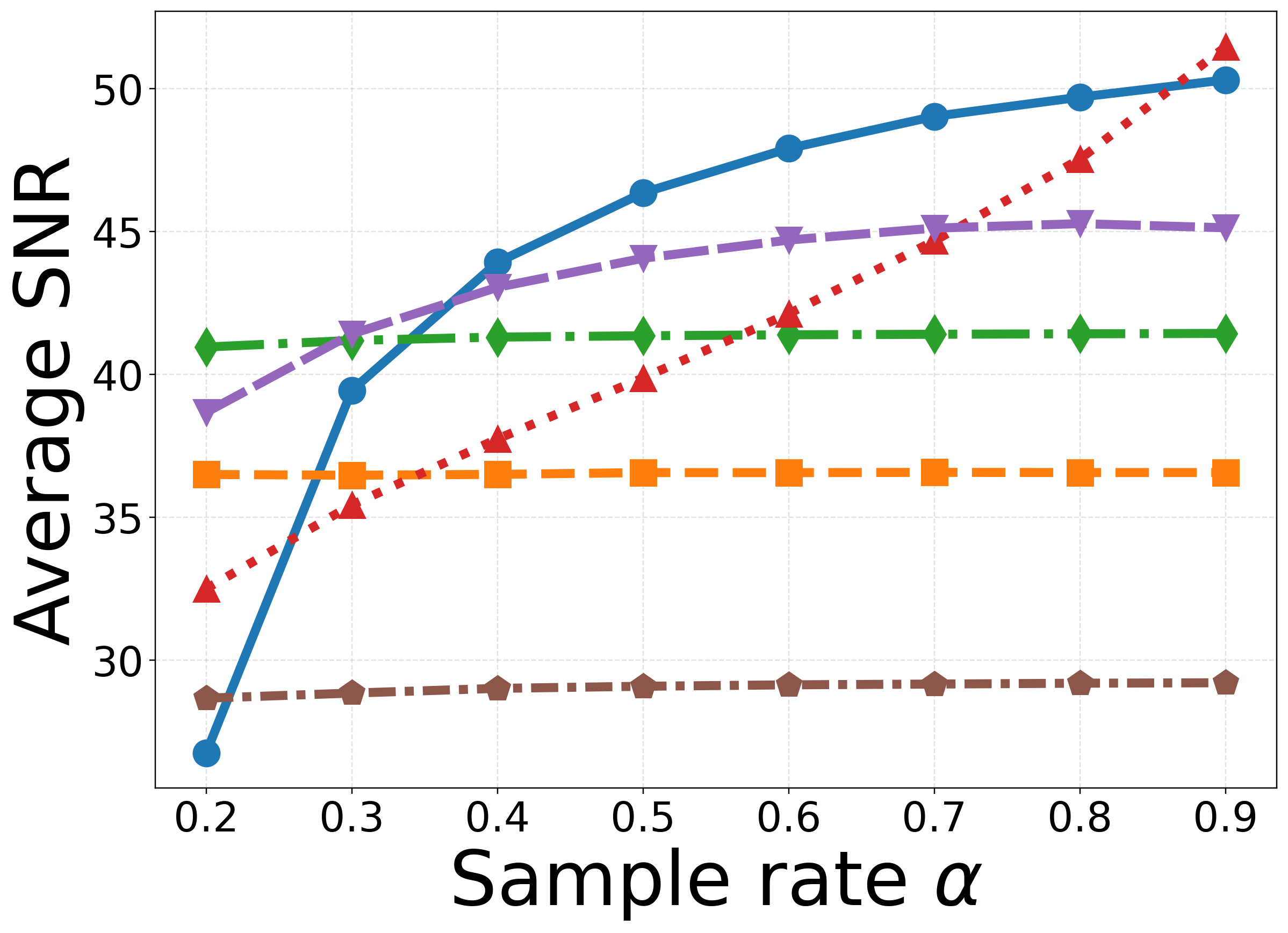}
\includegraphics[width=0.32\textwidth]{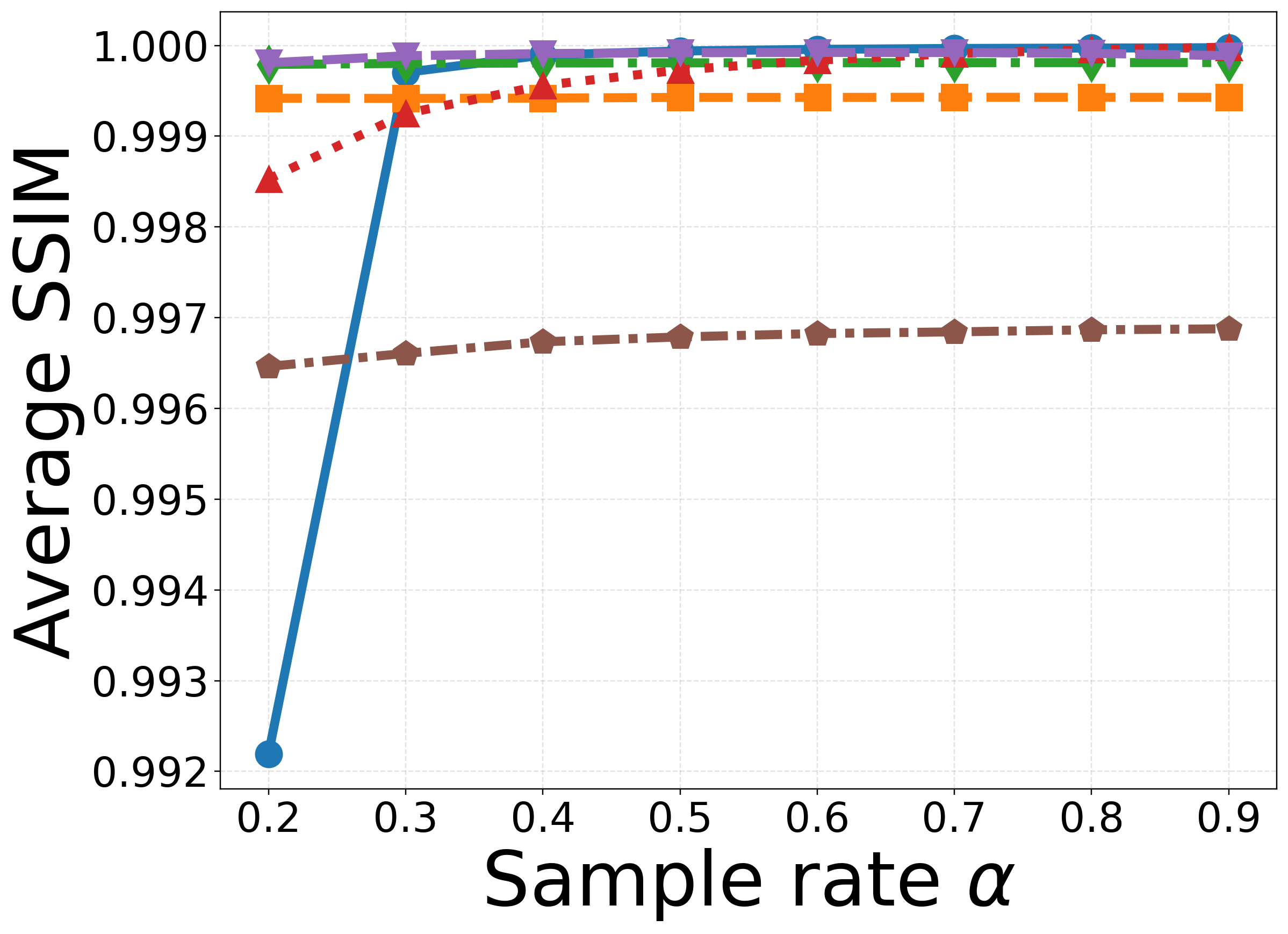}
\includegraphics[width=0.32\textwidth]{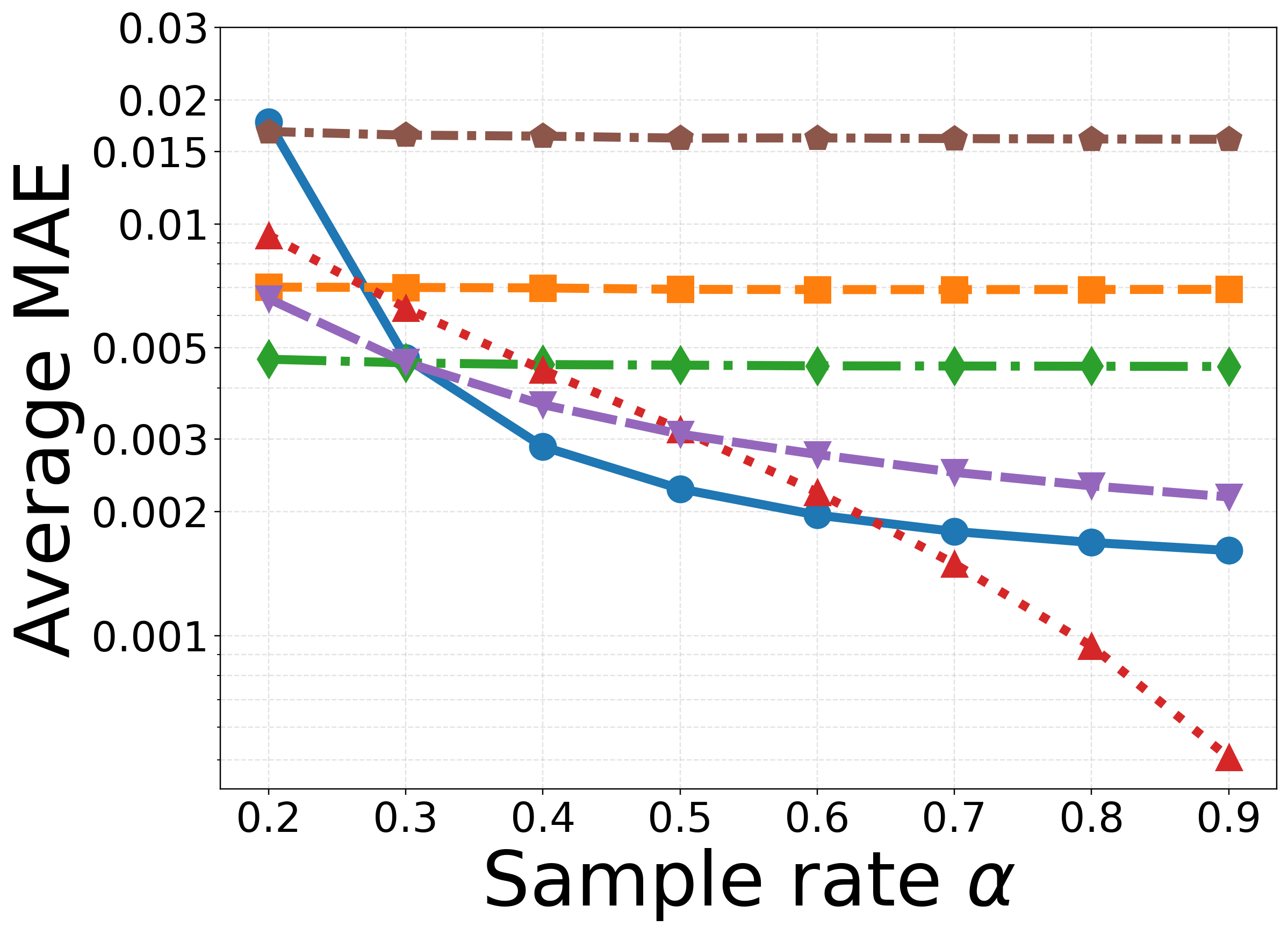}

  \vspace{-0.2in}
\includegraphics[width=0.88\textwidth]{Figures/Fig7-8/velocity_legend_only.png}
\caption{Average evolving-state metric comparison on the velocity dataset over the prediction horizon. From left to right: averaged SNR, averaged SSIM, and averaged MAE versus sample rate $\alpha$. The MAE panel is displayed on a logarithmic scale. Tensor DS maintains its advantage over the reconstruction horizon, indicating that the learned tensor dynamics preserve temporal structure more effectively than the vectorized low-rank, Tucker-type, and graph-based baselines.}
\label{fig:velocity_evolve_metrics}
\end{figure*}

Figure~\ref{fig:velocity_initial_metrics} presents the initial-state reconstruction performance on the velocity dataset. Tensor DS gives the strongest overall performance across the reported sampling range, producing higher SNR and SSIM together with lower MAE than the competing methods. This is notable because the velocity field is less smooth than the SST field and contains more localized texture. The result suggests that the gain of Tensor DS is not limited to smooth geophysical data.

Figure~\ref{fig:velocity_evolve_metrics} reports the average performance over the six-frame reconstruction horizon. Tensor DS remains the most stable and accurate method throughout the multi-step reconstruction. The improvement over POD-LD and Tucker indicates that preserving the t-product tensor organization is more effective than learning dynamics only after matrix or Tucker-core reduction. The improvement over GR, Vector DS, and SRS further shows that graph-spectral structure alone does not fully capture the localized turbulent patterns in this experiment.

For the fixed-rate framewise error table and the qualitative visualization below, we use the representative sampling rate $\alpha=0.6$. The quantitative curves in Figures~\ref{fig:velocity_initial_metrics} and~\ref{fig:velocity_evolve_metrics} are obtained by averaging over the random-mask trials for each reported sampling rate.

\begin{table}[ht!]
\centering
\caption{Velocity framewise RE at $\alpha=0.6$ with $T=6$.}
\label{tab:velocity_re_t6_alpha06}
\setlength{\tabcolsep}{4pt}
\resizebox{\linewidth}{!}{
\begin{tabular}{c|cccccc}
\toprule
Step & Tensor DS & POD-LD & Tucker & GR & Vector DS & SRS \\
\midrule
\midrule
1 & $\mathbf{4.69\times10^{-3}}$ & $1.49\times10^{-2}$ & $7.81\times10^{-3}$ & $8.24\times10^{-3}$ & $\underline{7.69\times10^{-3}}$ & $3.48\times10^{-2}$ \\
2 & $\mathbf{4.04\times10^{-3}}$ & $1.37\times10^{-2}$ & $8.27\times10^{-3}$ & $7.77\times10^{-3}$ & $\underline{4.36\times10^{-3}}$ & $3.49\times10^{-2}$ \\
3 & $\mathbf{3.79\times10^{-3}}$ & $1.40\times10^{-2}$ & $8.79\times10^{-3}$ & $7.67\times10^{-3}$ & $\underline{4.97\times10^{-3}}$ & $3.49\times10^{-2}$ \\
4 & $\mathbf{3.66\times10^{-3}}$ & $1.69\times10^{-2}$ & $9.28\times10^{-3}$ & $7.71\times10^{-3}$ & $\underline{6.87\times10^{-3}}$ & $3.50\times10^{-2}$ \\
5 & $\mathbf{4.04\times10^{-3}}$ & $2.22\times10^{-2}$ & $9.92\times10^{-3}$ & $\underline{7.90\times10^{-3}}$ & $9.81\times10^{-3}$ & $3.51\times10^{-2}$ \\
6 & $\mathbf{5.82\times10^{-3}}$ & $2.93\times10^{-2}$ & $1.10\times10^{-2}$ & $\underline{8.48\times10^{-3}}$ & $1.34\times10^{-2}$ & $3.51\times10^{-2}$ \\
\bottomrule
\end{tabular}
}
\par\smallskip
\footnotesize Bold and underlined entries denote the smallest and second-smallest RE values in each row, respectively.
\end{table}

Table~\ref{tab:velocity_re_t6_alpha06} reports the framewise relative errors at $\alpha=0.6$ over the six-frame horizon. Tensor DS yields the smallest error at every step. 
Vector DS is close to second-best in the early frames, while GR is typically second-best in the later frames. Nevertheless, both remain consistently less accurate than Tensor DS. POD-LD and Tucker capture part of the large-scale structure but are less accurate over the reconstruction horizon. SRS is less effective for this turbulent velocity field.

\medskip \noindent
\textbf{Qualitative comparison.}

\begin{figure}[ht!]
    \centering
    \includegraphics[width=1\linewidth]{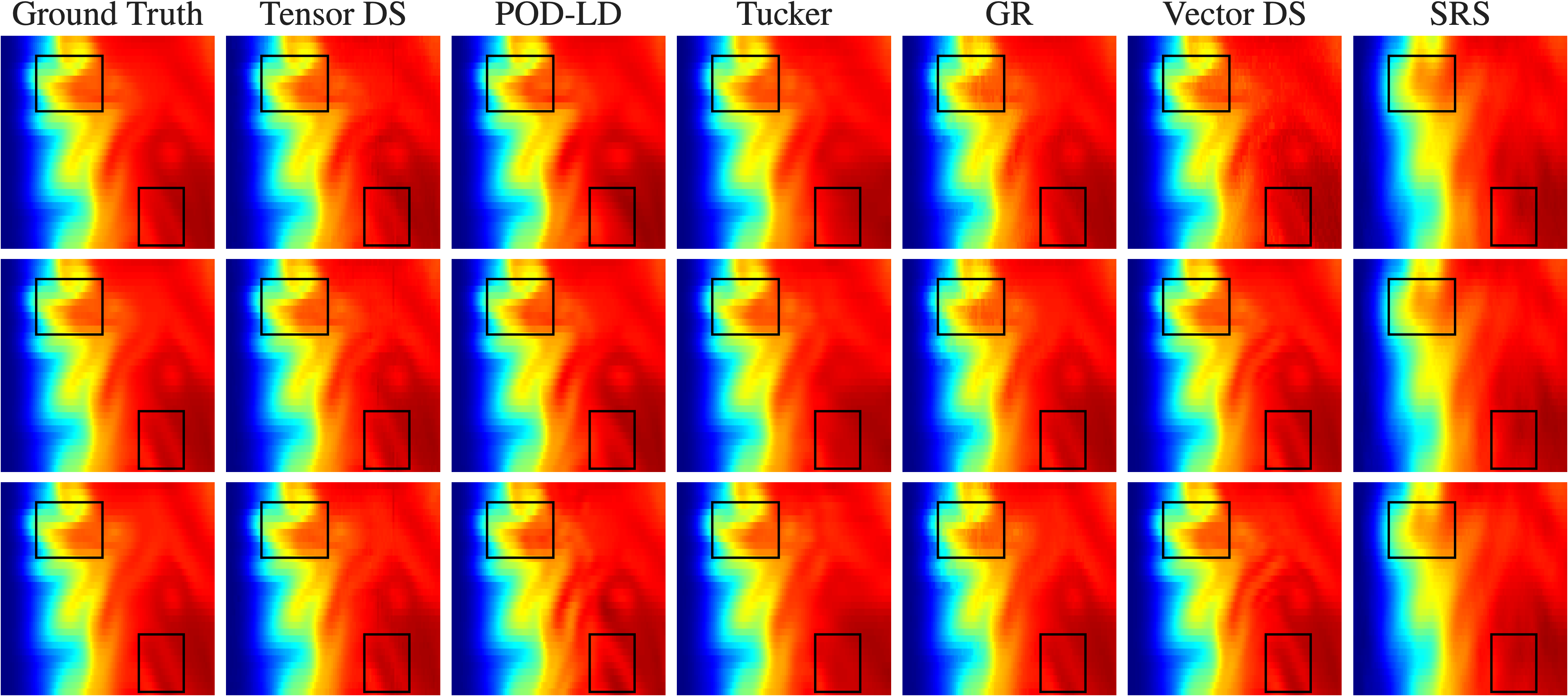}
    \caption{Visual comparison on velocity data from Frames 61, 63 and 65, with highlighted regions of interest. Tensor DS remains visually closest to the ground truth throughout the reconstruction horizon and preserves both the large-scale velocity-magnitude pattern and finer localized texture. POD-LD, Tucker, and GR recover much of the global transition pattern but tend to smooth or weaken small-scale variations, while Vector DS loses some localized texture and SRS is more strongly oversmoothed.}
    \label{fig:velocity_visual_compare_box}
\end{figure}

Figure~\ref{fig:velocity_visual_compare_box} presents qualitative comparisons over Frames 61, 63, and 65. Tensor DS remains visually closest to the ground truth throughout the reconstruction horizon. It preserves the main large-scale transition pattern together with finer local texture inside the highlighted ROIs. POD-LD, Tucker, and GR recover much of the global structure but tend to smooth or weaken small-scale variations. Vector DS is competitive in the early framewise RE table, but the visual comparison shows that it loses part of the localized texture in the highlighted regions. SRS is more strongly oversmoothed and does not preserve the localized texture inside the highlighted regions. These visual observations are consistent with the quantitative metrics and show that Tensor DS better preserves both global flow organization and localized turbulent features.

\subsection{Summary of numerical findings}

The synthetic experiments confirm the two main theory-facing predictions of the paper. The phase transition diagrams support the blockwise sampling scale governed by the aggregated dimensions $d_\lambda$, while the temporal stability experiment illustrates that identifiability and numerical stability are distinct once the Vandermonde factor is inverted. The real-data experiments then show that the same tensor dynamical modeling principle remains useful beyond the exact diagonalizable setting used in the theory.

On both SST and turbulence velocity data, Tensor DS gives stable and structurally faithful reconstruction performance among the tested methods. The improvement is not only quantitative, as measured by SNR, SSIM, MAE, and RE, but also structural. Tensor DS better preserves coherent warm-band geometry in the SST experiment and localized texture in the velocity experiment. 
The experiments also show that the relative advantage depends on the sampling level and the smoothness of the dataset. At very low sampling rates, the empirical observation system can be poorly conditioned, so strongly regularized low-rank or graph-spectral baselines may be competitive. On the smooth SST field, the static graph-spectral SRS baseline is competitive because the dominant SST component is well aligned with a low-frequency graph representation, although the ROI comparison still reveals loss of localized thermal contrast. On the more textured velocity field, Tensor DS separates more clearly from graph-spectral smoothing baselines. The full Vector DS baseline is useful for comparison because it tests the effect of vectorizing the field and learning a high-dimensional homogeneous dynamics, and its weaker SST behavior highlights the importance of preserving tensor-valued structure rather than relying only on full vectorization. This is consistent with the role of the empirical Gram and observation matrices in the theory.

\section{Conclusion}

We developed a tensor dynamical sampling framework for third-order tensor-valued systems evolving under the t-product. The main point is that the Fourier-domain slicewise representation does not by itself describe the recovery geometry. Shared eigenvalues across frontal slices create aggregated eigenspace blocks, and these blocks determine the correct deterministic frame conditions, the spatial sampling burden, and the randomized recovery scale.

The randomized theory shows that the sampling rate is governed by intrinsic quantities of these aggregated blocks, including their dimensions, incoherence, and population conditioning. The stability analysis further shows that temporal depth has a dual role. It is needed for temporal separation, but it may also worsen numerical conditioning through the Vandermonde factor when the spectral magnitudes are unbalanced.

The numerical experiments support these conclusions. Synthetic phase transitions track the predicted blockwise sampling thresholds, while the SST and turbulence experiments show that the learned t-product system operator can improve reconstruction in practical spatio-temporal datasets. They also indicate that tensor dynamics are most effective once the sampling level is high enough to produce an informative recovery system, and that carefully tuned graph-spectral baselines can be strong on very smooth data. Future work includes extending the theory to more general tensor products, nonstationary or time-varying tensor dynamics, and adaptive sampling strategies that exploit the aggregated spectral geometry directly.

\section{Proofs of the main results}

We now give the proofs of the main deterministic, randomized, and stability results. The order follows the development of the paper: we first prove the structural factorization underlying the deterministic theory, then the temporal decoupling and randomized Gram estimates, and finally the temporal conditioning statements.

\subsection{Proof of Theorem \ref{thm:stacked}}

We begin with an auxiliary factorization.
\begin{proposition}\label{prop:decomp}
For each fixed $j\in[p]$, the stacked matrix $M_j(T)$ admits 
\begin{equation*}
    M_j(T)Q = \rb{\mathsf{V}_T(\Lambda)\otimes I_{mn}}
\textrm{blkdiag}\rb{B_{j,\lambda_1},\ldots,B_{j,\lambda_R}},
\end{equation*}
where $Q=WP$ is defined after \eqref{eq:reordering1}, the blocks $B_{j,\lambda}$ are defined in \eqref{eq:Bblock}, and $\mathsf{V}_T(\Lambda)$ is the Vandermonde matrix
given by
\begin{equation*}
    \mathsf{V}_T(\Lambda) := \qm{ 1 &\cdots & 1\\
\lambda_1 & \cdots & \lambda_R\\
\vdots & & \vdots\\
\lambda_1^{T-1} & \cdots & \lambda_R^{T-1} } \in \C^{T\times R}.
\end{equation*}
\end{proposition}

\begin{proof}
Recall that $M_j(T)$ is the stacked matrix in \eqref{eq:prelim-Mj}, while $A_1(t)$ and $A_3(j)$ are defined in \eqref{eq:prelim-A1} and \eqref{eq:prelim-A3}.
Note that
\begin{equation}\label{pf:M1}
    M_j(T)Q = \qm{
A_3(j)A_1(0)Q\\
A_3(j)A_1(1)Q\\
\vdots\\
A_3(j)A_1(T-1)Q}.
\end{equation}
Using \eqref{eq:decomp}, \eqref{eq:reordering1}, and $Q=WP$, we have
\begin{equation*}
    A_3(j)A_1(t)Q = A_3(j)(W\Delta_t W^{-1})(WP)=A_3(j)W\Delta_t P = A_3(j)Q\, \textrm{blkdiag}\rb{\lambda_1^t I_{d_{\lambda_1}},\ldots,\lambda_R^t I_{d_{\lambda_R}}}.
\end{equation*}
Now partitioning $A_3(j)Q$ as in \eqref{eq:Bblock}, we obtain
\begin{align*}
    A_3(j)A_1(t)Q &= \qm{ B_{j,\lambda_1} & \cdots & B_{j,\lambda_R} }  \textrm{blkdiag}\rb{\lambda_1^t I_{d_{\lambda_1}},\ldots,\lambda_R^t I_{d_{\lambda_R}}} = \qm{ \lambda_1^t B_{j,\lambda_1} & \cdots & \lambda_R^t B_{j,\lambda_R} }.
\end{align*}

Substituting this identity into \eqref{pf:M1} gives
\begin{equation*}
M_j(T)Q = \qm{
B_{j,\lambda_1} & \cdots & B_{j,\lambda_R}\\
\lambda_1 B_{j,\lambda_1} & \cdots & \lambda_R B_{j,\lambda_R}\\
\vdots & & \vdots\\
\lambda_1^{T-1} B_{j,\lambda_1} & \cdots & \lambda_R^{T-1} B_{j,\lambda_R}}.
\end{equation*}
By the definition of the Kronecker product, the desired result follows immediately.
\end{proof}

\begin{proof}[Proof of Theorem~\ref{thm:stacked}]
By Proposition~\ref{prop:decomp}, we have
\begin{equation}\label{pf:decomp2}
    M_j(T)Q =
    \bigl(\mathsf{V}_T(\Lambda)\otimes I_{mn}\bigr)
    \operatorname{blkdiag}\!\left(
    B_{j,\lambda_1},\ldots,B_{j,\lambda_R}
    \right).
\end{equation}
Since $Q$ is invertible, right multiplication by $Q$ preserves column rank. Hence $M_j(T)$ has full column rank if and only if $M_j(T)Q$ has full column rank.

We first prove that $(i)\Rightarrow (ii)$. Suppose that $M_j(T)$ has full column rank. 
We claim that each block $B_{j,\lambda_\ell}$ has full column rank. Suppose, to the contrary, that for some $\ell\in\{1,\ldots,R\}$, there exists a nonzero vector $z_\ell$ such that $B_{j,\lambda_\ell}z_\ell=0$.
Define
\begin{equation*}
    z=
    \begin{bmatrix}
    0 & \cdots & 0 & z_\ell & 0 & \cdots & 0
    \end{bmatrix}^{\top},
\end{equation*}
where $z_\ell$ occupies the $\ell$-th block. Then $z\neq 0$ and
\begin{equation*}
    \operatorname{blkdiag}\!\left(
    B_{j,\lambda_1},\ldots,B_{j,\lambda_R}
    \right)z=0.
\end{equation*}
Using \eqref{pf:decomp2}, we obtain
\begin{equation*}
    M_j(T)Qz=0.
\end{equation*}
This contradicts the full column rank of $M_j(T)Q$. Therefore, each $B_{j,\lambda_\ell}$ has full column rank.

Next, we show that
\begin{equation*}
    \ker\bigl(\mathsf{V}_T(\Lambda)\otimes I_{mn}\bigr) \cap
    \operatorname{range}\!\left(
    \operatorname{blkdiag}\!\left(
    B_{j,\lambda_1},\ldots,B_{j,\lambda_R}
    \right) \right) = \{0\}.
\end{equation*}
Let
\begin{equation*}
    g \in \ker\bigl(\mathsf{V}_T(\Lambda)\otimes I_{mn}\bigr)
    \cap \operatorname{range}\!\left(
    \operatorname{blkdiag}\!\left(
    B_{j,\lambda_1},\ldots,B_{j,\lambda_R}
    \right)
    \right).
\end{equation*}
Since $g$ belongs to the range of the block diagonal matrix, there exists a vector $z$ such that
\begin{equation*}
    g= \operatorname{blkdiag}\!\left(
    B_{j,\lambda_1},\ldots,B_{j,\lambda_R}
    \right)z.
\end{equation*}
Since $g\in \ker(\mathsf{V}_T(\Lambda)\otimes I_{mn})$, we have
\begin{equation*}
    \bigl(\mathsf{V}_T(\Lambda)\otimes I_{mn}\bigr)g=0.
\end{equation*}
Therefore,
\begin{equation*}
    \bigl(\mathsf{V}_T(\Lambda)\otimes I_{mn}\bigr)
    \operatorname{blkdiag}\!\left(
    B_{j,\lambda_1},\ldots,B_{j,\lambda_R}
    \right)z=0.
\end{equation*}
By \eqref{pf:decomp2}, this gives $M_j(T)Qz=0.$
Since $M_j(T)Q$ has full column rank, it follows that $z=0$, and hence $g=0$. This proves the desired intersection condition.

We now prove $(ii)\Rightarrow(i)$. Suppose that each $B_{j,\lambda_\ell}$ has full column rank and that
\begin{equation*}
    \ker\bigl(\mathsf{V}_T(\Lambda)\otimes I_{mn}\bigr)
    \cap \operatorname{range}\!\left(
    \operatorname{blkdiag}\!\left(
    B_{j,\lambda_1},\ldots,B_{j,\lambda_R}
    \right) \right) = \{0\}.
\end{equation*}
We show that $M_j(T)Q$ has full column rank. Let $z$ satisfy $M_j(T)Qz=0$.
Using \eqref{pf:decomp2}, we obtain
\begin{equation*}
    \bigl(\mathsf{V}_T(\Lambda)\otimes I_{mn}\bigr)
    \operatorname{blkdiag}\!\left(
    B_{j,\lambda_1},\ldots,B_{j,\lambda_R}
    \right)z=0.
\end{equation*}
Set
\begin{equation*}
    g= \operatorname{blkdiag}\!\left(
    B_{j,\lambda_1},\ldots,B_{j,\lambda_R}
    \right)z.
\end{equation*}
Then
\begin{equation*}
    g\in \operatorname{range}\!\left(
    \operatorname{blkdiag}\!\left(
    B_{j,\lambda_1},\ldots,B_{j,\lambda_R}
    \right)\right)
\end{equation*}
and
\begin{equation*}
    g\in \ker\bigl(\mathsf{V}_T(\Lambda)\otimes I_{mn}\bigr).
\end{equation*}
By the assumed intersection condition, we have $g=0$. Hence
\begin{equation*}
    \operatorname{blkdiag}\!\left(
    B_{j,\lambda_1},\ldots,B_{j,\lambda_R}
    \right)z=0.
\end{equation*}
Since each $B_{j,\lambda_\ell}$ has full column rank, the block diagonal matrix also has full column rank. Therefore $z=0$. Thus $M_j(T)Q$ has full column rank. Since $Q$ is invertible, $M_j(T)$ has full column rank as well.
\end{proof}

\subsection{Proof of Corollary \ref{coro:conjecture}}

Assume that $\Xcal$ is uniquely recoverable from $\Psi_{\Xcal}$. Suppose, on the contrary, that there exists $j_0\in[p]$ such that $\Omega_{j_0}=\varnothing$. Then
\begin{equation*}
[P_\Omega]_{i,j_0,k}=0,
\qquad i\in[m],\ k\in[n].
\end{equation*}
Hence $[\widehat{P_\Omega}]_{\ell,j_0,:}=0$ for all $\ell\in [m]$, and therefore by \eqref{eq:prelim-A2},
\begin{equation*}
[A_2(j_0)]_{:,:,\,\ell}=\frac{1}{n} \operatorname{circ}\rb{[\widehat{P_\Omega}]_{\ell,j_0,:}}=0, \qquad \ell\in[m].
\end{equation*}
It follows from \eqref{eq:prelim-A3} that $A_3(j_0)=0$. Consequently, $A_3(j_0)Q=0$. By the block decomposition \eqref{eq:Bblock}, this means 
\begin{equation*}
B_{j_0,\lambda}=0, \qquad \lambda\in\Lambda,
\end{equation*}
contradicting Theorem~\ref{thm:stacked}. Therefore $\Omega_j\ne\varnothing$ for every $j\in[p]$. \qed

\subsection{Proof of Proposition \ref{prop:temp-regroup}}

Fix $j\in[p]$ and $(i,k)\in\Omega_j$. By the inverse Fourier transform along the third mode, we have
\begin{equation*}
y_t(i,k) = [\Acal^t * \Xcal]_{i,j,k} =  \frac{1}{n}\sum_{s=1}^n \wn^{(s-1)(k-1)} [\widehat{\Acal^t * \Xcal}]_{i,j,s}.
\end{equation*}
Since the t-product becomes slicewise multiplication in the Fourier domain,
\begin{equation*}
[\widehat{\Acal^t * \Xcal}]_{:,:,s} = [\widehat{\Acal}]_{:,:,s}^{\,t}[\widehat{\Xcal}]_{:,:,s}, \qquad s\in[n].
\end{equation*}
Hence $[\widehat{\Acal^t * \Xcal}]_{i,j,s} = e_i^* [\widehat{\Acal}]_{:,:,s}^{\,t} [\widehat{\Xcal}]_{:,j,s}$. By using \eqref{eq:diag} and \eqref{eq:spec1}, and by the definitions of $U_{\lambda,s}$ and $W_{\lambda,s}$,
\begin{equation*}
U_s D_s^t U_s^{-1} = \sum_{\lambda\in\Lambda_s}\lambda^t U_{\lambda,s}W_{\lambda,s}.
\end{equation*}
Therefore
\begin{align*}
[\widehat{\Acal^t * \Xcal}]_{i,j,s}
&= e_i^* \sum_{\lambda\in\Lambda_s}\lambda^t U_{\lambda,s}W_{\lambda,s} [\widehat{\Xcal}]_{:,j,s}= \sum_{\lambda\in\Lambda_s}\lambda^t e_i^*U_{\lambda,s}W_{\lambda,s}[\widehat{\Xcal}]_{:,j,s}= \sum_{\lambda\in\Lambda_s}\lambda^t \ab{ U_{\lambda,s}^* e_i,\; W_{\lambda,s}[\widehat{\Xcal}]_{:,j,s} }.
\end{align*}
Substituting this into the inverse Fourier expansion gives
\begin{align*}
y_t(i,k) &= \sum_{s=1}^n \frac{1}{n}\,\wn^{(s-1)(k-1)}
\sum_{\lambda\in\Lambda_s}\lambda^t \ab{ U_{\lambda,s}^* e_i,\; W_{\lambda,s}[\widehat{\Xcal}]_{:,j,s} }\\
&= \sum_{\lambda\in\Lambda}\lambda^t \sum_{s\in\mathcal I_\lambda} \frac{1}{n}\,\wn^{(s-1)(k-1)} \ab{ U_{\lambda,s}^* e_i,\; W_{\lambda,s}[\widehat{\Xcal}]_{:,j,s} }.
\end{align*}
This is exactly \eqref{sys1} with coefficients given by \eqref{eq:g-lambda-coord}. \qed

\subsection{Proof of Theorem \ref{thm:det}}

We first prove sufficiency. Assume that the condition of Theorem~\ref{thm:det} holds. It remains to show that the observed samples determine $\Xcal$ uniquely. Equivalently, we shall prove that if
\begin{equation*}
[\Acal^t * \Xcal]_{i,j,k}=0,
~\textnormal{for all~} (i,j,k)\in\Omega,\quad t=0,1,\ldots,T-1,
\end{equation*}
then $\Xcal=0$.

Fix $j\in[p]$ and $(i,k)\in\Omega_j$.
Let $\Lambda = \cb{\lambda_1,\ldots,\lambda_R}$ be the set of $R$ distinct eigenvalues across all slices $[\widehat{\Acal}]_{:,:,s}$, $s\in [n]$. Then by Proposition \ref{prop:temp-regroup}, for each $t= 0,1,\ldots,T-1$,
\begin{equation*}
0=\sum_{r=1}^R \lambda_r^t\, g_{\lambda_r}(i,k),
\end{equation*}
where
\begin{equation*}
g_{\lambda_r}(i,k) = \sum_{ s\in \mathcal{I}_{\lambda_r}} \frac{1}{n}\, \omega_n^{(s-1)(k-1)} \ab{ U_{\lambda_r,s}^* e_i,\; W_{\lambda_r,s} [\widehat{\Xcal}]_{:,j,s}}.
\end{equation*}
and $\mathcal{I}_\lambda$ is defined in \eqref{C1}. Equivalently,
\begin{equation*}
\mathsf{V}_T(\Lambda) \qm{g_{\lambda_1}(i,k)\\ \vdots\\ g_{\lambda_R}(i,k)}=0.
\end{equation*}
Since $T\ge |\Lambda|=R$ and the eigenvalues in $\Lambda$ are distinct, the Vandermonde matrix $\mathsf{V}_T(\Lambda)$ has full column rank. Hence
\begin{equation*}
g_{\lambda_r}(i,k)=0, \qquad r=1,\ldots,R.
\end{equation*}

We now rewrite these coefficients in the invariant form used in the theorem statement. By the definitions of $U_{\lambda,s}$ and $W_{\lambda,s}$,
\begin{equation*}
U_{\lambda,s}=U_sP_{\lambda,s}, \qquad W_{\lambda,s}=P_{\lambda,s}U_s^{-1}
\quad \text{on } \ran\rb{P_{\lambda,s}},
\end{equation*}
and therefore, for each $\lambda\in\Lambda$,
\begin{equation}\label{sum1}
    \sum_{ s\in \mathcal{I}_{\lambda}}   \frac{1}{n}\, \omega_n^{(s-1)(k-1)} \ab{  P_{\lambda,s} U_s^* e_i , U_s^{-1} [\widehat{\mathcal{X}}]_{:,j,s}}=0,\quad (i,k)\in \Omega_j,
\end{equation}
or equivalently, $\litab{V^{(i,k)}_\lambda, X^j_\lambda} =0$ for all $(i,k)\in \Omega_j$, where
\begin{equation}\label{vec1}
    V^{(i,k)}_\lambda:=\bigoplus_{s\in \mathcal{I}_\lambda} \omega_n^{-(s-1)(k-1)} P_{\lambda,s} U_{s}^* e_i,\quad X^j_\lambda :=\bigoplus_{s\in \mathcal{I}_\lambda} P_{\lambda,s} U_s^{-1}[\widehat{\mathcal{X}}]_{:,j,s}.
\end{equation}
Since, for each $j\in [p]$, the set of vectors $\big\{V^{(i,k)}_\lambda : (i,k)\in \Omega_j \big\}$ forms a frame for the concatenated eigenspace, given by $\bigoplus_{s\in \mathcal{I}_\lambda} \ran\rb{P_{\lambda,s}}$,
we have $X^j_\lambda =0$ for all $j\in [p]$ and $\lambda\in \Lambda$, i.e.,
\begin{equation*}
    P_{\lambda,s} U_s^{-1} [\widehat{\mathcal{X}}]_{:,j,s} =0,
\end{equation*}
for all $j\in [p]$, $\lambda\in \Lambda$ and $s\in[n]$, Hence, we conclude that $[\widehat{\Xcal}]_{:,j,:} =0$ for all $j\in [p]$, which completes the proof of sufficiency.

We now prove necessity.
Suppose, on the contrary, that there exist $j_0 \in [p]$ and $\lambda_*\in \Lambda$ such that the set of vectors $\big\{V^{(i,k)}_{\lambda_*} : (i,k)\in \Omega_{j_0} \big\}$ does not form a frame for $ \bigoplus_{s\in \mathcal{I}_{\lambda_*}} \ran\rb{P_{\lambda_*,s}}$.
Hence one may choose a nonzero vector
\begin{equation*}
b = \bigoplus_{s\in\mathcal I_{\lambda_*}} b_s
\in \bigoplus_{s\in\mathcal I_{\lambda_*}} \ran\rb{P_{\lambda_*,s}}
\end{equation*}
such that
\begin{equation*}
b \perp \textrm{span} \cb{ V_{\lambda_*}^{(i,k)} : (i,k)\in\Omega_{j_0} } .
\end{equation*}
Define $\Xcal \in \C^{m\times p\times n}$ by
\begin{equation*}
    [\widehat{\Xcal}]_{:,j,s} =
    \begin{cases}
        0, & \text{if } j\ne j_0,\\
        0, & \text{if } j= j_0,\ s \not\in \mathcal{I}_{\lambda_*},\\
        U_s b_s, & \text{if } j= j_0,\ s\in \mathcal{I}_{\lambda_*}.
    \end{cases}
\end{equation*}
Since $b\ne 0$, the tensor $\Xcal$ is nonzero.
On the other hand, by construction, \eqref{sum1} vanishes for all sampled indices. Thus the nonzero tensor $\Xcal$ produces the same observed spatio-temporal samples as the zero tensor, so unique recovery fails, which proves necessity. \qed

\subsection{Proof of Lemma \ref{lem:diag}}

Recall that ${\bf v}_\lambda^{(i,k)}$ is the coordinate sensing vector defined in \eqref{sensing2}, and that $\Sigma_\lambda=n \sum_{i=1}^m\sum_{k=1}^n{\bf v}_\lambda^{(i,k)}({\bf v}_\lambda^{(i,k)})^*$.
Fix $\lambda\in\Lambda$ and write $\mathcal{I}_\lambda=\{s_1,\dots,s_r\}$.
For each $(i,k)$, the reduced sensing vector has block form
\begin{equation*}
    {\bf v}_\lambda^{(i,k)}= \frac{1}{n}
\begin{bmatrix}
\omega_n^{-(s_1-1)(k-1)} U_{\lambda,s_1}^* e_i \\
\vdots \\
\omega_n^{-(s_r-1)(k-1)} U_{\lambda,s_r}^* e_i
\end{bmatrix} \in \C^{d_\lambda\times 1}.
\end{equation*}
Hence the $(s,s')$-block of $\Sigma_\lambda$ equals
\begin{align*}
        (\Sigma_\lambda)_{s,s'} &= \frac{1}{n} \sum_{i=1}^m \sum_{k=1}^n \omega_n^{-(s-s')(k-1)} U_{\lambda,s}^* e_i e_i^* U_{\lambda,s'} \\
        &= \rb{ \frac{1}{n}\sum_{k=1}^n \omega_n^{-(s-s')(k-1)}} U_{\lambda,s}^* \rb{\sum_{i=1}^m e_i e_i^*} U_{\lambda,s'}= \mathbf 1_{\{s=s'\}} U_{\lambda,s}^* U_{\lambda,s'},
\end{align*}
where we used
\begin{equation*}
    \sum_{k=1}^n \omega_n^{-(s-s')(k-1)}=n\,\mathbf 1_{\{s=s'\}}, \quad \sum_{i=1}^m e_i e_i^* = I_m,
\end{equation*}
and $\mathbf 1_{\{s=s'\}}$ equals $1$ if $s=s'$ and $0$ otherwise. This completes the proof. \qed

\subsection{Proof of Theorem \ref{thm:union}}

We use the following standard matrix Chernoff inequality, stated in the form needed below; see \cite[Theorem 1.1]{Tropp2012}.

\begin{thm}[Matrix Chernoff inequality]
Let $X_1,\ldots,X_N$ be independent, random, self-adjoint matrices in $\C^{d\times d}$ such that
\begin{equation*}
\lambda_{\max}(X_\ell)\le R, \qquad \text{almost surely}
\quad \ell=1,\ldots,N.
\end{equation*}
Assume that
\begin{equation*}
\mu_{\min}:=\lambda_{\min}\rb{\sum_{\ell=1}^N \mathbb E[X_\ell]}.
\end{equation*}
Then for every $\eta\in[0,1]$,
\begin{equation*}
\mathbb P\cb{ \lambda_{\min}\rb{\sum_{\ell=1}^N X_\ell}\le (1-\eta)\mu_{\min} } \le d \rb{\frac{e^{-\eta}}{(1-\eta)^{1-\eta}}}^{\mu_{\min}/R}.
\end{equation*}
\end{thm}

We first establish the following local concentration estimate.
\begin{theorem}\label{chernoff1}
Assume that the tensor evolution is diagonalizable as in \eqref{eq:diag} and $T \ge |\Lambda|$. Fix $j\in [p]$ and $\lambda \in \Lambda$ and let $\delta_{i,j,k} \sim \mathrm{Bernoulli}(\alpha)$ be independent random variables. Let $G_\lambda^j$ be the empirical Gram matrix defined by
\begin{equation*}
    G_\lambda^j = n \sum_{i=1}^m \sum_{k=1}^n \delta_{i,j,k} \mathbf{v}_\lambda^{(i,k)} (\mathbf{v}_\lambda^{(i,k)})^*.
\end{equation*}
Then for any $\eta \in (0,1)$,
\begin{equation}\label{tail1}
    \mathbb{P} \left\{ \lambda_{\min}(G_\lambda^j) \le (1-\eta) \alpha \gamma_{\lambda, \min} \right\} \le d_\lambda \exp \rb{ -\frac{\eta^2}{3} \frac{\alpha \gamma_{\lambda, \min}}{R_\lambda} },
\end{equation}
where $\gamma_{\lambda, \min} = \lambda_{\min}(\Sigma_\lambda) > 0$ and $R_\lambda = n \max_{(i,k)} \litnorm{\mathbf{v}_\lambda^{(i,k)}}_2^2$ is the uniform spectral bound for a single summand.

Moreover, for any $\varepsilon_\lambda\in (0,1)$, if the sampling rate $\alpha$ satisfies
\begin{equation}\label{bound1}
    \alpha \ge \frac{3 \mu_\lambda d_\lambda}{\eta^2 mn \gamma_{\lambda, \min}} \log \rb{ \frac{d_\lambda}{\varepsilon_\lambda} },
\end{equation}
then the recovery condition holds with high probability
\begin{equation*}
    \mathbb{P} \left\{ \lambda_{\min}(G_\lambda^j) \ge (1-\eta) \alpha \gamma_{\lambda, \min} \right\} \ge 1-\varepsilon_\lambda.
\end{equation*}
\end{theorem}

\begin{proof}
The empirical Gram matrix $G_\lambda^j$ is a sum of independent Hermitian positive semidefinite matrices $X_{i,k} :=n\, \delta_{i,j,k} \mathbf{v}_\lambda^{(i,k)} (\mathbf{v}_\lambda^{(i,k)})^*$. Since $\mathbb{E}[\delta_{i,j,k}] = \alpha$ for all $(i,j,k)\in [m]\times [p] \times [n]$, the expectation of the sum is given by
\begin{equation*}
    \mathbb{E}[G_\lambda^j] = \alpha n \sum_{i=1}^m \sum_{k=1}^n \mathbf{v}_\lambda^{(i,k)} (\mathbf{v}_\lambda^{(i,k)})^* = \alpha \Sigma_\lambda,
\end{equation*}
which implies that the minimum eigenvalue of the expected sum is $\mu_{\min} := \alpha \gamma_{\lambda, \min}$. Furthermore, each summand $X_{i,k}$ is a rank-one positive semidefinite matrix (and thus self-adjoint) whose maximum eigenvalue is bounded by
\begin{equation*}
    \lambda_{\max}(X_{i,k}) \le n \litnorm{\mathbf{v}_\lambda^{(i,k)}}_2^2 \le R_\lambda
\end{equation*}
for all $i\in[m]$ and $k\in [n]$. 
Applying the matrix Chernoff inequality stated above with $d = d_\lambda$, $R = R_\lambda$, and $\mu_{\min} = \alpha \gamma_{\lambda,\min}$, we obtain
\begin{equation*}
    \mathbb{P} \cb{ \lambda_{\min}(G_\lambda^j) \le (1-\eta) \alpha \gamma_{\lambda,\min} } \le d_\lambda \rb{ \frac{e^{-\eta}}{(1-\eta)^{1-\eta}} }^{ \alpha \gamma_{\lambda,\min} / R_\lambda}.
\end{equation*}
Using the scalar inequality $e^{-\eta} / (1-\eta)^{1-\eta} \le \exp(-\eta^2/3)$ for $\eta \in (0,1)$ and substituting the expression for $\mu_{\min}$, we obtain the desired tail bound in \eqref{tail1}.

Finally, to ensure that the recovery condition holds with probability at least $1-\varepsilon_\lambda$, we set the right-hand side of \eqref{tail1} to be at most $\varepsilon_\lambda$. Solving for $\alpha$ and utilizing the relation $R_\lambda \le \frac{\mu_\lambda d_\lambda}{mn}$ from the definition of incoherence in \eqref{incoherence1}, we arrive at the sufficient condition for the sampling rate provided in \eqref{bound1}. 
\end{proof}

We now prove Theorem~\ref{thm:union}.
  \begin{proof}[The proof of Theorem~\ref{thm:union}]
For each lateral slice $j \in [p]$ and each distinct eigenvalue $\lambda \in \Lambda$, let
\begin{equation*}
    \mathcal{E}_{j,\lambda} := \cb{ \lambda_{\min}(G_\lambda^j) < (1-\eta)\alpha \gamma_{\lambda,\min} }.
\end{equation*}
Applying Theorem~\ref{chernoff1} with $\varepsilon/(p|\Lambda|)$ in place of $\varepsilon$, we obtain
that if \eqref{eq:alpha_uniform} holds, then
\begin{equation*}
    \mathbb{P}(\mathcal{E}_{j,\lambda}) \le d_\lambda \exp \left( -\frac{\eta^2}{3} \frac{\alpha \gamma_{\lambda, \min}}{R_\lambda} \right) \le \frac{\varepsilon}{p |\Lambda|},
\end{equation*}
for every pair $(j, \lambda)$. 

Applying the union bound over all $j\in [p]$ and $\lambda \in \Lambda$, we obtain
\begin{equation*}
    \mathbb{P} \left( \bigcup_{j=1}^p \bigcup_{\lambda \in \Lambda} \mathcal{E}_{j,\lambda} \right) \le \sum_{j=1}^p \sum_{\lambda \in \Lambda} \mathbb{P}(\mathcal{E}_{j,\lambda}) \le p |\Lambda| \cdot \frac{\varepsilon}{p |\Lambda|} = \varepsilon.
\end{equation*}
Hence, with probability at least $1-\varepsilon$,
\begin{equation*}
    \lambda_{\min}(G_\lambda^j)\ge (1-\eta)\alpha\gamma_{\lambda,\min}>0
\end{equation*}
for all $j\in[p]$ and $\lambda\in\Lambda$.
In particular, every $G_\lambda^j$ is invertible.
Therefore each component ${\bf x}_\lambda^j$ is uniquely determined, and hence so is the tensor signal $\Xcal$.   \end{proof}

\subsection{Proof of Proposition \ref{prop:temp-stability}}

Recall that the noisy temporal data are written as
\begin{equation*}
\widetilde{\bf y}^{(i,k)}={\bf y}^{(i,k)}+{\boldsymbol\eta}^{(i,k)},
\qquad
{\bf y}^{(i,k)}=\mathsf V_T(\Lambda){\bf g}^{(i,k)}
\end{equation*}
for each $(i,k)\in\Omega_j$, where ${\bf g}^{(i,k)}$ is defined in \eqref{yg-system}.

Since the eigenvalues in $\Lambda$ are pairwise distinct and $T\ge |\Lambda|$, the Vandermonde matrix $\mathsf{V}_T(\Lambda)$ has full column rank. Hence, for each $(i,k)\in\Omega_j$, the least-squares estimate
\begin{equation*}
\widetilde{\bf g}^{(i,k)}=\mathsf{V}_T(\Lambda)^\dagger \widetilde{\bf y}^{(i,k)}
\end{equation*}
satisfies
\begin{equation*}
\widetilde{\bf g}^{(i,k)}-{\bf g}^{(i,k)} = \mathsf{V}_T(\Lambda)^\dagger {\boldsymbol\eta}^{(i,k)}.
\end{equation*}
Therefore
\begin{equation}\label{pf:temp1}
\norm{\widetilde{\bf g}^{(i,k)}-{\bf g}^{(i,k)}}_2 \le
\norm{\mathsf{V}_T(\Lambda)^\dagger}_2\,\norm{{\boldsymbol\eta}^{(i,k)}}_2 = \frac{1}{\sigma_{\min}\rb{\mathsf{V}_T(\Lambda)}}\norm{{\boldsymbol\eta}^{(i,k)}}_2,
\qquad (i,k)\in\Omega_j.
\end{equation}
Here we used the identity $\norm{\mathsf{V}_T(\Lambda)^\dagger}_2 = 1/\sigma_{\min}\rb{\mathsf{V}_T(\Lambda)}$.

Now fix $\lambda\in\Lambda$. The reconstruction of the $\lambda$-block is obtained from the scalar data $\cb{\widetilde g_\lambda(i,k):(i,k)\in\Omega_j}$.
Since the family $\cb{{\bf v}_\lambda^{(i,k)}:(i,k)\in\Omega_j}$ has lower frame bound $C_\lambda^j$, the associated analysis operator has a bounded left inverse with norm at most $(C_\lambda^j)^{-1/2}$.
Consequently,
\begin{equation}\label{pf:temp2}
\norm{\widetilde{\bf x}_\lambda^j-{\bf x}_\lambda^j}_2 \le
(C_\lambda^j)^{-1/2} \rb{ \sum_{(i,k)\in\Omega_j}
\vb{\widetilde g_\lambda(i,k)-g_\lambda(i,k)}^2 }^{1/2}.
\end{equation}

Finally, for each $(i,k)\in\Omega_j$, the $\lambda$-block of the vector error $\widetilde{\bf g}^{(i,k)}-{\bf g}^{(i,k)}$
is exactly $\widetilde g_\lambda(i,k)-g_\lambda(i,k)$, and therefore
\begin{equation*}
\vb{\widetilde g_\lambda(i,k)-g_\lambda(i,k)} \le \norm{\widetilde{\bf g}^{(i,k)}-{\bf g}^{(i,k)}}_2.
\end{equation*}
Substituting \eqref{pf:temp1} into \eqref{pf:temp2}, we obtain
\begin{align*}
\norm{\widetilde{\bf x}_\lambda^j-{\bf x}_\lambda^j}_2
&\le (C_\lambda^j)^{-1/2} \rb{ \sum_{(i,k)\in\Omega_j}
\norm{\widetilde{\bf g}^{(i,k)}-{\bf g}^{(i,k)}}_2^2 }^{1/2}\le (C_\lambda^j)^{-1/2}
\frac{1}{\sigma_{\min}\rb{\mathsf{V}_T(\Lambda)}} \rb{ \sum_{(i,k)\in\Omega_j} \norm{{\boldsymbol\eta}^{(i,k)}}_2^2 }^{1/2},
\end{align*}
which is exactly the desired estimate. \qed

\subsection{Proof of Theorem \ref{thm:temp-conditioning}}

Recall that $\Lambda=\cb{\lambda_1,\ldots,\lambda_R}$, $R=|\Lambda|$, and that $\rho_-=\min_{\lambda\in\Lambda}\vb{\lambda}$, $\rho_+=\max_{\lambda\in\Lambda}\vb{\lambda}$, $\phi_T$ is defined in \eqref{eq:phiT-def}, and $\Phi_T=\phi_T(\rho_-)/\phi_T(\rho_+)$.

For each $r=1,\ldots,R$, where $R:=|\Lambda|$, let
\begin{equation*}
c_r:=\qm{1 & \lambda_r & \cdots & \lambda_r^{T-1}}^\top
\end{equation*}
denote the $r$th column of $\mathsf V_T(\Lambda)$. Then
\begin{equation*}
\norm{c_r}_2=\rb{\sum_{t=0}^{T-1}\vb{\lambda_r}^{2t}}^{1/2}=\phi_T\rb{\vb{\lambda_r}}.
\end{equation*}
Hence
\begin{equation*}
\min_{1\le r\le R}\norm{c_r}_2=\phi_T(\rho_-), \qquad \max_{1\le r\le R}\norm{c_r}_2=\phi_T(\rho_+).
\end{equation*}

We first prove {\rm(i)}. Since
\begin{equation*}
\sigma_{\min}\rb{\mathsf V_T(\Lambda)}=\min_{\norm{x}_2=1}\norm{\mathsf V_T(\Lambda)x}_2\le \norm{\mathsf V_T(\Lambda)e_r}_2=\norm{c_r}_2
\end{equation*}
for every $r$, we obtain
\begin{equation*}
\sigma_{\min}\rb{\mathsf V_T(\Lambda)}\le \phi_T(\rho_-).
\end{equation*}
Similarly,
\begin{equation*}
\sigma_{\max}\rb{\mathsf V_T(\Lambda)}=\max_{\norm{x}_2=1}\norm{\mathsf V_T(\Lambda)x}_2\ge \norm{\mathsf V_T(\Lambda)e_r}_2=\norm{c_r}_2
\end{equation*}
for every $r$, so
\begin{equation*}
\sigma_{\max}\rb{\mathsf V_T(\Lambda)}\ge \phi_T(\rho_+).
\end{equation*}
Therefore
\begin{equation*}
\frac{1}{\kappa_2\rb{\mathsf V_T(\Lambda)}}=\frac{\sigma_{\min}\rb{\mathsf V_T(\Lambda)}}{\sigma_{\max}\rb{\mathsf V_T(\Lambda)}}\le \frac{\phi_T(\rho_-)}{\phi_T(\rho_+)}=\Phi_T.
\end{equation*}

We next prove {\rm(ii)}. For $T\ge R$, set
\begin{equation*}
S_T(\rho):=\sum_{t=0}^{T-1}\rho^{2t}, \qquad h_T(\rho):=\frac{\rho^{2T}}{S_T(\rho)}.
\end{equation*}
Then
\begin{equation}\label{pf:ratio}
    \frac{\Phi_{T+1}^2}{\Phi_T^2}=\frac{S_{T+1}(\rho_-)}{S_T(\rho_-)}\frac{S_T(\rho_+)}{S_{T+1}(\rho_+)}=\frac{1+h_T(\rho_-)}{1+h_T(\rho_+)}.
\end{equation}
Thus it is enough to show that $h_T(\rho)$ is strictly increasing in $\rho>0$.

Let $0<\rho_1<\rho_2$. Then
\begin{align*}
\rho_1^{2T}S_T(\rho_2)&=\sum_{t=0}^{T-1}\rho_1^{2T}\rho_2^{2t}<\sum_{t=0}^{T-1}\rho_2^{2T}\rho_1^{2t}=\rho_2^{2T}S_T(\rho_1),
\end{align*}
because for each $t=0,\ldots,T-1$, $\rho_1^{2T}\rho_2^{2t}<\rho_2^{2T}\rho_1^{2t}$
and $T-t\ge 1$. Hence
\begin{equation*}
h_T(\rho_1)<h_T(\rho_2).
\end{equation*}
Since $\rho_-<\rho_+$, \eqref{pf:ratio} gives
\begin{equation*}
\frac{\Phi_{T+1}^2}{\Phi_T^2}<1,
\end{equation*}
and therefore $\Phi_{T+1}<\Phi_T$. This proves {\rm(ii)}.

We now prove {\rm(iii)}. First consider the case $\rho_+=1$. Since $\rho_+>\rho_-$, we have $\rho_-<1$, and therefore
\begin{equation*}
\phi_T(\rho_-)\le \rb{\sum_{t=0}^{\infty}\rho_-^{2t}}^{1/2}=\frac{1}{\sqrt{1-\rho_-^2}}, \qquad \phi_T(\rho_+)=\phi_T(1)=\sqrt{T}.
\end{equation*}
Hence
\begin{equation*}
0\le \Phi_T\le \frac{1}{\sqrt{1-\rho_-^2}}\frac{1}{\sqrt{T}}\to 0 \qquad\text{as } T\to\infty.
\end{equation*}

Next assume $\rho_+>1$. Since $\phi_T(\rho_+)\ge \rho_+^{T-1}$, while
\begin{equation*}
\phi_T(\rho_-)\le
\begin{cases}
\dfrac{1}{\sqrt{1-\rho_-^2}}, & 0<\rho_-<1,\\[0.8em]
\sqrt{T}, & \rho_-=1,\\[0.4em]
\sqrt{T}\,\rho_-^{T-1}, & \rho_->1,
\end{cases}
\end{equation*}
we obtain
\begin{equation*}
0\le \Phi_T\le
\begin{cases}
\dfrac{1}{\sqrt{1-\rho_-^2}}\,\rho_+^{-(T-1)}, & 0<\rho_-<1,\\[0.8em]
\sqrt{T}\,\rho_+^{-(T-1)}, & \rho_-=1,\\[0.4em]
\sqrt{T}\,\rb{\dfrac{\rho_-}{\rho_+}}^{T-1}, & \rho_->1.
\end{cases}
\end{equation*}
In every case the right-hand side tends to $0$ as $T\to\infty$, because $\rho_+>1$ and $\rho_-/\rho_+<1$. 
\qed

\subsection{Proof of Corollary \ref{coro:explicit}}
We use the notation $\rho_-$, $\rho_+$, $\phi_T$, and $\Phi_T$ from Theorem~\ref{thm:temp-conditioning}.

Since $0\le \rho_-<1$ and $\rho_+>1$, we have
\begin{align*}
    \phi_T(\rho_-)^2=\sum_{t=0}^{T-1}\rho_-^{2t}\le \frac{1}{1-\rho_-^2},\qquad
    \phi_T(\rho_+)^2=\sum_{t=0}^{T-1}\rho_+^{2t}\ge \rho_+^{2(T-1)},
\end{align*}
which implies that
\begin{equation*}
\Phi_T=\frac{\phi_T(\rho_-)}{\phi_T(\rho_+)}
\le \frac{1}{\sqrt{1-\rho_-^2}}\,\rho_+^{-(T-1)}.
\end{equation*}
If \eqref{eq:Ttau} holds, then $\Phi_T\le \tau$, and the conclusion follows from Theorem~\ref{thm:temp-conditioning} {\rm(i)}.

It remains to justify \eqref{eq:Ttau-rough}. If
\begin{equation*}
\frac{1}{\tau\sqrt{1-\rho_-^2}}\le 1,
\end{equation*}
then \eqref{eq:Ttau} holds automatically for every $T\ge R$. Otherwise, using
\begin{equation*}
\log \rho_+ \ge \frac{\rho_+-1}{\rho_+},
\end{equation*}
the bound \eqref{eq:Ttau-rough} implies
\begin{equation*}
(T-1)\log \rho_+ \ge \log\rb{\frac{1}{\tau\sqrt{1-\rho_-^2}}},
\end{equation*}
which is equivalent to \eqref{eq:Ttau}. \qed

\medskip\noindent
\section*{Acknowledgments}
The NOAA Optimum Interpolation (OI) SST V2 data used in this study were provided by the NOAA Physical Sciences Laboratory, Boulder, Colorado, USA, from the \href{https://psl.noaa.gov}{NOAA PSL website}. The forced isotropic turbulence data were obtained from the Johns Hopkins Turbulence Database (JHTDB), available at \href{https://turbulence.idies.jhu.edu/datasets/homogeneousTurbulence/isotropic}{the JHTDB forced isotropic turbulence dataset page}. The research is in part supported by NSF Grant DMS-2603463.

\printbibliography

\appendix

\section{Jordan structure across frontal slices in Fourier domain}\label{app:jordan}

In this appendix, we extend the deterministic identifiability theory beyond the diagonalizable setting. In the presence of Jordan blocks, temporal samples are no longer governed solely by powers $\lambda^t$ of distinct eigenvalues, but also involve additional polynomial factors induced by the nilpotent structure. Consequently, the relevant recovery spaces are no longer individual eigenspaces but aggregated generalized eigenspaces. Accordingly, the frame condition in Theorem~\ref{thm:det} must be replaced by a higher-order condition that captures the contribution of the nilpotent components in the Jordan decomposition.

 For each frontal slice $[\widehat{\Acal}]_{:,:,s}$, we consider its Jordan canonical form. That is, for each $s\in[n]$, there exists an invertible matrix $U_s$ such that
\begin{equation}\label{eq:jordan}
[\widehat{\Acal}]_{:,:,s}=U_sJ_sU_s^{-1},
\end{equation}
where $J_s$ is in Jordan canonical form.

Let $\Lambda$ denote the set of distinct eigenvalues appearing across all
slices. For each $s\in[n]$ and $\lambda\in\Lambda_s$, let $E_{\lambda,s}$ denote the coordinate projection onto the Jordan blocks of $J_s$
corresponding to $\lambda$, and define the generalized eigenspace
\begin{equation*}
\mathcal{J}_{\lambda,s}: =\ker\big((J_s-\lambda I)^{\nu_{\lambda,s}}\big)=\ran(E_{\lambda,s}),
\end{equation*}
where $\nu_{\lambda,s}$ is the size of the largest Jordan block of $J_s$ corresponding to $\lambda$ in slice $s$.

We also define the nilpotent operator
\begin{equation*}
N_{\lambda,s}:=(J_s-\lambda I)E_{\lambda,s},
\end{equation*}
which satisfies $N_{\lambda,s}^{\nu_{\lambda,s}}=0$.
For each $\lambda\in\Lambda$, let
\begin{equation*}
\nu_\lambda:=\max_{s\in\mathcal{I}_\lambda}\nu_{\lambda,s}.
\end{equation*}

Thus, compared with the diagonalizable case, one must recover not only the
spectral amplitudes associated with each $\lambda$, but also the additional
nilpotent directions generated by the Jordan structure.
The following theorem characterizes the identifiability of the  signal
in the presence of Jordan blocks.

\begin{theorem}\label{thm:jordan}
Assume
\begin{equation*}
T\ge \sum_{\lambda\in\Lambda}\nu_\lambda.
\end{equation*}
The signal $\Xcal$ is uniquely recoverable from $\Psi_\Xcal$ if and only if, for each $j\in[p]$ and each $\lambda\in\Lambda$, the collection
\begin{equation*}
\Big\{V_{\lambda,r}^{(i,k)}:(i,k)\in\Omega_j,\ 0\le r\le \nu_\lambda-1\Big\}
\end{equation*}
forms a frame for $\bigoplus_{s\in\mathcal{I}_\lambda}\mathcal{J}_{\lambda,s},$
where
\begin{equation*}
V_{\lambda,r}^{(i,k)} := \bigoplus_{s\in\mathcal{I}_\lambda} \omega_n^{-(s-1)(k-1)} E_{\lambda,s}(N_{\lambda,s}^*)^rU_s^*e_i.
\end{equation*}

\end{theorem}

The diagonalizable and Jordan results above provide deterministic characterizations of injectivity in terms of frame conditions on the corresponding aggregated spectral spaces.  In practice, however, one would like to know when these conditions hold under random spatial subsampling. 

This appendix completes the deterministic extension to non-diagonalizable slice dynamics. 
The randomized theory in the main text is stated in the diagonalizable setting, where the frame geometry admits a cleaner Gram-matrix formulation.

\begin{proof}
We use the Jordan notation introduced above. Thus $[\widehat{\Acal}]_{:,:,s}=U_sJ_sU_s^{-1}$, $E_{\lambda,s}$ is the coordinate projection onto the Jordan blocks for $\lambda$, $N_{\lambda,s}=(J_s-\lambda I)E_{\lambda,s}$, and $\nu_\lambda=\max_{s\in\mathcal I_\lambda}\nu_{\lambda,s}$.

We prove sufficiency by showing that the measurement operator has a trivial null space.
Suppose that $\Xcal$ produces zero measurements, i.e.
\begin{equation*}
[\Acal^t*\Xcal]_{i,j,k}=0,
\qquad (i,j,k)\in\Omega,\ t=0,1,\dots,T-1.
\end{equation*}

Proceeding as in the proof of Theorem~\ref{thm:det}, we obtain
\begin{equation*}
0= \sum_{s=1}^n\frac{1}{n}\omega_n^{(s-1)(k-1)}
\langle e_i,[\widehat{\Acal}]_{:,:,s}^t[\widehat{\Xcal}]_{:,j,s}\rangle.
\end{equation*}
Using $[\widehat{\Acal}]_{:,:,s}^t=U_sJ_s^tU_s^{-1}$ and setting
\begin{equation*}
z_{j,s}:=U_s^{-1}[\widehat{\Xcal}]_{:,j,s},
\end{equation*}
this becomes
\begin{equation*}
0= \sum_{s=1}^n\frac{1}{ n}\omega_n^{(s-1)(k-1)}
\langle U_s^*e_i,J_s^t z_{j,s}\rangle.
\end{equation*}

Fix $s$. Since $J_s$ is in Jordan form, we have the block decomposition
\begin{equation*}
J_s=\sum_{\lambda\in\Lambda_s}(\lambda I+N_{\lambda,s})E_{\lambda,s}.
\end{equation*}
Using the binomial expansion on each $\lambda$-block,
\begin{equation*}
(\lambda I+N_{\lambda,s})^t = \sum_{r=0}^{\nu_{\lambda,s}-1}\binom{t}{r}\lambda^{t-r}N_{\lambda,s}^r,
\end{equation*}
where $\binom{t}{r}$ with $r>t$ are understood to be zero.
Since $N_{\lambda,s} E_{\lambda,s} = E_{\lambda,s} N_{\lambda,s}$ and $E_{\lambda,s} E_{\lambda',s}=0$ for $\lambda \ne \lambda'$, it follows
\begin{align*}
    J_s^t &=  \sum_{\lambda\in\Lambda_s}(\lambda I+N_{\lambda,s})^t E_{\lambda,s} = \sum_{\lambda\in\Lambda_s}\sum_{r=0}^{\nu_{\lambda,s}-1} \binom{t}{r}\lambda^{t-r}N_{\lambda,s}^rE_{\lambda,s}
\end{align*}

Substituting into the measurement equation and regrouping by distinct eigenvalues yields
\begin{equation*}
0=\sum_{\lambda\in\Lambda}\sum_{r=0}^{\nu_\lambda-1}\binom{t}{r}\lambda^{t-r}\,
g_{\lambda,r}(i,k),
\qquad t=0,1,\dots,T-1,
\end{equation*}
where
\begin{equation*}
g_{\lambda,r}(i,k) :=
\sum_{s\in\mathcal{I}_\lambda}\frac{1}{n}\omega_n^{(s-1)(k-1)}
\langle U_s^*e_i,\,N_{\lambda,s}^rE_{\lambda,s}z_{j,s}\rangle.
\end{equation*}
Equivalently, by moving $N_{\lambda,s}^r$ to the left,
\begin{equation*}
g_{\lambda,r}(i,k)
=
\sum_{s\in\mathcal{I}_\lambda}\frac{1}{ n}\omega_n^{(s-1)(k-1)}
\langle (N_{\lambda,s}^*)^rU_s^*e_i,\,E_{\lambda,s}z_{j,s}\rangle.
\end{equation*}

Since $T\ge\sum_{\lambda\in\Lambda}\nu_\lambda$, the family of functions
$\{\binom{t}{r}\lambda^{t-r}\}_{\lambda\in\Lambda,\,0\le r\le\nu_\lambda-1}$
is linearly independent on $\{0,1,\dots,T-1\}$ (equivalently, the associated
confluent Vandermonde matrix is invertible). Therefore
\begin{equation*}
g_{\lambda,r}(i,k)=0
\qquad \text{for all }\lambda\in\Lambda,\ 0\le r\le\nu_\lambda-1,\ (i,k)\in\Omega_j.
\end{equation*}

Define, for each $\lambda\in\Lambda$,
\begin{equation*}
X_\lambda^j
:=
\bigoplus_{s\in\mathcal{I}_\lambda}E_{\lambda,s}z_{j,s}
\in
\bigoplus_{s\in\mathcal{I}_\lambda}\mathcal{J}_{\lambda,s}.
\end{equation*}
By definition of $V_{\lambda,r}^{(i,k)}$ in the theorem, we have
\begin{equation*}
\langle V_{\lambda,r}^{(i,k)},X_\lambda^j\rangle =
\sum_{s\in\mathcal{I}_\lambda}\omega_n^{(s-1)(k-1)}
\langle E_{\lambda,s}(N_{\lambda,s}^*)^rU_s^*e_i,\,E_{\lambda,s}z_{j,s}\rangle
= n \, g_{\lambda,r}(i,k),
\end{equation*}
where we used $E_{\lambda,s}=E_{\lambda,s}^*$ and $E_{\lambda,s}^2=E_{\lambda,s}$.
Hence
\begin{equation*}
\langle V_{\lambda,r}^{(i,k)},X_\lambda^j\rangle=0
\qquad (i,k)\in\Omega_j,\ 0\le r\le\nu_\lambda-1.
\end{equation*}

Denote this frame family by
\begin{equation*}
\mathcal V_{\lambda,j}:=\cb{V_{\lambda,r}^{(i,k)}:(i,k)\in\Omega_j,\ 0\le r\le\nu_\lambda-1}.
\end{equation*}
By the frame assumption in the theorem, the collection $\mathcal V_{\lambda,j}$ forms a frame for
$\bigoplus_{s\in\mathcal{I}_\lambda}\mathcal{J}_{\lambda,s}$, and therefore the only vector
orthogonal to all frame elements is the zero vector. Thus
\begin{equation*}
X_\lambda^j=0
\qquad \text{for all }\lambda\in\Lambda.
\end{equation*}
In particular, $E_{\lambda,s}z_{j,s}=0$ for all $\lambda\in\Lambda_s$. Since
$\sum_{\lambda\in\Lambda_s}E_{\lambda,s}=I$, we conclude that $z_{j,s}=0$, hence
$[\widehat{\Xcal}]_{:,j,s}=U_sz_{j,s}=0$ for all $j$ and $s$. Therefore $\Xcal=0$.

We now prove necessity.
Suppose that the condition of Theorem~\ref{thm:jordan} fails.
Then there exist $j_0\in[p]$ and $\lambda_*\in\Lambda$ such that
\begin{equation*}
\mathcal V_{\lambda_*,j_0} = \cb{V_{\lambda_*,r}^{(i,k)} : (i,k)\in\Omega_{j_0},\ 0\le r\le \nu_{\lambda_*}-1}
\end{equation*}
does not form a frame for
\begin{equation*}
\bigoplus_{s\in\mathcal I_{\lambda_*}}\mathcal J_{\lambda_*,s}.
\end{equation*}
Hence one may choose a nonzero vector
\begin{equation*}
b=\bigoplus_{s\in\mathcal I_{\lambda_*}} b_s
\in \bigoplus_{s\in\mathcal I_{\lambda_*}}\mathcal J_{\lambda_*,s}
\end{equation*}
such that $b \perp \textrm{span} \rb{\mathcal{V}_{\lambda_*,j_0}}$.

Define $\Xcal\in\C^{m\times p\times n}$ by
\begin{equation*}
[\widehat{\Xcal}]_{:,j,s} =
\begin{cases}
0, & j\ne j_0,\\
0, & j=j_0,\ s\notin\mathcal I_{\lambda_*},\\
U_s b_s, & j=j_0,\ s\in\mathcal I_{\lambda_*}.
\end{cases}
\end{equation*}
Then $\Xcal\ne 0$, and
\begin{equation*}
X_{\lambda_*}^{j_0} = \bigoplus_{s\in\mathcal I_{\lambda_*}}E_{\lambda_*,s}b_s = b, \qquad
X_\lambda^{j_0}=0 \ \text{for } \lambda\ne\lambda_*.
\end{equation*}
Therefore,
\begin{equation*}
\langle V_{\lambda_*,r}^{(i,k)},X_{\lambda_*}^{j_0}\rangle = \langle V_{\lambda_*,r}^{(i,k)},b \rangle =0
\end{equation*}
for every $(i,k)\in\Omega_{j_0}$ and every $0\le r\le \nu_{\lambda_*}-1$.
Equivalently, all coefficients $g_{\lambda,r}(i,k)$ vanish on the sampled set.
Hence
\begin{equation*}
[\Acal^t * \Xcal]_{i,j,k}=0,
\qquad (i,j,k)\in\Omega,\quad t=0,1,\ldots,T-1.
\end{equation*}
Thus the nonzero tensor $\Xcal$ produces the same observed spatio-temporal
samples as the zero tensor, and unique recovery fails.
This proves necessity.
\end{proof}

\section{Conditioning across slices and basis scaling}\label{app:scaling}

In this appendix we clarify how the population conditioning parameter $\gamma_{\lambda,\min}$ depends on the choice of slice-wise eigenbases.
This issue is relevant for the randomized sampling bounds, since the parameter $\gamma_{\lambda,\min}$ enters directly into the lower bound for the sampling rate $\alpha$.

Fix $\lambda\in\Lambda$ and write
\begin{equation*}
\mathcal I_\lambda=\cb{s_1,\ldots,s_r}.
\end{equation*}
By Lemma~\ref{lem:diag},
\begin{equation*}
\Sigma_\lambda=\mathrm{diag}\rb{U_{\lambda,s_1}^*U_{\lambda,s_1},\ldots,U_{\lambda,s_r}^*U_{\lambda,s_r}},
\end{equation*}
and hence
\begin{equation*}
\gamma_{\lambda,\min}=\lambda_{\min}\rb{\Sigma_\lambda}
=\min_{s\in\mathcal I_\lambda}\lambda_{\min}\rb{U_{\lambda,s}^*U_{\lambda,s}}
=\min_{s\in\mathcal I_\lambda}\sigma_{\min}^2\rb{U_{\lambda,s}}.
\end{equation*}
In particular, $\gamma_{\lambda,\min}>0$ if and only if each $U_{\lambda,s}$ has full column rank.

This identity shows that $\gamma_{\lambda,\min}$ is not purely a subspace quantity.
Indeed, for each $s\in\mathcal I_\lambda$, replacing $U_{\lambda,s}$ by
\begin{equation*}
\widetilde U_{\lambda,s}=U_{\lambda,s}D_s,
\end{equation*}
where $D_s$ is any invertible diagonal matrix of size $m_{\lambda,s}\times m_{\lambda,s}$, does not change the eigenspace $\ran\rb{U_{\lambda,s}}$.

In particular, without fixing a normalization, one can make $\gamma_{\lambda,\min}$ arbitrarily small by a purely artificial diagonal rescaling.
For example, if
\begin{equation*}
D_s(\varepsilon):=\diag\rb{\varepsilon,1,\ldots,1},
\qquad 0<\varepsilon\ll 1,
\end{equation*}
then
\begin{equation*}
\sigma_{\min}\rb{U_{\lambda,s}D_s(\varepsilon)}
\le \norm{U_{\lambda,s}D_s(\varepsilon)e_1}_2
=\varepsilon \norm{U_{\lambda,s}e_1}_2 \to 0
\qquad \text{as } \varepsilon\to 0.
\end{equation*}
Hence the corresponding value of $\gamma_{\lambda,\min}$, and therefore the lower bound for the sampling rate $\alpha$, may become arbitrarily pessimistic under such a choice of scaling.

A simple way to remove this trivial scaling freedom is to normalize the columns of $U_{\lambda,s}$ to unit Euclidean norm.
Writing $u_{s,1},\ldots,u_{s,m_{\lambda,s}}$ for the columns of $U_{\lambda,s}$, define
\begin{equation*}
\widehat D_s:=\diag\rb{\norm{u_{s,1}}_2^{-1},\ldots,\norm{u_{s,m_{\lambda,s}}}_2^{-1}},
\qquad
\widehat U_{\lambda,s}:=U_{\lambda,s}\widehat D_s.
\end{equation*}
Then the columns of $\widehat U_{\lambda,s}$ all have unit Euclidean norm.
Moreover, any other diagonal scaling with unit-norm columns has the form
$\widehat U_{\lambda,s}J_s$, where $J_s$ is diagonal and each diagonal entry has modulus one.
Hence the singular values of the normalized matrix are independent of the remaining phase or sign freedom.

We use the following $2$-norm special case of a diagonal-scaling estimate of van der Sluis
\cite[Theorem~3.5(b)]{vandersluis1969}, stated in the form needed here.

\begin{theorem}[{\cite[Theorem~3.5(b)]{vandersluis1969}}]
Let $U\in\C^{M\times r}$ have full column rank, and let $\widehat D$ be an invertible diagonal matrix such that the columns of $U\widehat D$ have equal Euclidean norms.
Then
\begin{equation*}
\kappa_2\rb{U\widehat D} \le \sqrt r \inf\cb{\kappa_2\rb{UD}:D\in\C^{r\times r}\text{ is invertible diagonal}}.
\end{equation*}
\end{theorem}

Applying the theorem with $U=U_{\lambda,s}$ and $r=m_{\lambda,s}$ gives
\begin{equation*}
\kappa_2\rb{\widehat U_{\lambda,s}} \le \sqrt{m_{\lambda,s}}
\inf\cb{\kappa_2\rb{U_{\lambda,s}D_s}:D_s\in\C^{m_{\lambda,s}\times m_{\lambda,s}}\text{ is invertible diagonal}}.
\end{equation*}
Thus normalizing the eigenvectors to have unit Euclidean norm is not merely a cosmetic convention.
It gives a diagonally rescaled eigenbasis whose $2$-norm condition number is within the standard factor $\sqrt{m_{\lambda,s}}$ of the best possible diagonal rescaling.

Consequently, in the randomized estimates one should interpret $\gamma_{\lambda,\min}$ after fixing such a normalization convention for the slice-wise eigenbases.
Under this convention, a small value of $\gamma_{\lambda,\min}$ reflects nonorthogonality or near-dependence in the chosen eigenspace representation, rather than arbitrary rescaling of individual eigenvectors.
In the normal or unitary diagonalizable case, one may choose the eigenvectors orthonormally, and then
\begin{equation*}
U_{\lambda,s}^*U_{\lambda,s} = I_{m_{\lambda,s}},
\qquad
\gamma_{\lambda,\min} = 1,
\end{equation*}
so this basis-scaling issue disappears.

\end{document}